\documentclass[11pt]{amsart}

\usepackage{graphicx}
\usepackage[USenglish,english]{babel}
\usepackage[T1]{fontenc}
\usepackage[latin1]{inputenc}
\usepackage{amsfonts}
\usepackage{amssymb}
\usepackage{amsthm}
\usepackage{amsmath}
\usepackage{tikz}
\usetikzlibrary{decorations.pathreplacing}
\usepackage{float}
\usepackage{enumerate}
\usepackage{accents}
\usepackage{mathtools}
\usepackage{mathrsfs}
\usepackage{comment}
\usepackage{afterpage}
\usepackage{bbm}
\usepackage{dsfont}
\usepackage{stmaryrd}
\usepackage{hyperref}
\usepackage{mleftright}
\usepackage{subfig}
\usepackage{faktor}
\usepackage{url}

\theoremstyle{plain}
\newtheorem{thm}{Theorem}[section]
\newtheorem{lem}[thm]{Lemma}
\newtheorem{prop}[thm]{Proposition}
\newtheorem{cor}[thm]{Corollary}
\newtheorem*{thm*}{Main Theorem}
\newtheorem*{prop*}{Proposition}
\newtheorem*{cor*}{Corollary}
\newtheorem{thmintro}{Theorem}

\newtheorem{propintro}[thmintro]{Proposition}

\theoremstyle{definition}
\newtheorem{mydef}[thm]{Definition}
\newtheorem{exam}[thm]{Example}
\newtheorem{rem}[thm]{Remark}
\newtheorem*{quest*}{Question}

\newtheorem{defprop}[thm]{Definition/Proposition}

\DeclareMathOperator{\Ima}{Im}
\DeclareMathOperator{\diam}{diam}

\newcommand{\CAT}{{\rm CAT(0)}}

\newcommand{\codim}{\mathrm{codim}}
\newcommand{\RomanNumeralCaps}[1]
    {\MakeUppercase{\romannumeral #1}}
\newcommand{\grad}{\mathrm{grad}}

\title{Circumcenter extension maps for non-positively curved spaces}

\author{Merlin Incerti-Medici}
\address{Institut f\"ur Mathematik, Universit\"at Z\"urich, Switzerland}
\email{merlin.incerti-medici@math.uzh.ch}

\begin{document}

\maketitle

\begin{abstract}
We show that every cross ratio preserving homeomorphism between boundaries of Hadamard manifolds extends to a continuous map, called circumcenter extension, provided that the manifolds satisfy certain visibility conditions. We show that this map is a rough isometry, whenever the manifolds admit cocompact group actions by isometries and we improve the quasi-isometry constants provided by Biswas in the case of $\mathrm{CAT(-1)}$ spaces. Finally, we provide a sufficient condition for this map to be an isometry in the case of Hadamard surfaces.
\end{abstract}

\tableofcontents

%---------------------------------------------------------------------------------------------------------------------------------------

%INTRODUCTION

%---------------------------------------------------------------------------------------------------------------------------------------

\section{Introduction} \label{sec:Introduction}

The visual boundary of a geodesically complete $\mathrm{CAT(-1)}$ space is naturally endowed with a cross ratio. In this paper, we show that this cross ratio can still be defined on the visual boundary of a geodesically complete $\CAT$ space and that it still retains a lot of information about the interior space, provided that certain visibility conditions are satisfied.

Our core motivation is the following, informal question, which is related to several results from geometry and geometric group theory in recent decades.

\begin{quest*}
Let $(X,d)$ and $(Y,d)$ be two geodesically complete $\CAT$ spaces and $f : \partial X \rightarrow \partial Y$ a cross ratio-preserving homeomorphism between their boundaries. Can $f$ be extended to an isometry $F : X \rightarrow Y$?

Alternatively, if a group $G$ acts on $X$ and $Y$ by isometries and $f$ is a $G$-equivariant, cross ratio-preserving homeomorphism, can we construct $F$ to be $G$-equivariant?
\end{quest*}

This question has seen a series of complete and partial answers for various special cases over the course of the last few decades. The first series of results concern situations where $X$ and $Y$ admit geometric actions, i.e.\,proper, cocompact actions by isometries, by some group $G$. Specifically, if $X$ and $Y$ have constant negative curvature, the extension of $f$ to a $G$-equivariant isometry is used in Thurston's proof of Mostow rigidity \cite{MostowThurston}. If $X$ and $Y$ are universal coverings of negatively curved surfaces, the fact that $f$ extends to a $G$-equivariant isometry is crucial to Otal's proof of Marked Length Spectrum Rigidity \cite{Otal}. In a series of papers, Hamenst\"adt and Besson-Courtois-Gallot show that $f$ extends to a $G$-equivariant isometry, if both spaces are universal coverings of negatively curved manifolds and one of them is a locally symmetric space \cite{BessonCourtoisGallot95, Hamenstadt99}.

A second, more recent series of results mostly drops the assumptions about group actions. If $X$ and $Y$ are proper, geodesically complete $\mathrm{CAT(-1)}$ spaces, Biswas proved that $f$ can be extended to a rough isometry, i.e.\,a $(1,C)$-quasi-isometry \cite{Biswas15}. One may also consider classes of spaces for which the appropriate notions are no longer the visual boundary and isometries. Beyrer, Fioravanti and the author proved a similar extension theorem for $\CAT$ cube complexes and their Roller boundaries, equipped with a suitable cross ratio \cite{BeyFioMedici}. Furthermore, Beyrer-Fioravanti proved additional extension theorems for cubulable hyperbolic groups and for certain group actions on nice $\CAT$ cube complexes \cite{BeyFio,BeyFio2}.

A third collection of results can be found when thinking of the question above in more dynamical terms. There are various instances, where the visual boundary is closely related to the geodesic flow of the interior space. For example, if $X$ is the universal covering of a closed, negatively curved Riemannian manifold, then the cross ratio on its boundary and the Marked Length Spectrum of the closed manifold determine each other (see for example \cite{Hamenstadt92, Biswas15} for how the cross ratio determines the Marked Length Spectrum and \cite{Otal92, Kim01, Kim04} for the converse). De Simoi-Kaloshin-Wei have proven a Length Spectrum Rigidity result for billiards that feature sufficient symmetries and whose boundaries are sufficiently close to the circle \cite{SimoiKaloshinWei}. We also mention the work of Kaloshin-Sorrentino, who show that, if a strictly convex billiard has the same maximal marked length spectrum as an ellipse, then it is an ellipse \cite{KaloshinSorrentino}. It seems to be an open question, whether these results can be formulated in a geometric language similar to the one of this paper.

Finally, one may take a coarse viewpoint and only require that $f$ coarsely preserves the cross ratio. This is called a quasi-M\"obius map. In \cite{CCM}, Charney-Cordes-Murray show that, under a mild stability condition, quasi-M\"obius maps between Morse boundaries of finitely generated groups extend to quasi-isometries of groups.

While cross ratios have been used on numerous occasions in spaces of negative curvature, they have not been studied very much for general non-positively curved spaces. In part, this is due to the fact that visual boundaries of hyperbolic or $\mathrm{CAT(-1)}$ spaces have several properties that visual boundaries of $\CAT$ spaces don't. In this paper, we show how to work around these difficulties to define a cross ratio on the visual boundary of a proper, geodesically complete $\CAT$ space. We then show that the circumcenter construction introduced in \cite{Biswas15, Biswas17b} can be generalized to a large class of non-positively curved manifolds and that this circumcenter extension provides a good framework to study the initially stated question.\\

In order to state our results, we first need to define the cross ratio. Let $X$ be a proper, connected, geodesically complete $\CAT$ space. Fix a base point $o \in X$. For any admissible quadruple $(\xi_1, \xi_2, \xi_3, \xi_4) \in \partial X^4$ (see section \ref{subsec:crossratios} for the definition of admissibility), we can define the cross ratio
\[ cr(\xi_1, \xi_2, \xi_3, \xi_4) := \frac{ \rho_o(\xi_1, \xi_2) \rho_o(\xi_3, \xi_4) }{ \rho_o(\xi_1, \xi_3) \rho_o(\xi_2, \xi_4) }, \]
where $\rho_o(\xi, \eta) := e^{-(\xi \vert \eta)_o}$, with $(\xi \vert \eta)_o$ denoting the Gromov product on $\partial X$ with respect to the base point $o$.

\begin{propintro} \label{propintro:PropositionA}
Let $X$ be a proper, connected, geodesically complete $\CAT$ space. The cross ratio $cr$ is well-defined for all admissible quadruples and independent of the choice of $o$.
\end{propintro}

It turns out that the boundary, together with the cross ratio, contains a lot of information about the interior space, provided that the boundary satisfies certain visibility properties. Specifically, we say:
\begin{enumerate}
\item The visual boundary $\partial X$ satisfies {\it 4-visibility}, if for every quadruple $(\xi_1, \xi_2, \xi_3, \xi_4) \in \partial X^4$, there exists $\eta \in \partial X$, such that for all $i \in \{1,2,3,4\}$, $(\xi_i \vert \eta)_o$ is finite for some base point $o \in X$.

\item We say that $\xi \in \partial X$ is {\it in a rank 1 hinge} if there exist $\eta, \zeta \in \partial X$, such that there is a bi-infinite geodesic from $\eta$ to $\zeta$ and there exist bi-infinite rank 1 geodesics from $\xi$ to $\eta$ and from $\xi$ to $\zeta$.
\end{enumerate}

Given two proper, connected, geodesically complete $\CAT$ spaces $X$ and $Y$, we say that a map $f : \partial X \rightarrow \partial Y$ is {\it M\"obius} if and only if it preserves the cross ratio.

Let $X, Y$ be Hadamard manifolds, i.e.\,simply connected, geodesically complete Riemannian manifolds such that all sectional curvatures are non-positive. The main result of this paper is a construction that allows us to extend M\"obius homeomorphisms that satisfy one mild extra condition to the interior spaces, provided that $\partial X$ and $\partial Y$ satisfy conditions (1) and (2) above. We call this extension the {\it circumcenter extension} of $f$. We prove a very general Theorem about the circumcenter extension in section \ref{sec:HolderLipschitz}. This Theorem has several consequences, as soon as one adds some extra assumption. If $X$ and $Y$ admit a cocompact group action, we obtain

\begin{thmintro} \label{thmintro:CocompactAction}
Let $X, Y$ be Hadamard manifolds such that $\partial X, \partial Y$ satisfy (1) and all points in $\partial X$ and $\partial Y$ satisfy (2). Suppose, the group $G$ acts cocompactly by isometries on $X$ and $Y$. Let $f : \partial X \rightarrow \partial Y$ be a $G$-equivariant M\"obius homeomorphism, such that $f$ and $f^{-1}$ send visible pairs to visible pairs. Then the circumcenter extension is a $G$-equivariant $(1, 2M)$-quasi-isometry $F : X \rightarrow Y$ for some constant $M \geq 0$.
\end{thmintro}

The constant $M$ will be the supremum of a Lipschitz continuous function $M : X \rightarrow [0, \infty)$, which essentially measures by how much $F$ fails to be an isometry at a certain point. A better understanding of the function $M$ has the potential to significantly improve this result, the main result of the paper and the results stated below.\\

If we drop the assumption about cocompact group actions, we can restrict to more specialised situations and obtain other, sometimes even stronger results. The first of these results concerns surfaces.

\begin{thmintro} \label{thmintro:Surfaces}
Let $X, Y$ be Hadamard manifolds whose sectional curvatures are bounded from below by $-b^2$ such that $\partial X, \partial Y$ satisfy (1) and all points in $\partial X$ and $\partial Y$ satisfy (2). Assume that $X, Y$ are $2$-dimensional. Let $f : \partial X \rightarrow \partial Y$ be a M\"obius homeomorphism such that $f$ and $f^{-1}$ send visible pairs to visible pairs. Then, the circumcenter extension is a homeomorphism $F: X \rightarrow Y$, it is locally Lipschitz continuous on a dense subset and differentiable almost everywhere.

Furthermore, to every $x \in X$, we can associate a set $K_x \subset \partial X$ with the following property: For almost every $x$, $F$ is differentiable at $x$ and if $K_x$ contains at least five points, then $DF_x$ is an isometry between tangent spaces.
\end{thmintro}

We emphasize that the Theorem above includes the statement that the circumcenter extension is invertible, a result that we do not obtain in higher dimensions.\\

In \cite{Biswas15}, Biswas proved that the circumcenter extension provides a $(1, \ln(2))$-quasi-isometry, if $X$ and $Y$ are $\mathrm{CAT(-1)}$ spaces. For manifolds with a lower curvature bound $-b^2$, it is a $(1, (1 - \frac{1}{b}) \ln(2))$-quasi-isometry (cf. \cite{Biswas17a}). We can recover Biswas' first constant and improve it in dimension two.

\begin{thmintro} \label{thmintro:CAT(-1)}
Let $X, Y$ be Hadamard manifolds whose sectional curvatures are bounded from below by $-b^2$ and from above by $-1$. Let $f : \partial X \rightarrow \partial Y$ be a M\"obius homeomorphism. Then, the circumcenter extension is a $(1, \ln(2))$-quasi-isometry.

Furthermore, if $X$ and $Y$ are $2$-dimensional, then the circumcenter extension is a $(1, \ln( \frac{4}{3} ))$-quasi-isometry.
\end{thmintro}

We end the introduction with a discussion on what kind of spaces satisfy the visibility properties necessary for our results (see section \ref{sec:Preliminaries} for all definitions not given here). Many examples can be constructed by using the notion of visibility points in the boundary. A point $\xi \in \partial X$ is called a {\it visibility point} if it can be connected with every other point in the boundary by a bi-infinite geodesic in $X$. Let $\xi \in \partial X$ be a visibility point, $\eta \in \partial X$ and $\gamma$ a bi-infinite geodesic from $\xi$ to $\eta$. It follows that $\gamma$ is a rank one geodesic, as the end points of geodesic lines that are not rank one (i.e.\,that bound a euclidean halfplane) cannot be visibility points. Since two points in the boundary that can be connected by a bi-infinite geodesic always have finite Gromov product, we conclude that, whenever $\partial X$ contains at least five visibility points, both visibility properties introduced above are satisfied.

We now present a class of Hadamard manifolds that do not have strictly negative curvature and do admit five visibility points. Consider a closed, non-positively curved Riemannian manifold $M$. By the rank rigidity theorem (see Theorem C in \cite{Ballmann95}), we obtain that its universal covering $\tilde{M}$ is either a finite Riemannian product, a higher rank symmetric space, or contains at least one bi-infinite rank one geodesic. By Lemma 1.7 in \cite{BallmannBuyalo}, the endpoints of this rank one geodesic in $\partial X$ have Tits distance strictly greater than $\pi$. Because $\pi_1(M)$ acts properly and cocompactly, its limit-set, denoted $\Lambda$, satisfies $\Lambda = \partial X$ (see the introduction of \cite{BallmannBuyalo}). By Proposition 1.10 in \cite{BallmannBuyalo}, this implies that there exists an element $g \in \pi_1(M)$ and a rank one geodesic $\gamma$ in $\tilde{M}$, such that $g$ acts as translation on $\gamma$. By Theorem 5.4 in \cite{BestvinaFujiwara}, an axis for some isometry in a proper $\CAT$ space is rank one if and only if it is contracting. Therefore, $\gamma$ is a contracting geodesic line. In particular, both of its endpoints in $\partial X$ are visibility points by Proposition 3.6 in \cite{CharneySultan}. Since every orbit of $\pi_1(M)$ in $\partial X$ is dense (see for example \cite{Hamenstadt09}, in particular Lemma 5.1) and isometries send visibility points to visibility points, we conclude that $\partial \tilde{M}$ has infinitely many visibility points. Therefore, the universal covering of any closed, non-positively curved Riemannian manifold $M$ is either a finite Riemannian product, a higher rank symmetric space, or it satisfies the visibility properties (1) and (2). This provides us with a large class of spaces satisfying our assumptions. In particular, this includes most graph manifolds.

We also mention a non-cocompact example that can be obtained as follows. Consider five copies of the euclidean upper halfplane $\mathbb{R} \times [0, \infty)$ and glue them together isometrically along their boundary $\mathbb{R} \times \{ 0 \}$ such that all five halfplanes intersect at the origin. The space obtained this way is a $\CAT$ space and its Tits boundary is a circle of circumference $5\pi$. Using properties of the Tits metric and rank one geodesics, one can see that this space satisfies visibility properties (1) and (2) as well. While this example is only a $\CAT$ and not a Riemannian manifold it seems feasible that a Hadamard manifold with the same behaviour can be constructed. Both of the examples above illustrate that there is a large and flexible class of Hadamard manifolds that satisfy our visibility properties, but do not admit a negative upper curvature bound, which shows that the circumcenter extension map indeed can be constructed in a more general setting than previously thought.\\

The remainder of the paper is organised as follows. In section \ref{sec:Preliminaries}, we develop all the necessary preliminary theory. Specifically, we give a brief introduction to asymptotic geometry and generalise several results known for $\mathrm{CAT(-1)}$ spaces to $\CAT$ spaces, including the proof of Proposition \ref{propintro:PropositionA}; we generalise the theory of metric derivatives as needed and we give a brief primer on the facts we will need about convex functions and Jacobi fields. We also provide an example of a space whose boundary contains points that have finite Gromov product but are not visible (this example is the reason why $f$ and $f^{-1}$ have to send visible pairs to visible pairs). In section \ref{sec:Definition}, we construct the circumcenter extension and define all the notions we will use to prove the results above. In section \ref{sec:HolderLipschitz}, we prove a result on H\"older continuity of the circumcenter extension and finish the proof of the main result (see Theorem \ref{thm:Extensionresult}). In section \ref{sec:Applications}, we prove Theorems \ref{thmintro:CocompactAction}, \ref{thmintro:Surfaces} and \ref{thmintro:CAT(-1)}. Sections \ref{sec:HolderLipschitz} and \ref{sec:Applications} are written so that they can be read independently.\\

{\bf Acknowledgments} The author thanks Viktor Schroeder for countless discussions, many helpful questions, ideas and suggestions. The author thanks Jean-Claude Picaud for several discussions, ideas, for proposing the original project and for his comments on an early draft of the first part of this paper. The author is grateful to Giuliano Basso, Ruth Charney, Matthew Cordes, Peter Feller, Yannick Krifka, Urs Lang, Marc Lischka, Alessandro Sisto, Davide Spriano and Luca Studer for their patient listening, their inputs, suggestions and counter examples concerning all things geometric and to Elia Fioravanti for the same and additionally for teaching the author to always consult \cite{BH}. The author thanks Manuela Gehrig and Simone Steinbruechel for their patience, ideas and for pointing out several dead ends concerning all things analytic. Finally, the author is grateful to Kingshook Biswas for his comments, suggestions and for his work in \cite{Biswas15, Biswas17a, Biswas17b, Biswas18a}, which was crucial to the work presented here. Part of this work was carried out at the Simons Semester in spring 2019 supported by the grant 346300 for IMPAN from the Simons Foundation and the matching 2015-2019 Polish MNiSW fund. The author has been supported by the Swiss National Science Foundation Grant 175567.

%---------------------------------------------------------------------------------------------------------------------------------------

%PRELIMINARIES

%---------------------------------------------------------------------------------------------------------------------------------------

\section{Preliminaries} \label{sec:Preliminaries}

\subsection{Boundaries at infinity, Gromov products and Busemann functions} \label{subsec:Boundaries}

For a general introduction to spaces of non-positive curvature, we refer to \cite{BGS} and \cite{BH}. For more material on asymptotic geometry, we additionally refer to \cite{BS}.

Let $(X,d)$ be a proper, connected $\CAT$ space. (Later, we will specialize to an $n$-dimensional, connected, geodesically complete Riemannian manifold $(X,g)$, such that all sectional curvatures are non-positive.) Since $(X,d)$ is $\CAT$, the functions $d(\gamma(t), \gamma'(t))$ and $d(x, \gamma(t))$ are convex and strictly convex respectively for all geodesics $\gamma, \gamma'$ and $x \in X$.

A {\it geodesic ray} is an isometric embedding $\gamma : [a, \infty) \rightarrow X$. A {\it bi-infinite geodesic} is an isometric embedding defined on $\mathbb{R}$. To make notation easier, all our geodesic rays will start at time $a = 0$. A metric space is called {\it geodesically complete} if and only if all geodesic segments can be extended to geodesic lines.

Assume from now on, that $(X,d)$ is a proper, connected, geodesically complete $\CAT$ space. Two geodesic rays $\gamma$, $\gamma'$ are called {\it asymptotic} if there exists $B > 0$ such that for all $t \geq 0$, $d(\gamma(t), \gamma'(t)) \leq B$. The {\it boundary at infinity} of $X$ is defined as the space of equivalence classes of geodesic rays
\[ \partial X := \faktor{\{ \gamma | \gamma \text{ a geodesic ray } \}}{asymptoticity}. \]

Given a representative $\gamma$ of $\xi$, we sometimes write $\xi = \gamma(\infty)$ and call $\xi$ the {\it endpoint of $\gamma$}. A geodesic ray $\gamma : (-\infty, 0] \rightarrow X$ can also be interpreted as a representative of a point $\xi$ at infinity and we write $\xi = \gamma(-\infty)$. In particular, a bi-infinite geodesic $\gamma$ defines two points $\gamma(\infty), \gamma(-\infty)$ in $\partial X$. Since $X$ is $\CAT$, any two bi-infinite geodesics whose endpoints are $\xi, \eta$ are parallel. We denote the set of geodesics with $\gamma(-\infty) = \xi, \gamma(\infty) = \eta$ by $[\xi, \eta]$. Whenever there exists a bi-infinite geodesic with endpoints $\xi, \eta$ we call it a geodesic from $\xi$ to $\eta$ and we say that $(\xi, \eta)$ is {\it visible}. We call an $n$-tuple $(\xi_1, \dots, \xi_n)$ a {\it visible $n$-tuple}, whenever for all $i \neq j, (\xi_i, \xi_j)$ is visible.

We can equip the boundary at infinity with a topology called the visual topology. It is defined as follows. Fix a base point $x \in X$. For any $\xi \in \partial X$, denote the unique geodesic ray starting at $x$, representing $\xi$ by $\xi_x$. (The existence of such geodesics is a well-known application of the Theorem of Arzela-Ascoli.) Let $\xi \in \partial X$. For all $R > 0$, $\epsilon > 0$, define
\[ÊU_{R, \epsilon, x}(\xi) := \{ \eta \in \partial X \vert d(\eta_x(R), \xi_x(R)) < \epsilon \}. \]

It is easy to see that the collection $\{ U_{R, \epsilon, x}(\xi) \}_{R, \epsilon, \xi}$ forms a basis for a topology on $\partial X$, the visual topology. Furthermore, this topology is independent of $x$ (see Part \RomanNumeralCaps{2}, Section 8 in \cite{BH}).

The following family of functions is a valuable tool when studying $\partial X$. Fix a base point $x \in X$ and define the Gromov product of two points $\xi$, $\eta \in \partial X$ with respect to $x$ to be
\begin{equation*}
\begin{split}
(\xi | \eta)_x & := \lim_{t \rightarrow \infty} \frac{1}{2} \left( d(x, \xi_x(t)) + d(x, \eta_x(t)) - d(\xi_x(t), \eta_x(t)) \right)\\
& = \lim_{t \rightarrow \infty} t - \frac{1}{2} d(\xi_x(t), \eta_x(t)).
\end{split}
\end{equation*}

This limit exists, since the function on the right-hand-side is non-decreasing in $t$, although $(\xi | \eta)_x$ may be infinite, e.g. if $\xi = \eta$, or if $\xi_x, \eta_x$ span a flat sector in $X$ (think of a sector in $\mathbb{R}^2$). If there exists $\epsilon > 0$, such that $X$ is $\rm{CAT(-\epsilon)}$, then $(\xi | \eta)_x$ is infinite if and only if $\xi = \eta$. Note that, if $(\xi \vert \eta)_x = \infty$ for some base point $x$, then $(\xi \vert \eta)_y = \infty$ for all $y \in X$. We define
\[ \rho_x(\xi, \eta) := e^{-(\xi | \eta)_x}. \]

\begin{rem}
We remark that we could define the Gromov product to be the limit $\lim_{n \rightarrow \infty} (\xi_x(t_n) \vert \eta_x(t'_n))_x$ for any two sequences $t_n, t'_n \xrightarrow{n \rightarrow \infty} \infty$. Due to monotony, all these limits are equal.
\end{rem}

We will also use the following notation. In analogy to the Gromov product on the boundary, we define for any triple $x, y, z \in X$ the Gromov product of $y, z$ with respect to $x$ by
\[ (y \vert z)_x := \frac{1}{2} (d(x,y) + d(x,z) - d(y,z)). \]

In order to understand how the Gromov product depends on the choice of base point, we use the Busemann function: Let $x$, $y \in X$ and $\xi \in \partial X$. The Busemann function is defined by
\[ B(x,y, \xi) := \lim_{t \rightarrow \infty} d(y, \xi_x(t)) - d(x, \xi_x(t)). \]

Using the triangle-inequality, we see that the function on the right-hand-side is bounded in absolute value and non-increasing, hence this limit exists and is finite. Further, the Busemann function is continuous in $x$ and $y$.

In Appendix A.2 of \cite{DPS}, it is shown that in $\CAT$ spaces, for all geodesic rays $\gamma$ asymptotic to $\xi_x$,
\[ B(x,y,\xi) = \lim_{t \rightarrow \infty} d(y, \gamma(t)) - d(x, \gamma(t)). \]
This independence of the representative of $\xi$ implies that for all $x, y, z \in X$ and $\xi \in \partial X$, we have the cocycle equation
\begin{equation*}
\begin{split}
B(x,y,\xi) + B(y,z,\xi) & = \lim_{t \rightarrow \infty} d(y, \gamma(t)) - d(x, \gamma(t)) + d(z, \gamma(t)) - d(y, \gamma(t))\\
& = B(x,z,\xi).
\end{split}
\end{equation*}
In particular,
\[ B(x,y,\xi) = - B(y,x,\xi). \]
Finally, the Busemann function is also continuous in $\xi$ with respect to the visual topology on $\partial X$

Let $\gamma, \gamma'$ be two asymptotic geodesic rays representing $\xi \in \partial X$. Extend both of them to bi-infinite geodesic lines. For all $T \in \mathbb{R}$, there exists $T'$ such that $\gamma(T), \gamma'(T')$ lie on the same horosphere of $\xi$, i.e. $B(\gamma'(T'), \gamma(T), \xi) = 0$. In fact,
\begin{equation*}
\begin{split}
T' & = B(\gamma'(T'), \gamma'(0), \xi)\\
& = B(\gamma'(T'), \gamma(T), \xi) + B(\gamma(T), \gamma(0), \xi) + B(\gamma(0), \gamma'(0), \xi)\\
& = T + B(\gamma(0), \gamma'(0), \xi).
\end{split}
\end{equation*}

\begin{figure} 
\begin{tikzpicture}[scale=1.50]
\draw [] (0,0) circle [radius=2cm];
\draw [] (1.25,0) circle [radius=0.75cm];
\draw [] (0.985,0) circle [radius=1.015cm];
\node [right] at (2,0) {$\xi$};
\draw [] (0,2) arc [radius=2, start angle =180, end angle=270 ];
\draw [fill] (0.385,0.82) circle [radius=0.015cm];
\node [above right] at (0.24, 0.87) {$\gamma(0)$};
\draw [] (2,0) arc [radius=11.35, start angle= 90, end angle = 110];
\draw [fill] (-1, -0.405) circle [radius=0.015cm];
\node [below] at (-1, -0.405) {$\gamma'(0)$};
\draw [fill] (-0.015,-0.18) circle [radius=0.015cm];
\node [above left] at (-0.015,-0.18) {$\gamma_{new}'(0)$};
\draw [fill] (0.685,0.49) circle [radius=0.015cm];
\node [right] at (0.685,0.51) {$\gamma(T)$};
\draw [fill] (0.505,-0.103) circle [radius=0.015cm];
\node [below right] at (0.505,-0.08) {$\gamma'(T')$};
\end{tikzpicture}
\caption{Any two geodesic rays representing the same point $\xi$ at infinity can be reparametrised so that at any given time, they lie on the same horosphere centered at $\xi$. In this figure, one may reparametrise $\gamma'$ as $\gamma_{new}'(t) = \gamma'(t +T' - T)$.}
\label{fig:Horospheres} 
\end{figure}
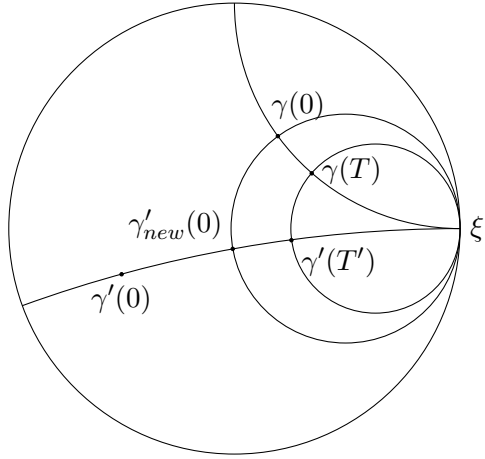

\begin{mydef} \label{def:comparison triangle}
Let $x, y \in X$ and $\xi \in \partial X$. For $k > 0$, there exists an - up to isometry unique - constellation $\overline{x}, \overline{y} \in \mathbb{H}_{-k}^2, \overline{\xi} \in \partial \mathbb{H}_{-k}^2$, such that $d(\overline{x}, \overline{y}) = d(x, y)$ and $B(\overline{x}, \overline{y}, \overline{\xi}) = B(x, y, \xi)$. We call $(\overline{x}, \overline{y}, \overline{\xi})$ a {\it comparison triangle} of $(x,y, \xi)$ in $\mathbb{H}_{-k}^2$. We call the angle $\angle_{\overline{x}}(\overline{y}, \overline{\xi})$ the {\it comparison angle} to $(x,y,\xi)$ in curvature $-k$ and denote it by $\angle^{(-k)}_{x}(y, \xi)$.

For $k = 0$, we analogously find $\overline{x}, \overline{y} \in \mathbb{R}^2, \overline{\xi} \in \partial \mathbb{R}^2$ satisfying the same equations. We use the same terminology and denote the comparison angle to $(x,y,\xi)$ in curvature $0$ by $\angle_x^{(0)}(y,\xi)$.
\end{mydef}

Note that $\angle^{(-k)}_x(y,\xi)$ is continuous in $x, y$ and $\xi$ for all $k \geq 0$. The Busemann function $B(x,y,\xi)$ can be described completely in terms of the distance $d(x,y)$ and its comparison triangle with respect to one model space of constant curvature. Specifically,
\[ \forall k > 0 : e^{k B(x,y,\xi)} = \cosh(k d(x,y)) - \sinh(k d(x,y))\cos( \angle^{(-k^2)}_{x}(y,\xi) )Ê\]
\[ B(x,y,\xi) = -d(x,y) \cos( \angle^{(0)}_{x}(y,\xi) ). \]

The first of these formulas is proven in \cite{Biswas17a}. We prove the second one here, since it is not easy to find in the literature.

\begin{proof}
Let $x, y \in X, \xi \in \partial X$ and let $\overline{\gamma}$ be a geodesic ray in $\mathbb{R}^2$. Denote $\overline{x} := \overline{\gamma}(0)$ and $\overline{\xi} := [\overline{\gamma}]$. For every $t > 0$, there exist exactly two points $\overline{y_1(t)}, \overline{y_2(t)} \in \mathbb{R}^2$, such that $(\overline{x}, \overline{y_i}, \overline{\gamma(t)})$ is a comparison triangle to $(x,y,\xi_x(t))$ (by definition a triangle with the same side lengths). For every $t$, choose one of these two points, denoted $\overline{y}(t)$, such that $\overline{y}(t)$ varies continuously in $t$. Since $(\overline{x}, \overline{y}(t), \overline{\gamma}(t))$ are comparison triangles to $(x,y,\xi_x(t))$, we have
\[ d(\overline{y}(t), \overline{\gamma}(t)) - d(\overline{x},\overline{\gamma}(t)) = d(y,\xi_x(t)) - d(x,\xi_x(t)) \xrightarrow{t \rightarrow \infty} B(x,y,\xi). \]

Since $\overline{y}(t)$ is a bounded curve, it admits a converging subsequence. The equation above implies that any convergent subsequence of $\overline{y}(t)$ converges to a point $\overline{y}$ such that $(\overline{x},Ê\overline{y}, \overline{\xi})$ is a comparison triangle for $(x,y,\xi)$. Since there are exactly two such points and $\overline{y}(t)$ is continuous, we see that $\overline{y}(t)$ converges to one of these points. This implies that
\[ \angle_{\overline{x}}(\overline{y}(t),\overline{\gamma}(t)) \xrightarrow{t \rightarrow \infty} \angle_{\overline{x}}(\overline{y}, \overline{\xi}) = \angle^{(0)}_x(y,\xi). \]

By the law of cosines in Euclidean space, we have
\begin{equation*}
\begin{split}
d(\overline{y}(t), \overline{\gamma}(t))^2 & = d(\overline{x}, \overline{y}(t))^2 + d(\overline{x}, \overline{\gamma}(t))^2 - 2d(\overline{x}, \overline{y}(t)) d(\overline{x}, \overline{\gamma}(t)) \cos(\angle_{\overline{x}}(\overline{y}(t), \overline{\gamma}(t)))\\
& = d(x,y)^2 + t^2 - 2d(x,y)t \cos(\angle_{\overline{x}}(\overline{y}(t), \overline{\gamma}(t))).
\end{split}
\end{equation*}

Therefore,
\begin{equation*}
\begin{split}
B(x,y,\xi) & = \lim_{t \rightarrow \infty} ( d(y, \xi_x(t)) - d(x, \xi_x(t)) ) \frac{d(y, \xi_x(t)) + d(x, \xi_x(t))}{2t}\\
& = \lim_{t \rightarrow \infty} \frac{ d(x,y)^2 - 2d(x,y)t \cos(\angle_{\overline{x}}(\overline{y}(t), \overline{\gamma}(t))) }{2t}\\
& = -d(x,y)\cos(\angle^{(0)}_x(y,\xi)).
\end{split}
\end{equation*}

\end{proof}

%---------------------------------------------------------------------------------------------------------------------------------------
%CROSS RATIOS
%---------------------------------------------------------------------------------------------------------------------------------------

\subsection{Cross ratios} \label{subsec:crossratios}

Let $X$ be a proper, connected, geodesically complete $\CAT$ space. We obtain a family of functions $( \cdot \vert \cdot)_x : \partial X \times \partial X \rightarrow [0, \infty]$.

\begin{mydef} \label{def:algebraicvisibility}
Let $(\xi_1, \dots, \xi_n) \in \partial X^n$ be an $n$-tuple. Choose $x \in X$. We say $(\xi_1, \dots, \xi_n)$ is {\it algebraically visible}, if for all $i \neq j$, $(\xi_i \vert \xi_j)_x < \infty$.

\end{mydef}

As noted in the last section, this definition does not depend on the choice of $x$. We can reformulate it in terms of the maps $\rho_x$, by requiring $\rho_x(\xi_i, \xi_j) > 0$ for all $i \neq j$ instead. Note that $\rho_x$ is symmetric and non-negative, but it does not satisfy the triangle inequality and there may be pairs $\xi \neq \eta$ such that $\rho_x(\xi, \eta) = 0$. Nevertheless, we can use $\rho_x$ to define a cross ratio as follows.

Define the set of {\it admissible quadruples} in $\partial X$ to be the set
\begin{equation*}
\begin{split}
\mathcal{A} := \{ (\xi_1, \xi_2, \xi_3, \xi_4) \in \partial X^4 \vert & \forall i \neq j \neq k \neq i, \text{ at least two of the pairs }\\
& (\xi_i, \xi_j), (\xi_i, \xi_k), (\xi_j, \xi_k) \text{ are algebraically visible} \}.
\end{split}
\end{equation*}

In other words, $\mathcal{A}$ consists of the quadruples, whose points do not include a chain in which pairs of consecutive points are not algebraically visible. For all admissible quadruples, we can define a cross ratio by
\[ cr_x(\xi_1, \xi_2, \xi_3, \xi_4) := \frac{\rho_x(\xi_1, \xi_2) \rho_x(\xi_3, \xi_4)}{\rho_x(\xi_1, \xi_3) \rho_x(\xi_2, \xi_4)} \in [0, \infty]. \]

The goal of this section is to prove the following theorem.

\begin{thm} \label{thm:CrossRatio}

Let $X$ be a proper, connected, geodesically complete $\CAT$ space. Then, for all $x$, $y \in X$, $cr_x = cr_y$ on all of $\mathcal{A}$.

\end{thm}

The proof is based on the following

\begin{lem} \label{lem:derivative}
For all $x, y \in X$, and for all $\xi, \eta \in \partial X$,
\[ (\xi | \eta)_x = \frac{1}{2} \left( B(y,x,\xi) + B(y,x,Ê\eta) \right) + (\xi | \eta)_y \]
\end{lem}

This formula is well known for $\mathrm{CAT(-1)}$ spaces (see \cite{Bourdon95}). However, the case of $\CAT$ spaces is hard to find in the literature, which is why we provide a proof here.

\begin{proof}[Proof of Lemma \ref{lem:derivative}]
Since $(\xi \vert \eta)_x = \infty$ if and only if $(\xi \vert \eta)_y = \infty$, the equation trivially holds in that case. Suppose $(\xi \vert \eta)_x < \infty$. We first show the inequality `$\geq$'. Let $\epsilon > 0$. Since the function $(\xi_y(t) \vert \eta_y(t))_y$ is monotone increasing, we find $T \geq 0$, such that for all $t \geq T$, we have
\[ (\xi_y(t) \vert \eta_y(t))_y \geq (\xi \vert \eta)_y - \frac{\epsilon}{2}. \]

From the properties of Busemann functions in the last section, we know that there are $T_{\xi}, T_{\eta}$ such that $B(\xi_x(T_{\xi}), \xi_y(T), \xi) = 0 = B(\eta_x(T_{\eta}), \eta_y(T), \eta)$. Specifically,
\[ T_{\xi} = T + B(y,x,\xi) \]
\[ T_{\eta} = T + B(y,x,\eta). \]

Since $B(\xi_x(T_{\xi}), \xi_y(T), \xi) = B(\eta_x(T_{\eta}), \eta_y(T), \eta) = 0$, there exists $S \geq T$ such that for all $s \geq S$,
\[ \vert d(\xi_y(T), \xi_x(s)) - d(\xi_x(T_{\xi}), \xi_x(s)) \vert \leq \frac{\epsilon}{2} \]
\[ \vert d(\eta_y(T), \eta_x(s)) - d(\eta_x(T_{\eta}), \eta_x(s)) \vert \leq \frac{\epsilon}{2}. \]
We obtain for all $s \geq S$
\begin{equation*}
\begin{split}
2(\xi \vert \eta)_x & \geq d(x, \xi_x(s)) + d(x, \eta_x(s)) - d(\xi_x(s), \eta_x(s))\\
& \geq s + s - d(\xi_x(s), \xi_y(T)) - d(\xi_y(T), \eta_y(T)) - d(\eta_y(T), \eta_x(s))\\
& \geq s + s - d(\xi_x(s), \xi_x(T_{\xi})) - d(\eta_x(s), \eta_x(T_{\eta})) - d(\xi_y(T), \eta_y(T)) - \epsilon\\
& = T_{\xi} + T_{\eta} - d(\xi_y(T), \eta_y(T)) - \epsilon\\
& = B(y,x,\xi) + B(y,x,\eta) + 2(\xi_y(T) \vert \eta_y(T))_y - \epsilon\\
& \geq B(y,x,\xi) + B(y,x,\eta) + 2(\xi \vert \eta)_y - 2\epsilon.
\end{split}
\end{equation*}
Since $\epsilon$ was chosen arbitrarily, we obtain
\[ (\xi \vert \eta)_x \geq (\xi \vert \eta)_y + \frac{1}{2}(B(y,x,\xi) + B(y,x,\eta))Ê\]
The same argument with $x$ and $y$ swapped yields
\begin{equation*}
\begin{split}
(\xi \vert \eta)_y & \geq (\xi \vert \eta)_x + \frac{1}{2}(B(x,y,\xi) + B(x,y,\eta))\\
& = (\xi \vert \eta)_x - \frac{1}{2}(B(y,x,\xi) + B(y,x,\eta))
\end{split}
\end{equation*}
and thus
\[ (\xi \vert \eta)_x \leq (\xi \vert \eta)_y + \frac{1}{2}(B(y,x,\xi) + B(y,x,\eta)), \]
which concludes the proof.
\end{proof}

\begin{proof}[Proof of Theorem \ref{thm:CrossRatio}]

We know from Lemma \ref{lem:derivative} that for all $\xi, \eta \in \partial X$
\[ \rho_y(\xi, \eta) = \sqrt{ e^{-B(x,y,\xi)} e^{-B(x,y,\eta)}Ê} \rho_x(\xi, \eta) \]

Therefore, for all admissible, algebraically visible quadruples $(\xi_1, \xi_2, \xi_3, \xi_4)$,
\begin{equation*}
\begin{split}
cr_y(\xi_1, \xi_2, \xi_3, \xi_4) & = \frac{\sqrt{e^{-B(x,y,\xi_1)-B(x, y, \xi_2)-B(x,y,\xi_3)-B(x,y,\xi_4)}}\rho_x(\xi_1, \xi_2) \rho_x(\xi_3, \xi_4)}{\sqrt{e^{-B(x,y,\xi_1)-B(x, y, \xi_2)-B(x,y,\xi_3)-B(x,y,\xi_4)}}\rho_x(\xi_1, \xi_3) \rho_x(\xi_2, \xi_4)}\\
& = cr_x(\xi_1, \xi_2, \xi_3, \xi_4).
\end{split}
\end{equation*}

We are left to check the special cases where $(\xi_1, \xi_2, \xi_3, \xi_4)$ is admissible, but not algebraically visible. If an admissible quadruple is not algebraically visible, there has to be at least one pair in the quadruple that is not algebraically visible. If $(\xi_1, \xi_2)$ or $(\xi_3, \xi_4)$ is not algebraically visible, then $cr_x(\xi_1, \xi_2, \xi_3, \xi_4) = 0 = cr_y(\xi_1, \xi_2, \xi_3, \xi_4)$. If $(\xi_1, \xi_3)$ or $(\xi_2, \xi_4)$ is not algebraically visible, then $cr_x(\xi_1, \xi_2, \xi_3, \xi_4) = \infty = cr_y(\xi_1, \xi_2, \xi_3, \xi_4)$. Otherwise, the equation from above still applies. We conclude that $cr_x = cr_y$ on all of $\mathcal{A}$.

\end{proof}

If $X$ is $\mathrm{CAT(-1)}$, it is a well-known result that $(\xi \vert \eta)_x$ is continuous with respect to the visual topology. For $\CAT$ spaces, this is not true anymore, which is illustrated by the fact that the Gromov product on the boundary of the euclidean plane obtains exactly the values zero and infinity and the set of pairs for which the Gromov product is infinite is dense. Nevertheless, some continuity properties remain true. We say that a bi-infinite geodesic $\gamma$ in $X$ is {\it rank 1} if and only if it does not bound an isometrically embedded half plane in $X$. Note that $\gamma$ might still have parallel geodesics, however there is a bound on the distance of any such parallel geodesic to $\gamma$.

\begin{lem} \label{lem:ContinuityGromovProductRank1}

Let $\xi_n \rightarrow \xi, \eta_n \rightarrow \eta$ be two converging sequences in $\partial X$. Then the following statements hold.

\begin{enumerate}

\item $(\xi \vert \eta)_x \leq \liminf\limits_{n \rightarrow \infty} (\xi_nÊ\vert \eta_n)_x$.

\item If $(\xi_n, \eta_n)$ is visible for all $n$ and $(\xi, \eta)$ can be connected by a rank 1 geodesic, then $\lim\limits_{n \rightarrow \infty} (\xi_n \vert \eta_n)_x = (\xi \vert \eta)_x$.

\end{enumerate}

\end{lem}

The proof of Lemma \ref{lem:ContinuityGromovProductRank1} requires several preliminary results. We begin with

\begin{lem} \label{lem:RightAngleExistence}
Let $\xi_n \rightarrow \xi$, $\eta_n \rightarrow \eta$ be two converging sequences in $\partial X$ such that $(\xi_n, \eta_n)$ can be connected by a geodesic line $\gamma_n$ for all $n$ and $(\xi, \eta)$ can be connected by a geodesic line $\gamma$. Then, for all $n$ sufficiently large, there exists a point $p_n \in \gamma_n$ such that $\angle_{\gamma(0)}(\xi, p_n) = \frac{\pi}{2}$.
\end{lem}

\begin{proof}
By Proposition 9.2 in Part II of \cite{BH}, the function $(\xi, \eta) \mapsto \angle_{\gamma(0)}(\xi, \eta)$ that sends two points in $X \cup \partial X$ to the angle between the unique geodesics from $\gamma(0)$ to $\xi$ and $\eta$ respectively is continuous in the cone topology (which restricts to the visual topology on the boundary). This implies that
\[ \angle_{\gamma(0)}(\xi, \xi_n) \rightarrow 0, \]
\[ \angle_{\gamma(0)}(\xi, \eta_n) \rightarrow \pi. \]
Therefore, we find $N$ such that for all $n \geq N$, we have
\[ \angle_{\gamma(0)}(\xi, \xi_n) \leq \frac{\pi}{4}, \]
\[ \angle_{\gamma(0)}(\xi, \eta_n) \geq \frac{3 \pi}{4}. \]
Let $\gamma_n$ be a geodesic from $\xi_n$ to $\eta_n$. Since the angle function varies continuously along $\gamma_n$, the intermediate value theorem tells us that there has to exist some point $p_n \in \gamma_n$, for which $\angle_{\gamma(0)}(\xi, p_n) = \frac{\pi}{2}$.
\end{proof}

Given a subset $A \subset X$ and $\epsilon > 0$, we denote the $\epsilon$-neighbourhood of $A$ by $N_{\epsilon}(A) := \{ x \in X \vert d(x,A) < \epsilon \}$. Next, we need

\begin{lem} \label{lem:GeodesicConvergenceRank1}
Let $\xi_n \rightarrow \xi$, $\eta_n \rightarrow \eta$ be two converging sequences in $\partial X$, such that $(\xi_n, \eta_n)$ is visible for all $n$ and $(\xi, \eta)$ can be connected by a rank 1 geodesic $\gamma$. Let $\gamma_n$ be a geodesic from $\xi_n$ to $\eta_n$. For $n$ sufficiently large, Lemma \ref{lem:RightAngleExistence} allows us to choose $p_n \in \gamma_n$, such that $\angle_{\gamma(0)}(\xi, p_n) = \frac{\pi}{2}$. Reparametrize $\gamma_n$, such that $\gamma_n(0) = p_n$.

Then for all $\epsilon > 0$, $T > 0$, there exists an $N$ such that for all $n \geq N$, there exists a geodesic $\tilde{\gamma}$ from $\xi$ to $\eta$ such that $\gamma_n\vert_{[-T, T]} \subset N_{\epsilon}(\tilde{\gamma}\vert_{[-T,T]})$.
\end{lem}

\begin{proof}[Proof of Lemma \ref{lem:GeodesicConvergenceRank1}]
We set the convention that, throughout this proof, $\tilde{\gamma}$ denotes a bi-infinite geodesic from $\xi$ to $\eta$, which is parametrised such that $B(\gamma(0), \tilde{\gamma}(0),Ê\xi) = 0$.

Suppose, the statement of the Lemma was not true. Then, we would find $\epsilon > 0, T > 0$ and subsequences $(\xi_{n_i})_i, (\eta_{n_i})_i$ such that for all $\tilde{\gamma}$, we find $t_{n_i} \in [-T, T]$ such that $d(\tilde{\gamma}(t_{n_i}), \gamma_{n_i}(t_{n_i})) \geq \epsilon$.

Step 1: There exists $N$, such that for all $n_i \geq N$ and for all $\tilde{\gamma}$, we have $d(\tilde{\gamma}(0), \gamma_{n_i}(0)) \geq \frac{\epsilon}{2}$.

Suppose, $d(\tilde{\gamma}(0), \gamma_{n_i}(0)) < \frac{\epsilon}{2}$. Since $\gamma$ is a rank 1 geodesic, the set $\{ \tilde{\gamma}(0) \vert \tilde{\gamma} \}$ is bounded. Therefore, we can choose $N$ sufficiently large, such that for all $\tilde{\gamma}$, we have
\[ d(\tilde{\gamma}(-T), \xi_{n_i, \tilde{\gamma}(0)}(T)) \leq \frac{\epsilon}{2}, \]
\[ d(\tilde{\gamma}(T), \eta_{n_i, \tilde{\gamma}(0)}(T)) \leq \frac{\epsilon}{2}. \]
By construction and assumption,
\[ d(\xi_{n_i, \tilde{\gamma}(0)}(0), \gamma_{n_i}(0)) < \frac{\epsilon}{2}, \]
\[ d(\eta_{n_i, \tilde{\gamma}(0)}(0), \gamma_{n_i}(0)) < \frac{\epsilon}{2}.\]
Since distance functions are convex in $\CAT$ spaces, this implies that
\[ d(\xi_{n_i, \tilde{\gamma}(0)}(T), \gamma_{n_i}(-T)) < \frac{\epsilon}{2},\]
\[ d(\eta_{n_i, \tilde{\gamma}(0)}(T), \gamma_{n_i}(T)) < \frac{\epsilon}{2} \]
and therefore,
\[ d(\tilde{\gamma}(t), \gamma_{n_i}(t)) < \epsilonÊ\]
for all $t \in [-T, T]$. This contradicts our assumption that we have $t_{n_i} \in [-T, T]$ such that for all $\tilde{\gamma}$, $d(\tilde{\gamma}(t_{n_i}), \gamma_{n_i}(t_{n_i})) \geq \epsilon$. We conclude that, for all $n_i \geq N$ and for all $\tilde{\gamma}$, $d(\tilde{\gamma}(0), \gamma_{n_i}(0)) \geq \frac{\epsilon}{2}$.\\

\begin{figure}
\begin{tikzpicture}[scale=1.50]
\draw [] (-3,0) -- (3,0);
\node [left] at (-3,0) {$\xi$};
\node [right] at (3,0) {$\eta$};
\draw [] (-2.9, 1) to [out=-30, in = 180] (-1.5, 0.5) to [out = 0, in= 180] (1.5,0.5) to [out = 0, in = 210] (2.9, 1);
\node [left] at (-2.9,1) {$\xi_{n_i}$};
\node [right] at (2.9,1) {$\eta_{n_i}$};
\draw [] (0,0) -- (0,0.5);
\node [right] at (0, 0.25) {$\delta_{n_i}$};
\draw [] (0,0.5) arc [radius = 1.5, start angle=0, end angle=45 ];
\draw [] (0,0.5) arc [radius = 1.5, start angle=0, end angle=-45 ];
\draw [] (0,0.5) arc [radius = 1.5, start angle=180, end angle=225 ];
\draw [] (0,0.5) arc [radius = 1.5, start angle=180, end angle=135 ];
\node [above] at (1.2,1.6) {$B(\gamma_{n_i}(0), \cdot, \eta_{n_i}) = 0$};
\node [above] at (-1.2,1.6) {$B(\gamma_{n_i}(0), \cdot, \xi_{n_i}) = 0$};
\end{tikzpicture}
\caption{Visualisation of proof: Since the geodesic from $\xi_{n_i}$ to $\eta_{n_i}$ stays away from the flat strip from $\xi$ to $\eta$, the geodesic $\delta_{n_i}$ has to move away from that flat strip. The shape of the two horospheres illustrates why $\phi_{n_i}$ is strictly decreasing along $\delta_{n_i}$. As $n_i \rightarrow \infty$, any finite segment of $\delta_{n_i}$ is pushed into the flat strip from $\xi$ to $\eta$, because $\xi_{n_i} \rightarrow \xi$ and $\eta_{n_i} \rightarrow \eta$.}
\label{fig:GeodesicConvergenceRank1} 
\end{figure}
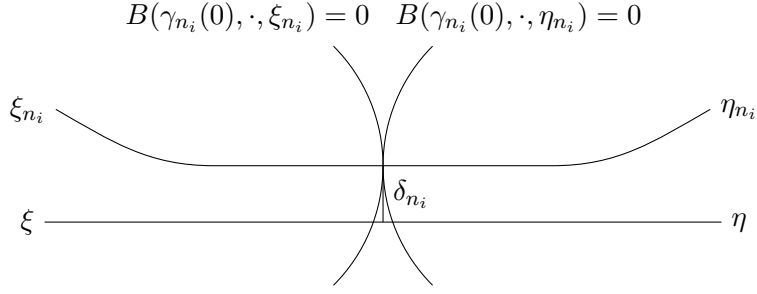

Step 2: Let $\delta_{n_i}$ be the geodesic from $\gamma(0)$ to $\gamma_{n_i}(0)$. Because of the way we parametrised $\gamma_n$ in the statement of the Lemma, we know that $\delta_{n_i}$ meets $\gamma\vert_{[0, -\infty)}$ at a right angle for all sufficiently large $n_i$. By the Theorem of Arzela-Ascoli, $\delta_{n_i}$ has a converging subsequence in compact-open topology. Passing to a subsequence if necessary, we assume without loss of generality that $\delta_{n_i}$ converges to a geodesic $\delta$. Since $\delta_{n_i}$ meets $\gamma\vert_{[0, -\infty)}$ at a right angle for all $n_i$, we know that the same is true for $\delta$. Therefore, $\delta$ cannot be a geodesic ray representing $\xi$ or $\eta$.

Choose $x \in \gamma$ and denote for all $\zeta \in \partial X$
\[ B_{\zeta}(x') := B(x, x', \zeta).Ê\]

Further, we define
\[ \phi_{n_i}(x') := B_{\xi_{n_i}}(x') + B_{\eta_{n_i}}(x') - B_{ \xi_{n_i}Ê} (\gamma(0)) - B_{ \eta_{n_i}}(\gamma(0)). \]

For $x \in \gamma$, we have $B(x, x', \xi) + B(x,x', \eta) \geq 0$ with equality if and only if $x'$ lies on a geodesic from $\xi$ to $\eta$. Using the cocycle equation, we see that
\[ \phi_{n_i}(\gamma_{n_i}(0)) \leq 0 \]
and
\[ \phi_{n_i}(\delta(0)) = \phi_{n_i}(\gamma(0)) = 0. \]

In particular, since Busemann functions are convex, $\phi_{n_i}(\delta_{n_i}(s)) \leq 0$ for all $s$, where $\delta_{n_i}$ is defined. Since $B_{\xi}(x)$ is continuous in $\xi$ and Lipschitz continuous in $x$, we see that on every compact interval, on which $\delta_{n_i}$ is defined for $n_i$ large, we have
\[ 0 \geq \phi_{n_i}(\delta_{n_i}(s)) \xrightarrow{i \rightarrow \infty} B_{\xi}(\delta(s)) + B_{\eta}(\delta(s)) \geq 0. \]

Therefore, $B_{\xi}(\delta(s)) + B_{\eta}(\delta(s)) \equiv 0$ for all $s$ for which $\delta$ is defined, which implies that $\delta$ lies completely in the set of points that are contained in geodesics from $\xi$ to $\eta$. Since $\xi$ and $\eta$ are connected by a rank 1 geodesic, the geodesic $\delta$ can only have infinite length if it represents either $\xi$ or $\eta$. As discussed above, this cannot happen given the way we constructed $\delta$. We conclude that $\delta$ has finite length. Therefore, the forward-endpoints of $\delta_{n_i}$ converge to the forward-endpoint of $\delta$, i.e. $\gamma_{n_i}(t_{n_i}) \rightarrow p$, where $p$ lies on a geodesic connecting $\xi$ with $\eta$ and -- because $\delta$ meets $\gamma\vert_{[0, -\infty)}$ at a right angle -- $B(\gamma(0), p, \xi) = 0$. Therefore, $p = \tilde{\gamma}(0)$ for some $\tilde{\gamma}$. This is a contradiction to our original assumption that the subsequence $\gamma_{n_i}(0)$ stays away from $\tilde{\gamma}(0)$ for all $\tilde{\gamma}$. This completes the proof.
\end{proof}

We also need another characterisation of the Gromov product. Consider $x \in X, \xi, \eta \in \partial X$ and let $h, h'$ be the horoballs centered at $\xi$ and $\eta$ respectively such that $x \in \partial h \cap \partial h'$. Denote by $h_{-m}$ the horoball centered at $\xi$ such that for any $y \in h_{-m}, B(x,y,\xi) \leq -m$ and analogously for $h'_{-m}$. Note that $\xi_x(t) \in \partial h_{-t}$ and $\eta_x(t) \in \partial h'_{-t}$. Define
\[ m_x(\xi, \eta) := \sup \{ m' \vert h_{-m'} \cap h'_{-m'} \neq \emptyset \}. \]

\begin{lem} \label{lem:GromovProductHorospheres}
Let $X$ be a proper, geodesically complete $\CAT$ space. For all $x \in X, \xi, \eta \in \partial X$, we have
\[ (\xi \vert \eta)_x = m_x(\xi, \eta).Ê\]

Furthermore, if $(\xi, \eta)$ is visible, then every geodesic $\gamma$ from $\xi$ to $\eta$ contains a point in $\partial h_{-m_x(\xi, \eta)} \cap \partial h'_{-m_x(\xi, \eta)}$ and every $p \in h_{-m_x(\xi, \eta)} \cap h'_{-m_x(\xi, \eta)}$ lies on a geodesic from $\xi$ to $\eta$.
\end{lem}

Before we prove this lemma, we introduce a convenient notation. Given two real numbers $a, b$ and $\delta > 0$, we write $a \asymp_{\delta} b$, whenever $\vert a - b \vert \leq \delta$.

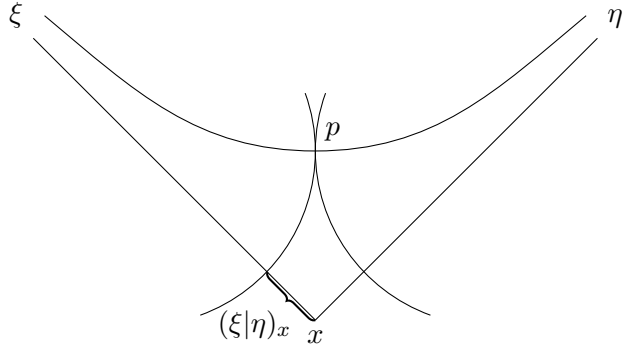
\begin{figure} 
\begin{tikzpicture}[scale=1.50]
\draw [] (0,0) -- (2.5,2.5);
\draw [] (0,0) -- (-2.5,2.5);
\node [above left] at (-2.5,2.5) {$\xi$};
\node [above right] at (2.5,2.5) {$\eta$};
\node [below] at (0,0) {$x$};
\draw [] (-2.4,2.7) to [out=-40, in = 180] (0, 1.5) to [out = 0, in = 220] (2.4, 2.7);
\draw [] (0,1.5) arc [radius = 1.5, start angle=0, end angle=20 ];
\draw [] (0,1.5) arc [radius = 1.55, start angle=0, end angle=-70 ];
\draw [] (0,1.5) arc [radius = 1.5, start angle=180, end angle=160 ];
\draw [] (0,1.5) arc [radius = 1.55, start angle=180, end angle=250 ];
\node [above right] at (0,1.5) {$p$};
\draw [decorate,decoration={brace}, thick] (0,0) -- (-0.425,0.425);
\node [below left] at (-0.12, 0.18) {$(\xi \vert \eta)_x$};
\end{tikzpicture}
\caption{The horospheres in the picture are the smallest ones centered at $\xi$ and $\eta$ that have non-empty intersection, provided that we shrink them at equal speed, starting with the horospheres containing $x$.}
\label{fig:GromovProductHorospheres} 
\end{figure}

\begin{proof}[Proof of Lemma \ref{lem:GromovProductHorospheres}]
Denote $m := m_x(\xi, \eta)$. We first show that $2(\xi \vert \eta)_x \leq 2m$. Suppose $(\xi \vert \eta)_x < \infty$. Let $m' > m$ and $\epsilon > 0$. There exists $t_0 \geq 0$, such that
\[ 2(\xi \vert \eta)_x \leq 2t_0 - d(\xi_x(t_0), \eta_x(t_0)) + \epsilon. \]
Let $\gamma_t$ be the geodesic from $\xi_x(t)$ to $\eta_x(t)$. Denote the unique intersection point of $\gamma_t$ with $\partial h_{-m'}$ by $p(t)$ and the unique intersection point of $\gamma_t$ with $\partial h'_{-m'}$ by $q(t)$. Since $m' > m$, we know that there is a segment of $\gamma_t$ that lies outside of $h_{-m'} \cup h'_{-m'}$. For $t > t_0$, we compute
\begin{equation*}
\begin{split}
2(\xi \vert \eta)_x & \leq 2t_0 - d(\xi_x(t_0), p(t_0)) - d(p(t_0), q(t_0)) - d(q(t_0), \eta_x(t_0)) + \epsilon\\
& \leq 2t - d(\xi_x(t), p(t_0)) - d(q(t_0), \eta_x(t)) + \epsilon\\
& \leq B(p(t_0),x,\xi) + B(q(t_0),x,\eta) + \epsilon\\
& = 2m' + \epsilon,
\end{split}
\end{equation*}
as $p(t_0) \in \partial h_{-m'}, q(t_0) \in \partial h_{-m'}$. Since this computation applies for all $\epsilon > 0$ and $m' > m$, we conclude that $(\xi \vert \eta)_x \leq m$ whenever $(\xi \vert \eta)_x < \infty$. If $(\xi \vert \eta)_x = \infty$, we do the same computation as above, except that we drop $\epsilon$ and instead find for every $C > 0$ a time $t_0$, such that $C \leq 2t - d(\xi_x(t), p(t_0)) - d(p(t_0), q(t_0)) - d(q(t_0), \eta_x(t)) \xrightarrow{t \rightarrow \infty} 2m'$.

Now, let $m' < m$, i.e. $h_{-m'} \cap h'_{-m'} \neq \emptyset$ and define $\gamma_t$ as above. Choose $p \in \partial h_{-m'} \cap h'_{-m'}$. Let $\epsilon > 0$. Then, for $t$ sufficiently large,
\begin{equation*}
\begin{split}
2m' & \leq B(p,x,\xi) + B(p,x,\eta)\\
& \leq 2t - d(\xi_x(t), p) - d(p, \eta_x(t)) + \epsilon\\
& \leq 2t - d(\xi_x(t), \eta_x(t)) + \epsilon\\
& \leq 2(\xi \vert \eta)_x + \epsilon.
\end{split}
\end{equation*}

Therefore, $(\xi \vert \eta)_x \geq m$, which concludes the proof of the first statement.\\

For the second statement of the Lemma, suppose $(\xi, \eta)$ is visible and again denote $m:= m_x(\xi, \eta)$. Let $\gamma$ be a geodesic from $\xi$ to $\eta$. Denote the unique intersection point of $\gamma$ with $\partial h_{-m}$ by $p$.

We claim that $p \in \partial h'_{-m}$. Suppose not. Since $m = \sup \{ m' \vert h_{-m'} \cap h'_{-m'} \neq \emptyset \}$, this implies that $B(x,p,\eta) > -m$. Therefore, there exists $\epsilon > 0$ such that $B(x,p,\eta) > -m + 2 \epsilon$. Since $m = m_x(\xi,Ê\eta)$, we find $q \in h_{-m + \epsilon} \cap h'_{-m + \epsilon}$. We compute
\begin{equation*}
\begin{split}
B(x,p,\xi) + B(x,p,\eta) & > -2m + 2\epsilon\\
& \geq B(x,q,\xi) + B(x,q,\eta)\\
& = B(x,p,\xi) + B(x,p,\eta) + B(p,q,\xi) + B(p,q,\eta).
\end{split}
\end{equation*}
This implies that $B(p,q,\xi) + B(p,q,\eta) < 0$, which is a contradiction to the fact that $p \in \gamma$. We conclude that $p \in \partial h'_{-m}$.

Note that $h_{-m} \cap h'_{-m} = \partial h_{-m} \cap \partial h_{-m}$. Otherwise, we would find a point $p$ such that $B(x,p,\xi) = B(x,p,\eta) = -m - \epsilon$ with $\epsilon > 0$, which contradicts the assumption that $m = m_x(\xi, \eta)$. We now show that every $q \in \partial h_{-m} \cap \partial h'_{-m}$ lies on a geodesic from $\xi$ to $\eta$ as well. Let $p$ be as above. We compute
\begin{equation*}
\begin{split}
B(p,q,\xi) + B(p,q,\eta) & = B(p,x,\xi) + B(x,q,\xi) + B(p,x,\eta) + B(x,q,\eta)\\
& = m + (-m) + m + (-m) = 0.
\end{split}
\end{equation*}

Since $p$ lies on a geodesic from $\xi$ to $\eta$, the sum $B(p,q,\xi) + B(p,q,\eta)$ equals zero if and only if $q$ also lies on a geodesic from $\xi$ to $\eta$. This proves the second part of the Lemma.
\end{proof}

\begin{proof}[Proof of Lemma \ref{lem:ContinuityGromovProductRank1}]
We first prove (1). Let $\epsilon > 0$. There exists $T \geq 0$, such that for all $t \geq T$, $(\xi_x(t) \vert \eta_x(t))_x \geq (\xi \vert \eta)_x - \epsilon$. Since $\xi_n \rightarrow \xi$ and $\eta_n \rightarrow \eta$, we find $N$ such that for all $n \geq N, \xi_n \in U_{T, \frac{\epsilon}{2}, x}(\xi)$ and $\eta_n \in U_{T, \frac{\epsilon}{2}, x}$. Thus,
\begin{equation*}
\begin{split}
(\xi_n \vert \eta_n)_x & \geq (\xi_{n,x}(T) \vert \eta_{n,x}(T))_x\\
& = T - \frac{1}{2} d(\xi_{n,x}(T), \eta_{n,x}(T))\\
& \geq T - \frac{1}{2} (d(\xi_x(T), \eta_x(T)) + 2\epsilon)\\
& \geq (\xi \vert \eta)_x - 2 \epsilon.
\end{split}
\end{equation*}

Since $\epsilon$ was chosen to be any positive number, we conclude that $\liminf\limits_{n \rightarrow \infty} (\xi_n \vert \eta_n)_x \geq (\xi \vert \eta)_x$.\\

To prove (2), we start by using Lemma \ref{lem:GromovProductHorospheres} to describe the Gromov product as follows. Let $(\xi, \eta)$ be visible and let $h, h'$ be as in the definition of $m_x(\xi, \eta)$. Denote the unique point where $\xi_x$ intersects $\partial h_{-m_x(\xi, \eta)}$ by $p$ and the unique point where $\eta_x$ intersects $\partial h'_{-m_x(\xi, \eta)}$ by $q$. Since $(\xi, \eta)$ is visible, we know from Lemma \ref{lem:GromovProductHorospheres} that $h_{-m_x(\xi, \eta)} \cap h'_{-m_x(\xi, \eta)}$ is non-empty and contains only points that are contained in a geodesic from $\xi$ to $\eta$. Let $r \in h_{-m_x(\xi, \eta)} \cap h'_{-m_x(\xi, \eta)}$. The Gromov product is equal to $m_x(\xi, \eta)$ which is the same as the distance $d(x,p)$. Note that the following equations hold by construction:
\[ d(x,p) = d(x,q) \]
\[ B(p,r,\xi) = B(q,r,\eta) = 0.Ê\]

Let $\epsilon > 0$. Let $\gamma_n$ be a bi-infinite geodesic from $\xi_n$ to $\eta_n$. By Lemma \ref{lem:GeodesicConvergenceRank1}, there exists (after reparametrisation) a subsequence $\gamma_{n_i}$ converging to a geodesic $\gamma$ from $\xi$ to $\eta$. Choose $r$ from above such that $r \in \gamma$. We find triples $(p_{n}, q_{n}, r_{n})$ as in the construction above, where we choose $r_{n} \in \gamma_{n}$. By (1), we know that $\liminf\limits_{n \rightarrow \infty} d(x,p_{n}) \geq d(x,p)$.

Suppose, $\liminf\limits_{n \rightarrow \infty} d(x,p_{n}) \geq d(x,p) + \epsilon$. By choice of $\gamma$, we know that $\gamma_{n_i}$ converges to $\gamma$ in compact-open topology. In particular, for $n_i$ sufficiently large, $r \in N_{\frac{\epsilon}{4}}(\gamma_{n_i})$ and we find $r'_{n_i} \in \gamma_{n_i}$, such that $d(r, r'_{n_i}) < \frac{\epsilon}{4}$. Therefore, for $n_i$ sufficiently large, $B(r, r_{n_i}, \xi_{n_i}) \asymp_{\frac{\epsilon}{4}} B(r'_{n_i},r_{n_i}, \xi_{n_i})$.

Furthermore, since $\xi_n \rightarrow \xi$ and $\eta_n \rightarrow \eta$, we can choose $n_i$ sufficiently large such that $d(\xi_x(d(x,p)), \xi_{n_i,x}(d(x,p))) < \frac{\epsilon}{4}$ and $d(\eta_x(d(x,p)), \eta_{n_i,x}(d(x,p))) < \frac{\epsilon}{4}$. Together with our assumption on the convergence behaviour of $d(x,p_n)$, we obtain that for all $n_i$ sufficiently large
\[ B(p_{n_i}, p, \xi_{n_i}) \geq B(p_{n_i}, \xi_{n_i,x}(d(x,p)), \eta_{n_i}) - \frac{\epsilon}{8} \geq d(x,p_{n_i}) - d(x,p) - \frac{\epsilon}{4} \geq \frac{3\epsilon}{4}, \]
\[ B(q_{n_i}, q, \eta_{n_i}) \geq B(q_{n_i}, \eta_{n_i,x}(d(x,p)), \eta_{n_i}) - \frac{\epsilon}{8} \geq d(x,q_{n_i}) - d(x,q) - \frac{\epsilon}{4} \geq \frac{3\epsilon}{4}, \]
and thus,
\[ B(p_{n_i}, r, \xi_{n_i}) \geq B(p,r,\xi_{n_i}) + \frac{3\epsilon}{4}, \]
\[ B(q_{n_i}, r, \eta_{n_i}) \geq B(q,r, \eta_{n_i}) + \frac{3\epsilon}{4}. \]

Finally, since Busemann functions $B(x,y,\xi)$ are continuous in $\xi$, we can choose $n_i$ sufficiently large such that
\[ B(p, r, \xi_{n_i}) \geq B(p,r,\xi) - \frac{\epsilon}{4}, \]
\[ B(q,r,Ê\eta_{n_i}) \geq B(q,r,\eta) - \frac{\epsilon}{4}. \]

Altogether, this implies that there exists $N \in \mathbb{N}$ such that for all $n \geq N$,
\begin{equation*}
\begin{split}
B(p_{n_i}, r_{n_i}, \xi_{n_i}) & \geq B(p_{n_i}, r, \xi_{n_i}) + B(r'_{n_i}, r_{n_i}, \xi_{n_i}) - \frac{\epsilon}{4}\\
& \geq B(p, r, \xi_{n_i}) + B(r'_{n_i}, r_{n_i}, \xi_{n_i}) + \frac{\epsilon}{2}\\
& \geq B(p,r,\xi) + B(r'_{n_i}, r_{n_i}, \xi_{n_i}) + \frac{\epsilon}{4}\\
& = B(r'_{n_i}, r_{n_i}, \xi_{n_i}) + \frac{\epsilon}{4}
\end{split}
\end{equation*}
\begin{equation*}
\begin{split}
B(q_{n_i}, r_{n_i}, \eta_{n_i}) & \geq B(q_{n_i}, r, \eta_{n_i}) + B(r'_{n_i}, r_{n_i}, \eta_{n_i}) - \frac{\epsilon}{4}\\
& \geq B(q, r, \eta_{n_i}) + B(r'_{n_i}, r_{n_i}, \eta_{n_i}) + \frac{\epsilon}{2}\\
& \geq B(q,r,\eta) + B(r'_{n_i}, r_{n_i}, \eta_{n_i}) + \frac{\epsilon}{4}\\
& = B(r'_{n_i}, r_{n_i}, \eta_{n_i}) + \frac{\epsilon}{4}.
\end{split}
\end{equation*}

However, since $r_{n_i}, r'_{n_i}$ both lie on the geodesic $\gamma_{n_i}$ from $\xi_{n_i}$ to $\eta_{n_i}$, we have
\[ B(p_{n_i}, r_{n_i}, \xi_{n_i}) + B(q_{n_i}, r_{n_i}, \eta_{n_i})) \leq 0 \]
and
\[ B(r'_{n_i}, r_{n_i}, \xi_{n_i}) + B(r'_{n_i}, r_{n_i}, \eta_{n_i})) = 0, \]
which is a contradiction to the inequalities above. We conclude that $\liminf\limits_{n \rightarrow \infty} d(x,p_n) = d(x,p)$. Since this argument applies to any subsequence of $(\xi_n, \eta_n)$ as well, we conclude that $\lim\limits_{n \rightarrow \infty} d(x,p_n)$ exists and equals $d(x,p)$. This concludes the proof.

\end{proof}

We prove one more Lemma that characterizes rank 1 geodesics in terms of a local visibility property.

\begin{lem} \label{lem:rank1characterisation}
Let $\gamma$ be a rank 1 geodesic from $\xi$ to $\eta$. Then, there exists an open neighbourhood $U \times V$ of $(\xi, \eta)$, such that for all $(\xi', \eta') \in U \times V$, $(\xi', \eta')$ is a visible pair.

In particular, a pair $(\xi, \eta)$ can be connected by a rank 1 geodesic if and only if there exists a neighbourhood $U$ of $\xi$, such that for all $\xi' \in U$, $(\xi', \eta)$ is visible.
\end{lem}

\begin{proof}
The proof uses a similar idea as the proof of Lemma \ref{lem:GeodesicConvergenceRank1}. Since $\gamma$ is a rank 1 geodesic, there exists a constant $C > 0$, such that every geodesic $\gamma'$ from $\xi$ to $\eta$ is parallel to $\gamma$ and has Hausdorff distance $d_{Haus}(\gamma, \gamma') \leq C$.

Suppose the Lemma was not true. Then, there exist sequences $\xi_i \rightarrow \xi$, $\eta_i \rightarrow \eta$, such that for all $i$, $(\xi_i, \eta_i)$ is not visible. Denote $x := \gamma(0)$, $\gamma^{-}_i$ the geodesic ray starting at $x$ representing $\xi_i$, and $\gamma^{+}_i$ the geodesic ray starting at $x$ representing $\eta_i$. Let $\gamma_{T,i}$ be the unique geodesic from $\gamma^{-}_i(T)$ to $\gamma^{+}_i(T)$. Note that, if we fix $i$, the paths $\gamma_{T,i}$ vary continuously in $T$ in the sense that $d_{Haus}(\gamma_{T+\epsilon,i}, \gamma_{T,i}) \leq 2\epsilon$. Denote the points at infinity obtained by extending $\gamma_{T,i}$ by $\xi_{T,i}$ and $\eta_{T,i}$ respectively.

Since $\gamma_{T,i}$ varies continuously in $T$, there exists a time $T_i$ such that $d(x, \gamma_{T,i}) = 2C$ and there exists a unique point $x_i \in \gamma_{T_i,i}$ satisfying $d(x, x_i) = 2C$. Note that $T_i \xrightarrow{i \rightarrow \infty} \infty$, as $\xi_i \rightarrow \xi$ and $\eta_i \rightarrow \eta$. We reparametrise $\gamma_{T_i,i}$ such that it is an arc-length geodesic with $\gamma_{T_i,i}(0) = x_i$. Since $X$ is assumed to be proper, the Arzela-Ascoli theorem implies the existence of a subsequence $\gamma_{T_{n_i}, n_i}$ that converges to a bi-infinite geodesic line $\tilde{\gamma}$ from $\tilde{\xi}$ to $\tilde{\eta}$ with $x_{n_i}$ converging to a point $\tilde{x} \in \tilde{\gamma}$. Without loss of generality, we denote these subsequences by $\gamma_{T_i, i}$ and $x_i$.

We claim that $\angle_x(\tilde{x}, \xi) = \angle_x(\tilde{x}, \eta) = \frac{\pi}{2}$. To prove this, we denote $\alpha_i := \angle_x(x_i, \gamma^{-}_i(T_i))$, $\alpha'_i := \angle_x(x_i, \gamma^{+}_i(T_i))$, $\beta_i := \angle_{x_i}(x, \gamma^{-}_i(T_i))$, $\beta'_i := \angle_{x_i}(x, \gamma^{+}_i(T_i))$. Since $x_i$ minimizes the distance $d(x, \gamma_{T_i, i}(t))$, we conclude that $\beta_i, \beta'_i \geq \frac{\pi}{2}$. Since the sum of angles of a triangle in a $\CAT$ space is at most $\pi$, this implies that $\alpha_i, \alpha'_i \leq \frac{\pi}{2}$. However, since $\xi_i \rightarrow \xi$ and $\eta_i \rightarrow \eta$, we have that $\lim_{i \rightarrow \infty} \alpha_i + \alpha'_i \geq \pi$ (the limit exists, since $x_i \rightarrow \tilde{x}$). We conclude that $\lim_{i \rightarrow \infty} \alpha_i = \lim_{i \rightarrow \infty} \alpha'_i = \frac{\pi}{2}$, which means that $\angle_x(\tilde{x}, \xi) = \angle_x(\tilde{x}, \eta) = \frac{\pi}{2}$. Combined with the fact that $d(x, \tilde{x}) = 2C$ and any geodesic parallel to $\gamma$ is contained in the $C$-neighbourhood of $\gamma$, this implies that $\tilde{x}$ does not lie in the flat strip spanned by all geodesic lines from $\xi$ to $\eta$.

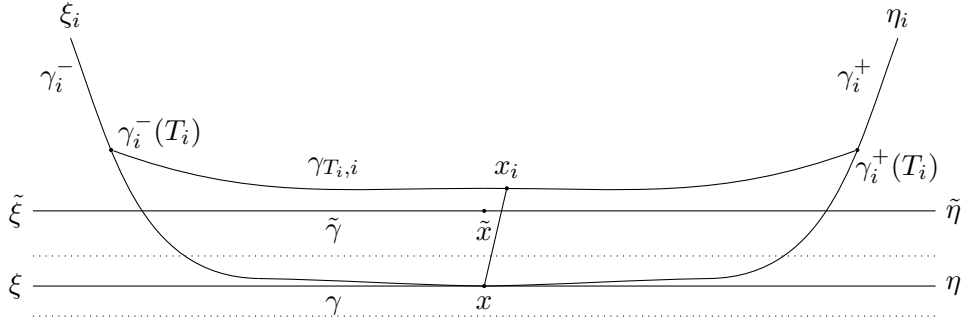
\begin{figure} 
\begin{tikzpicture}[scale=1]
%Geodesic gamma and flat strip
\draw (-6,0) -- (6,0);
\node [left] at (-6, 0) {$\xi$};
\node [right] at (6, 0) {$\eta$};
\node [below] at (-2, 0) {$\gamma$};
\draw [dotted] (-6,0.4) -- (6, 0.4);
\draw [dotted] (-6,-0.4) -- (6,-0.4);
\draw [fill] (0,0) circle[radius = 0.02cm];
\node [below] at (0,0) {$x$};
%Geodesic gamma^{+}_i
\draw [] (0,0) to [out = 1, in = 181] (3, 0.1);
\draw [] (3, 0.1) to [out = 1, in = 250] (5.5, 3.3);
\node [above] at (5.5, 3.3) {$\eta_i$};
\node [left] at (5.3, 2.8) {$\gamma^{+}_i$};
\draw [fill] (4.96, 1.81) circle [radius= 0.02cm];
\node [below right] at (4.8, 1.95) {$\gamma^{+}_i(T_i)$};
%Geodesic gamma^{-}_i
\draw [] (0,0) to [out = 179, in = -1] (-3, 0.1);
\draw [] (-3, 0.1) to [out = 179, in = -70] (-5.5, 3.3);
\node [above] at (-5.5, 3.3) {$\xi_i$};
\node [left] at (-5.3, 2.8) {$\gamma^{-}_i$};
\draw [fill] (-4.96, 1.81) circle [radius= 0.02cm];
\node [above right] at (-5, 1.7) {$\gamma^{-}_i(T_i)$};
%Geodesic gamma_{T_i,i}
\draw [] (0,1.3) to [out = 0, in = 200] (4.93, 1.8);
%\draw [dashed] (4.93, 1.8) to [out = 20, in = 200] (5.7, 2.05);
\draw [] (0,1.3) to [out = 180, in = -20] (-4.93, 1.8);
%\draw [dashed] (-4.93, 1.8) to [out = 160, in = -20] (-5.7, 2.05);
\node [above] at (-2, 1.3) {$\gamma_{T_i,i}$};
%\node [right] at (5.7, 2.05) {$\eta_{T_i, i}$};
%\node [left] at (-5.7, 2.05) {$\xi_{T_i, i}$};
\draw [fill] (0.3, 1.3) circle [radius = 0.02cm];
\node [above] at (0.3, 1.3) {$x_{i}$};
\draw [] (0,0) -- (0.3, 1.3);
%Geodesic \tilde{\gamma}
\draw [] (-6, 1) -- (6, 1);
\node [below] at (-2, 1.05) {$\tilde{\gamma}$};
\draw [fill] (0, 1) circle [radius = 0.02cm];
\node [below] at (0, 1) {$\tilde{x}$};
\node [right] at (6, 1) {$\tilde{\eta}$};
\node [left] at (-6, 1) {$\tilde{\xi}$};
\end{tikzpicture}
\caption{The times $T_i$ are chosen such that $d( x, x_{i}) = 2C$. This provides us with a subsequence of the geodesics $\gamma_{T_i,i}$ that converges to $\tilde{\gamma}$. The angles of the triangles $(x, x_i, \gamma^{\pm}_i(T_i))$ tell us that $\tilde{x}$ cannot lie in the flat strip from $\xi$ to $\eta$ (indicated by the dotted lines). However, the endpoints of $\tilde{\gamma}$ turn out to be $\xi$ and $\eta$, which leads to a contradiction.}
\label{fig:VisibilityofRankOne} 
\end{figure}

We now claim that $\tilde{\xi} = \xi$ and $\tilde{\eta} = \eta$, contradicting the fact that $\tilde{x}$ does not lie on any geodesic from $\xi$ to $\eta$. We show this by proving that $d_{Haus}(\gamma, \tilde{\gamma}) < \infty$. Fix $R > 0$. By the convergences established above, there exists $I$ such that for all $i \geq I$, we have $T_i \geq R$, and for all $\vert t \vert \leq R$, we have $d(\gamma_{T_i, i}(t), \tilde{\gamma}(t)) \leq C$, and $d(\gamma(t), \gamma^{\pm}_i(\vert t \vert)) \leq C$. We estimate for all $\vert t \vert \leq R$,
\begin{equation*}
\begin{split}
d(\gamma(t), \tilde{\gamma}(t)) & \leq d(\gamma(t), \gamma^{\pm}_i(\vert t \vert))\\
& \quad + d(\gamma^{\pm}_i(\vert t \vert), \gamma_{T_i, i}(t))\\
& \quad + d(\gamma_{T_i, i}(t), \tilde{\gamma}(t))\\
& \leq 4C,
\end{split}
\end{equation*}
where we used the fact that $\gamma^{\pm}_i(T_i) \in\gamma_{T_i, i}$ and convexity of distance functions to estimate $d(\gamma^{\pm}_i(\vert t \vert), \gamma_{T_i, i}(t)) \leq \max( d(x, x_i), d(\gamma^{\pm}_i(T_i), \gamma_{T_i, i}(T_i))) = 2C$.

This implies that $\tilde{\gamma}$ and $\gamma$ have bounded Hausdorff distance and, therefore, they are parallel. In particular, $\xi = \tilde{\xi}$ and $\eta = \tilde{\eta}$ and $\tilde{x}$ lies on a geodesic from $\xi$ to $\eta$. However, $\tilde{x}$ was constructed so that it cannot lie on such a geodesic. This is a contradiction and proves the Lemma.
\end{proof}

\begin{cor} \label{cor:ContinuityGromovProduct}
Let $x \in X, \xi, \gamma$ a rank 1 geodesic from $\xi$ to $\eta$. Then, $(Ê\cdot \vert \cdot )_x : \partial X \times \partial X \rightarrow [0, \infty]$ is continuous at $(\xi, \eta)$.
\end{cor}

We now define one of the properties necessary to make the circumcenter extension construction work.

\begin{mydef} \label{def:Rank1Hinge}
Let $X$ be a proper, connected, geodesically complete $\CAT$ space, $\xi \in \partial X$. We say that $\xi$ {\it is in a rank 1 hinge} if there exist $\eta, \zeta \in \partial X$, such that $(\eta, \zeta)$ is algebraically visible and the pairs $(\xi, \eta), (\xi, \zeta)$ both can be connected by a rank 1 geodesic.
\end{mydef}

%---------------------------------------------------------------------------------------------------------------------------------------
%METRIC DERIVATIVES
%---------------------------------------------------------------------------------------------------------------------------------------

\subsection{Metric derivatives} \label{subsec:MetricDerivatives}

In order to extend cross ratio preserving maps to maps of the interior, we need to generalize the notion of metric derivatives, which has been developed for general metric spaces (see \cite{Biswas15, Biswas17a, Biswas18a}). In this subsection, we show how this tool can be extended to boundaries of $\CAT$ spaces that `have sufficiently many algebraically visible pairs'. Since the underlying theory is more general, we will state the definitions and results in a more general form and then return to $\CAT$-spaces and boundaries.\\

Let $Z$ be a topological space, $\rho$ and $\rho'$ two non-negative, symmetric maps $\rho$, $\rho': Z \times Z \rightarrow [0, \infty]$ such that for all $z \in Z, \rho(z,z) = \rho'(z,z) = 0$. By analogy to the previous section, we call an $n$-tuple $(x_1, \dots, x_n) \in Z^n$ {\it algebraically visible} with respect to $\rho$ if and only if for all $i \neq j$, $\rho(x_i, x_j) > 0$. We say that a quadruple $(x_1,x_2,x_3,x_4) \in Z^4$ is {\it admissible} with respect to $\rho$ if it contains no triple $(x_i, x_j, x_k)$ with $i \neq j \neq k \neq i$ such that $\rho(x_i, x_j) = \rho(x_j,x_k) = 0$. Denote the set of quadruples admissible with respect to $\rho$ by $\mathcal{A}_{\rho}$. We will not indicate the $\rho$, whenever it is clear from context. Using the same formula as before, $\rho$ and $\rho'$ both define a cross ratio $cr_{\rho}$ and $cr_{\rho'}$ on the set $\mathcal{A}_{\rho}$ and $\mathcal{A}_{\rho'}$ respectively. We say that $\rho$ and $\rho'$ are {\it M\"obius equivalent} if $\mathcal{A}_{\rho} = \mathcal{A}_{\rho'}$ and $cr_{\rho} = cr_{\rho'}$. We write $\rho \overset{M}{\sim} \rho'$. Note that $\mathcal{A}_{\rho} = \mathcal{A}_{\rho'}$ if and only if $\rho$ and $\rho'$ define the same algebraically visible pairs.

\begin{mydef} \label{def:4visibility}

We say that $(Z,Ê\rho)$ satisfies the {\it 4-visibility} assumption, if the following holds:

\begin{enumerate}

\item[(4v)] For every quadruple $(z, x, x', y') \in Z^4$, there exists $w \in Z$, such that $w$ is algebraically visible with $z$, $x$, $x'$, $y'$.

\end{enumerate}

\end{mydef}

\begin{rem}
For any $n \in \mathbb{N}^{+}$, we can define the assumption ($n$v) by replacing quadruples by $n$-tuples. Note that ($n$v) implies ($k$v) for all $k \leq n$ and whenever $Z$ satisfies ($n$v), it has to contain at least $n+1$ points, as otherwise we could choose an $n$-tuple that contains all points in $Z$ to create a contradiction to ($n$v).

Further note that, if $(Z,Ê\rho)$ satisfies (4v) and $\rho \overset{M}{\sim} \rho'$, then $(Z, \rho')$ satisfies (4v) as well.
\end{rem}

We say that a point $z$ in $(Z, \rho)$ is {\it approximable}, if there exists a sequence $z_n \in Z$, such that $z_n \xrightarrow{n \rightarrow \infty} z$ and $(z, z_n)$ is algebraically visible for all $n$. Note that, if $\rho \overset{M}{\sim} \rho'$, then a point is approximable in $(Z, \rho)$ if and only if it is approximable in $(Z, \rho')$.

\begin{defprop}[cf. \cite{Biswas15}] \label{defprop:Derivative}

Suppose $\rho \overset{M}{\sim} \rho'$. Additionally, assume that $(Z, \rho)$ (and thus $(Z, \rho')$) satisfies (4v)). Let $z \in Z$ and choose $x, y \in Z$ such that $(x,y,z)$ is an algebraically visible triple with respect to $\rho$ (and thus $\rho'$). Then, the expression

\[ R_z(x,y) := \frac{\rho(z, x) \rho(z,y) \rho'(x,y)}{\rho(x,y) \rho'(z,x) \rho'(z,y)}, \]

is independent of the choice of $x$, $y$; it is continuous, whenever $\rho$ and $\rho'$ are continuous and, if $\rho$ and $\rho'$ are continuous, the following equality holds for every point $z \in Z$ that is approximable with respect to $\rho$:

\[ R_z(x,y) = \lim_{z' \rightarrow z, \rho(z,z') \neq 0} \frac{\rho(z, z')}{\rho'(z, z')}. \]

This equation motivates to define {\it the derivative of $\rho$ by $\rho'$ at $z$} by

\[ \frac{\partial \rho}{\partial \rho'}(z) := \frac{\rho(z, x) \rho(z,y) \rho'(x,y)}{\rho(x,y) \rho'(z,x) \rho'(z,y)} \]

\end{defprop}

\begin{proof}

We start by showing that every $z \in Z$ admits $x$, $y \in Z$, such that $(x,y,z)$ is an algebraically visible triple. Let $z \in Z$. We can extend $z$ to a quadruple $(z, a, b, c) \in Z^4$. By (4v), there exists a point $x \in Z$ that is algebraically visible to $z$, $a$, $b$, $c$. In particular, $z \neq x$. By extending the pair $(z,x)$ to a quadruple and using (4v) again, we obtain $y \in Z$, that is algebraically visible to both $z$ and $x$. We conclude that $(x,y,z)$ is an algebraically visible triple.\\

Next, we show independence of $x$ and $y$ for all possible choices of $x, y$. Let $x'$, $y' \in Z$ be another pair such that $(x', y', z)$ is an algebraically visible triple. We proceed in two steps.\\

Step 1: Suppose, one of the pairs $(x,x')$, $(x,y')$, $(y, x')$, $(y,y')$ is algebraically visible. Let's assume that $(x,x')$ is. We want to show that
\[ R_z(x,y) = \frac{\rho(z,x) \rho(z,y) \rho'(x,y)}{\rho(x,y) \rho'(z,x) \rho'(z,y)} = \frac{\rho(z,x') \rho(z,y') \rho'(x',y')}{\rho(x',y') \rho'(z,x') \rho'(z,y')} = R_z(x',y'). \]

This is true if and only if
\begin{equation*}
\begin{split}
\frac{\rho(z,x) \rho(z,y) \rho(x',y') \rho(x,x')}{\rho(x,y) \rho(z,x') \rho(z,y') \rho(x,x')}  = \frac{\rho'(z,x) \rho'(z,y) \rho'(x',y') \rho'(x,x')}{\rho'(x,y) \rho'(z,x') \rho'(z,y') \rho'(x,x')},
\end{split}
\end{equation*}

which is the same as
\[ cr_{\rho}(z,y,x',x) cr_{\rho}(z,x,y',x') = cr_{\rho'}(z,y,x',x) cr_{\rho'}(z,x,y',x'). \]

This last equation is true, since all appearing quadruples are admissible by assumption and $\rho \overset{M}{\sim} \rho'$. The cases, where $(x,y')$, $(y,x')$ or $(y,y')$ is algebraically visible are analogous.\\

Step 2: Suppose, all the pairs above are not algebraically visible. By assumption (4v), there exists a point $w \in Z$, which is algebraically visible with $z$, $x$, $x'$ and $y'$. By Step 1, we obtain that
\[ R_z(x,y) = R_z(w,y') = R_z(x',y'). \]

Therefore, $R_z(x,y) = R_z(x',y')$ for any two algebraically visible triples $(x,y,z)$, $(x',y',z)$.\\

In order to prove continuity, note that, if $\rho$ is continuous, algebraic visibility with respect to $\rho$ is an open condition and analogously for $\rho'$. Therefore, for any $z \in Z$, we find an open neighbourhood $U$ and a pair $(x,y)$, such that for all $z' \in U$, $(x,y,z')$ is an algebraically visible triple with respect to $\rho$ and $\rho'$. Thus, for all $z' \in U$, $\frac{\partial \rho}{\partial \rho'}(z') = R_{z'}(x,y)$, which is continuous in $z'$ by continuity of $\rho$ and $\rho'$.\\

Finally, if $z$ is approximable in $(Z,Ê\rho)$, we find a sequence of points $z_n$ that are algebraically visible with $z$ and converging to $z$. By continuity of $\rho$ and $\rho'$, we find a point $y \in Z$, such that $(z_n,y,z)$ is an algebraically visible triple for all sufficiently large $n$. Using the continuity of $\rho$ and $\rho'$ again, we obtain
\begin{equation*}
\begin{split}
\frac{\partial \rho}{\partial \rho'}(z) & = \lim_{n \rightarrow \infty} R_z(z_n,y)\\
& = \lim_{n \rightarrow \infty} \frac{\rho(z, z_n) \rho(z, y) \rho'(z_n, y)}{\rho(z_n, y) \rho'(z, z_n) \rho'(z,y)}\\
& = \frac{\rho(z,y) \rho'(z,y)}{\rho(z,y) \rho'(z,y)} \lim_{n \rightarrow \infty} \frac{\rho(z, z_n)}{\rho'(z,z_n)} = \lim_{n \rightarrow \infty} \frac{\rho(z, z_n)}{\rho'(z,z_n)}.
\end{split}
\end{equation*}

This implies that $\lim_{z' \rightarrow z} \frac{\rho(z,z')}{\rho'(z,z')}$ exists and the desired equality, which completes the proof.

\end{proof}

We require a few properties of these derivatives. If $\rho, \rho'$ are metrics, these properties are shown in \cite{Biswas15} and the proof is the same as here.

\begin{lem}[cf. \cite{Biswas15}] \label{lem:ChainRuleGMVT}
Let $\rho \overset{M}\sim \rho' \overset{M}\sim \rho''$, $z, z' \in Z$. Then

\begin{enumerate}[leftmargin = 0cm, rightmargin]

\item[(Chain rule)] $\frac{\partial \rho}{\partial \rho'}(z) \frac{\partial \rho'}{\partial \rho''}(z) = \frac{\partialÊ\rho}{\partial \rho''}(z)$

\item[(Geometric mean value theorem)] $\rho(z,z')^2 = \frac{\partial \rho}{\partial \rho'}(z) \frac{\partial \rho}{\partial \rho'}(z') \rho'(z,z')^2$

\end{enumerate}

\end{lem}

\begin{proof}
For the Chain rule, choose $x, y \in Z$ such that $(z,x,y)$ is an algebraically visible triple with respect to $\rho, \rho', \rho''$. Then
\[ \frac{\partial \rho}{\partial \rho'}(z) \frac{\partial \rho'}{\partial \rho''}(z) = \frac{\rho(z,x) \rho(z,y) \rho'(x,y) \rho'(z,x) \rho'(z,y) \rho''(x,y)}{\rho(x,y) \rho'(z,x) \rho'(z,y) \rho'(x,y) \rho''(z,x) \rho''(z,y)} = \frac{\partialÊ\rho}{\partial \rho''}(z). \]

For the Geometric mean value theorem, if $\rho(z,z') = 0$, the equation follows from $\mathcal{A}_{\rho} = \mathcal{A}_{\rho'}$. If $\rho(z,z') \neq 0$, we can choose $x \in Z$ such that $(z,z',x)$ is an algebraically visible triple. Then

\[ \frac{\partial \rho}{\partial \rho'}(z) \frac{\partial \rho}{\partial \rho'}(z') = \frac{\rho(z,z') \rho(z,x) \rho'(z',x) \rho(z',z) \rho(z',x) \rho'(z,x)}{\rho(z',x) \rho'(z,z') \rho'(z,x) \rho(z,x) \rho'(z',z) \rho'(z',x)} = \frac{\rho(z,z')^2}{\rho'(z,z')^2} \]

\end{proof}

\begin{rem}
If at least one point $z \in Z$ is approximable in $(Z, \rho)$ and both $\rho$ and $\rho'$ are continuous, then it is easy to see from the characterization of the derivative at the approximable point $z$ by $\frac{\partial \rho}{\partial \rho'}(z) = \lim_{z' \rightarrow z} \frac{\rho(z,z')}{\rho'(z,z')}$ that the Geometric mean value theorem uniquely determines the derivative of $\rho$ by $\rho'$.
\end{rem}

\begin{rem}[cf. \cite{Biswas15}] \label{rem:DerivativesonBoundaries}
Using Lemma \ref{lem:derivative} and the additivity of Busemann functions, it is easy to see that on boundaries of $\CAT$ spaces,
\[ \frac{\partial \rho_x}{\partial \rho_y}(\xi) = e^{B(x,y,\xi)}. \]
\end{rem}

\begin{lem}[cf. \cite{Biswas15}] \label{lem:MaxMin}
Let $\rho \overset{M}\sim \rho'$. Additionally, assume that $Z$ is compact, for all $z, z' \in Z$, $\rho(z,z') \leq 1, \rho'(z,z') \leq 1$ and that for every $z \in Z$ there exist $\bar{z}, \bar{z}' \in Z$ such that $\rho(z,\bar{z}) = 1$ and $\rho'(z, \bar{z}') = 1$. Then,
\[ \max_{z \in Z} \left\{ \frac{\partial \rho}{\partial \rho'}(z) \right\} \min_{z \in Z} \left\{ \frac{\partial \rho}{\partial \rho'}(z) \right\} = 1 \]
\end{lem}

Note that, if $Z = \partial X$ and $\rho = \rho_x, \rho' = \rho_{x'}$, then the assumptions of Lemma \ref{lem:MaxMin} are satisfied, so this Lemma applies in the context that we will be considering.

\begin{proof}
Let $z \in Z$ such that $\frac{\partial \rho}{\partial \rho'}(z)$ is maximal and $z' \in Z$ such that $\frac{\partial \rho}{\partial \rho'}(z')$ is minimal. Denote the obtained maximum and minimum by $\mu$ and $\lambda$ respectively. Let $\bar{z}' \in Z$ be such that $\rho'(z,\bar{z}') = 1$. then

\begin{equation*}
\begin{split}
1 \geq \rho(z,\bar{z}')^2 & = \frac{\partial \rho}{\partial \rho'}(z) \frac{\partial \rho}{\partial \rho'}(\bar{z}') \rho'(z,\bar{z}')^2 \geq \mu \lambda.
\end{split}
\end{equation*}

On the other hand, let $\bar{z} \in Z$ be such that $\rho(z',\bar{z}) = 1$. Then

\begin{equation*}
\begin{split}
1 \geq \rho'(z',\bar{z})^2 = \frac{\partial \rho'}{\partial \rho}(z') \frac{\partial \rho'}{\partial \rho}(\bar{z}) \rho(z',\bar{z})^2 \geq \frac{1}{\lambda}\frac{1}{\mu}.
\end{split}
\end{equation*}

We conclude that $\mu \cdot \lambda = 1$.

\end{proof}

%---------------------------------------------------------------------------------------------------------------------------------------
%CONVEX FUNCTIONS
%---------------------------------------------------------------------------------------------------------------------------------------

\subsection{Convex functions} \label{subsec:ConvexFunctions}

We need some basic results about convex functions. A function $f : I \rightarrow \mathbb{R}$ defined on an interval $I \subset \mathbb{R}$ is called {\it convex} if for all $a, b \in I$ and $t \in [0,1]$, we have $f( (1-t)a + tb) \leq (1-t) f(a) + tf(b)$. A function is called {\it strictly convex} if this inequality is a strict inequality for all $t \in (0,1)$.

In a geodesic metric space $X$, a function $f : X \rightarrow \mathbb{R}$ is called {\it convex} if for any geodesic $\gamma$ on $X$ and any $a, b$ on the domain of $\gamma$, we have
\[ \forall t \in [0,1] : f(\gamma( (1-t)a + tb ) ) \leq (1-t) f( \gamma(a) ) + t f( \gamma(b) ). \]

\begin{rem}
There is a sufficient, but generally not necessary analytic condition for (strict) convexity. If $f : I \rightarrow \mathbb{R}$ is a $C^2$-function, then $f$ is convex if and only if $\frac{\partial^2 f}{\partial t^2} \geq 0$ everywhere. Furthermore, if $\frac{ \partial^2 f}{\partial t^2} > 0$ everywhere, then $f$ is strictly convex. However, the converse is not necessarily true, as is illustrated by the example $t \mapsto t^4$ at the point zero.
\end{rem}

We recall the following standard result about convex functions.

\begin{lem} \label{lem:Convexity}
Let $f_{z} : X \rightarrow \mathbb{R}$ be a family of convex functions on a connected, geodesic $\CAT$ space $X$ parametrized by $z \in Z$. Define $F(x) := \sup_{z \in Z} \{ f_{z}(x) \}$. Then $F : X \rightarrow \mathbb{R}$ is convex.
\end{lem}

\subsection{Visibility and algebraic visibility} \label{subsec:VisibilityAlgebraicVisibility}

Let $X, Y$ be proper, connected, geodesically complete $\CAT$ spaces. A map $f : \partial X \rightarrow \partial Y$, is called {\it M\"obius} if and only if it sends algebraically visible pairs to algebraically visible pairs and preserves the cross ratio, i.e.
\[ \forall (\xi_1, \xi_2, \xi_3, \xi_4) \in \mathcal{A} : cr_X(\xi_1, \xi_2, \xi_3, \xi_4) = cr_Y(f(\xi_1), f(\xi_2), f(\xi_3), f(\xi_4)).Ê\]

In order to construct our extension map, we require that $f$ is not only M\"obius but also that $f$ and $f^{-1}$ both preserve visible pairs. It is tempting to try and show that M\"obius maps always preserve visible pairs by arguing that a pair $(\xi, \eta)$ in $\partial X$ is visible if and only if it is algebraically visible. It is known that visible pairs are always algebraically visible. However, while the converse is true if $X$ admits a cocompact group action by isometries, it is not true in general, as the following example -- provided by Jean-Claude Picaud and Viktor Schroeder -- illustrates.\\

Consider the manifold $\mathbb{R}^2$ with coordinates $(x,y)$ and equip it with the Riemannian metric $dx^2 + f(x)^2 dy^2$, where $f : \mathbb{R} \rightarrow \mathbb{R}$ is a $C^2$-function, such that $f(x) > 1$ for all $x$ and $\lim_{x \rightarrow \infty} f(x) = 1$. The curvature of this metric at $(x,y)$ is given by $-\frac{f''(x)}{f(x)}$. Hence, if $f$ is strictly convex, this space has negative curvature everywhere. We equip the tangent space of $\mathbb{R}^2$ with the standard basis $e_1, e_2$ everywhere. We denote the inner product with respect to the Riemannian metric above by $\langle \cdot, \cdot \rangle_f$.

This Riemannian manifold is the universal covering of a surface of revolution $\mathbb{R} \times S^1$ with coordinates $(x, \vartheta)$ and Riemannian metric $dx^2 + f(x) d\vartheta^2$. By abuse of notation, we call the projection of the vector fields $e_1, e_2$ onto the surface of revolution by $e_1, e_2$ as well. It is a classical result that a path $\gamma$ on a surface of revolution is a geodesic in the Riemannian sense if and only if the function $\langle \gamma'(t), e_2(\gamma(t)) \rangle_f$ is constant. (This is called Clairaut's constant, cf. \cite{doCarmo}.) We observe from this fact that a geodesic $\gamma(t) = (x(t), y(t))$ with $x'(0) > 0$, will have monotone increasing $x(t)$ for all $t \geq 0$ if and only if its Clairaut constant $\langle \gamma'(t), e_2\rangle_f \leq \lim_{t \rightarrow \infty} \langle e_2, e_2\rangle_f = \lim_{t \rightarrow \infty} f(t)^2 = 1$. Else, the geodesic $\gamma$ will eventually change its $x$-direction and have decreasing $x(t)$. This argumentation carries over to the universal covering, where we conclude that a geodesic ray $\gamma$ represents a point in the boundary with $x(t) \xrightarrow{ t \rightarrow \infty} \infty$ if and only if $\vert \langle\gamma'(0), e_2\rangle_f \vert \leq 1$. The Clairaut constant also implies that no two geodesic rays with $x(t) \xrightarrow{t \rightarrow \infty} \infty$ can be connected by a bi-infinite geodesic. Thus, any pair of geodesics with $x'(0) > 0$ and Clairaut constant at most one is a non-visible pair.

We focus our attention on the borderline case where the absolute value of the Clairaut constant equals one, i.e. $\vert \langle \gamma'(t), e_2\rangle \vert \equiv 1$. Fixing $(x_0, y_0) \in \mathbb{R}^2$, there are exactly two geodesic rays starting at $(x_0, y_0)$ whose Clairaut constant in absolute value equals $1$. We will show that, depending on the choice of the function $f$, this pair of points in the boundary may be algebraically visible or not algebraically visible.\\

We start with some general arguments that will allow us to reverse engineer the function $f$ and $y$, assuming that we know the $x$-coordinate of a geodesic with Clairaut constant one. Suppose, we have a geodesic $\gamma$ with a known $x$-coordinate. We know that the following two equations hold:
\[ \pm 1 \equiv \langle \gamma'(t), e_2 \rangle_f = f(x(t))^2 y'(t). \]
\[ 1 \equiv \langle \gamma'(t) , \gamma'(t) \rangle_f = x'(t)^2 + f(x(t))^2 y'(t)^2Ê\]

This implies that
\[ y'(t) = \pm \frac{1}{f(x(t))^2},Ê\]
\[ f(x(t))^2 = \pm \frac{1}{1 - x'(t)^2}.Ê\]

We now use these equations in two concrete cases.

\begin{exam}
Restrict to $t \geq 2$ and suppose, $x(t) = \ln(t)$. By the equations above, using the fact that we also require $f(x) > 1$, we obtain
\[ f(x(t)) = \sqrt{ \frac{t^2}{t^2 - 1} } = \frac{t}{\sqrt{t^2-1}},Ê\]
\[ y'(t) = \pm \left( 1 - \frac{1}{t^2} \right).Ê\]

In particular, we obtain
\[ f(x) = \frac{e^x}{\sqrt{e^{2x}-1}}, \]
which is a strictly convex function for $x > 0$ with $\lim_{x \rightarrow \infty} f(x) = 1$, as direct computation shows.

For every starting point $p_0$, we obtain two geodesic rays starting at that point that are described by the equations above. Choose some $p_0$ and denote the two geodesics starting there by $\gamma^{+}, \gamma^{-}$. We claim that their Gromov product is finite. Since $\int_2^\infty \frac{1}{t^2} dt < \infty$, we obtain that there exists some constant $C$ such that $y^{+}(t) > t - C$ and $y^{-}(t) < -t + C$ for all $t$. Let $\delta$ be the path connecting $\gamma^{+}(t), \gamma^{-}(t)$. Since the euclidean inner product satisfies $\langle \cdot, \cdot \rangle_{Eucl} \leq \langle \cdot, \cdot \rangle_f$, we see that the euclidean distance satisfies $2t - 2C \leq d_{Eucl}(\gamma^{+}(t), \gamma^{-}(t)) \leq d_f(\gamma^{+}(t), \gamma^{-}(t))$. Therefore, the Gromov product satisfies $( [\gamma^{+}] \vert[\gamma^{-}] )_{p_0} \leq 2C < \infty$. In particular, here we have an example of a non-visible pair that is algebraically visible.

\end{exam}

\begin{exam}

Restrict to $t > 1$, choose $\alpha \in (0, \frac{1}{2})$ and suppose, $x(t) = \frac{1}{1-\alpha}t^{1-\alpha}$. We obtain
\[ f(x(t)) = \frac{1}{\sqrt{ 1 - t^{-2\alpha} }},Ê\]
\[ y'(t) = \pm( 1 - t^{-2\alpha} ), \]
\[ y(t) = \pm( t - \frac{1}{1 - 2\alpha} t^{1 - 2 \alpha} ) + C.Ê\]

In particular,
\[ f(x) = \frac{1}{ \sqrt{ 1 - (1-\alpha)^{-\frac{2\alpha}{1 - \alpha}} x^{-\frac{2\alpha}{1-\alpha} } }}.Ê\]

Abbreviating $\sigma := \frac{1}{1-2\alpha}$ and $\tau := (1-\alpha)^{-\frac{2\alpha}{1-\alpha}}$, we rewrite
\[ f(x) = \frac{1}{\sqrt{ 1 - \tau x^{-\frac{2\alpha}{1-\alpha}} }} \]
\[ y(t) = \pm( t - \sigma t^{1-2\alpha} ) + C.Ê\]

Again, a computation shows that $f'' > 0$ and $f(x) \xrightarrow{x \rightarrow \infty} 1$. Again, we obtain two geodesics $\gamma^{+}, \gamma^{-}$ starting at the same starting point $p_0$, described by these equations. We claim that their Gromov product is infinite. For this, it is sufficient to show that $d(\gamma^{+}(t), \gamma^{-}(t)) \leq t - \psi(t)$ for some function $\psi \xrightarrow{t \rightarrow \infty} \infty$. Since $d_{Eucl}(\gamma^{+}(t), \gamma^{-}(t)) \leq 2 (t - \sigma t^{1-2\alpha}) + C'$ is the euclidean length of the euclidean geodesic between $\gamma^{+}(t), \gamma^{-}(t)$ and since $\sigma > 1$, we obtain that
\begin{equation*}
\begin{split}
d_f(\gamma^{+}(t), \gamma^{-}(t)) & \leq 2 f(x(t)) (t -  \sigma t^{1-2\alpha}) + f(x(t))C'\\
& \leq 2 \frac{1}{\sqrt{1 - t^{-2\alpha}}} (t - t^{1-2\alpha}) + \frac{C'}{\sqrt{1 - t^{-2\alpha}}}\\
& = 2t  \sqrt{ 1 - t^{-2\alpha}} + \frac{C'}{\sqrt{1 - t^{-2\alpha}}}\\
& \leq 2t (1 - \frac{1}{2} t^{-2\alpha} ) + \frac{C'}{\sqrt{1 - t^{-2\alpha}}}\\
& = 2t - \psi(t),
\end{split}
\end{equation*}
where $\psi(t) = t^{1 - 2\alpha} - \frac{C'}{\sqrt{1 - t^{-2\alpha}}} \xrightarrow{t\rightarrow \infty} \infty$. We conclude that $([\gamma^{+}] \vert [\gamma^{-}])_{p_0} = \infty$.

\end{exam}

These examples illustrate why we will assume not only that $f$ is M\"obius but also that it preserves visible pairs in the coming sections.

%---------------------------------------------------------------------------------------------------------------------------------------
%JACOBI FIELDS
%---------------------------------------------------------------------------------------------------------------------------------------

\subsection{Jacobi fields} \label{subsec:JacobiFields}

We now move fully into the realm of Riemannian manifolds. We refer to \cite{doCarmo} for all necessary background informations. Let $X$ be an $n$-dimensional, connected, geodesically complete Riemannian manifold such that all sectional curvatures are non-positive. Let $\gamma$ be a geodesic in $X$. A vector field $J$ defined along $\gamma$ is called a {\it Jacobi-field} if and only if it satisfies the following second-order ordinary differential equation:
\[ \frac{D^2}{dt^2}J(t) + R(J(t), \gamma'(t))\gamma'(t). \]

where $R$ denotes the Riemannian curvature tensor and $\frac{D}{dt}$ the covariant derivative along $\gamma$ with respect to the Levi-Civita connection. Any Jacobi field along $\gamma$ is uniquely determined by the initial conditions $J(0), \frac{DJ}{dt}(0)$. The space of Jacobi fields along $\gamma$ forms a real $2n$-dimensional vector space.

On complete manifolds, Jacobi fields are uniquely characterised as the vector fields arising from smooth one-parameter families of geodesics $\gamma_s$ with $\gamma_0 = \gamma$. The Jacobi field corresponding to $( \gamma_s)_s$ is given by $J(t) = \frac{d}{ds}\vert_{s=0} \gamma_s(t)$. A Jacobi field is called {\it perpendicular}, if $J(t) \perp \gamma'(t)$ for all $t$. A Jacobi field is called {\it stable} if $\sup_{t \geq 0}\{ \Vert J(t) \Vert^2 \} < \infty$. A Jacobi field is called {\it parallel} if $\Vert J(t) \Vert^2$ is constant along all of $\gamma$.\\

We now define a subset of $X$ that consists of all the points that have `asymptotic features of flatness'. Specifically,
\begin{equation*}
\begin{split}
F_X := \{ x \in X \vert & \exists \, \gamma \text{ geodesic ray, starting at $x$ and}\\
& \exists \, J \text{ perpendicular, parallel Jacobi field along } \gamma \}.
\end{split}
\end{equation*}

We first note that, whenever $x \in F_X$, we find a geodesic $\gamma$ as in the definition of $F_X$ and every point on $\gamma$ is contained in $F_X$. The following result that goes back to Eberlein connects the complement of $F_X$ with a convexity property of horospheres.

\begin{prop}[Lemma 3.1 in \cite{HiH}] \label{prop:HiH}
Let $X$ be a Hadamard manifold, $\Xi$ the radial field in the direction of $\xi \in \partial X$ and $B$ a Busemann function centered at $\xi$. Then $\Xi = - \grad(B), \Xi$ is $C^1$ and $\nabla_v \Xi= \frac{DJ}{dt}(0)$ for all $v \in T_xX$, where $J$ is the unique stable Jacobi field along the geodesic ray $\xi_x$ such that $J(0) = v$.
\end{prop}

The covariant derivative $\nabla_v \Xi$ can be thought of as a second derivative of the Busemann function $B$, because for all $v, w \in T_x X$,
\[ \frac{d}{dt}\vert_{t=0} \frac{d}{ds}\vert_{s=0} B(x, \gamma(t,s), \xi) = \langle \nabla_v\Xi, w\rangle,Ê\]
where $\gamma(t,s) = \exp_x(tv+sw)$.

%---------------------------------------------------------------------------------------------------------------------------------------

%DEFINITIONS OF \PHI AND F

%---------------------------------------------------------------------------------------------------------------------------------------

\section{Construction of $\Phi$ and $F$} \label{sec:Definition}

For the rest of this paper, let $X$, $Y$ be $n$-dimensional, connected, simply connected, geodesically complete Riemannian manifolds such that their sectional curvatures are bounded by $-b^2 \leq curv \leq 0$. Further, assume that $\partial X$ and $\partial Y$ satisfy (4v) and that all points in $\partial X$ and $\partial Y$ are in a rank 1 hinge. We denote the unit tangent bundle of $X$ by $T^1X$. Further, we have the canonical projection $\pi_X : TX \rightarrow X$. If the manifold $X$ is clear from the context, we simply write $\pi$. For all $x \in X, \xi \in \partial X$ we denote the unit tangent vector in $T_x^1X$ that `points to $\xi$', i.e. whose induced geodesic ray represents $\xi$, by $\overrightarrow{x\xi}$. This provides us with a homeomorphism between $\partial X$ and $T_x^1 X$ equipped with the standard topology. Analogously, for any two points $x, x' \in X$, we denote the tangent vector of the arc-length geodesic from $x$ to $x'$ at $x$ by $\overrightarrow{xx'}$.

Let $f : \partial X \rightarrow \partial Y$ be a M\"obius homeomorphism such that $f$ and $f^{-1}$ both preserve visible pairs. Our goal is to extend $f$ to a map $F : X \rightarrow Y$. The construction presented in this section is a generalisation of a construction by Biswas for $\mathrm{CAT(-1)}$ spaces. Its most similar presentation to the one below can be found in \cite{Biswas17b}.

%---------------------------------------------------------------------------------------------------------------------------------------
%DEFINING \PHI
%---------------------------------------------------------------------------------------------------------------------------------------

\subsection{Constructing $\Phi$} \label{subsec:Phi}

We start by constructing a map between the tangent bundles of $X$ and $Y$. However, it turns out that this map can only be defined after identifying certain vectors in the tangent bundle.

Let $v \in TX$. The geodesic flow on $X$ provides us with a unique bi-infinite geodesic $\gamma$ such that $\gamma'(0) = v$. Denote the two endpoints of $\gamma$ at infinity by $v_{-\infty} := \gamma(-\infty)$ and $v_{\infty} := \gamma(\infty)$. Let $v, w \in T^1X$ and denote their projection in $X$ by $x$ and $x'$ respectively. We say that $v \sim w$, if $\Vert v \Vert = \Vert w \Vert$, $v_{\infty} = w_{\infty}, v_{-\infty} = w_{-\infty}$ and $B(x, x', v_{\infty})$. Note that this is equivalent to the convex hull of the geodesics induced by $v$ and $w$ being a flat strip (see Theorem 2.13 in Part \RomanNumeralCaps{2} of \cite{BH}) and the foot points of $v$ and $w$ being on the same horosphere with respect to either endpoint of the strip. This defines an equivalence relation on $TX$ and we denote the quotient by $\overline{TX}$. Denote the quotient of the unit tangent bundle by the same equivalence relation by $\overline{T^1X}$. The equivalence class of a vector $v$ will be denoted by $[v]$. Since $v \sim w \Leftrightarrow -v \sim -w$, we define $-[v] := [-v]$.

We construct a map $\Phi : \overline{T^1X} \rightarrow \overline{T^1Y}$ which will be a geodesic conjugacy in the sense of Lemma \ref{lem:GeodesicConjugacy}. Let $v \in T_x^1X$ be a unit-vector. As above, we obtain two points $v_{\infty}, v_{-\infty}$ at infinity. Since $f$ preserves visible pairs, there exists at least one geodesic from $f(v_{-\infty})$ to $f(v_{\infty})$. Choose one such geodesic and denote it by $\gamma$. The image of $[v]$ under $\Phi$ will be the equivalence class of a unit-vector on the geodesic $\gamma$ pointing towards $f(v_{\infty})$. All that is left is to choose the foot point on $\gamma$.

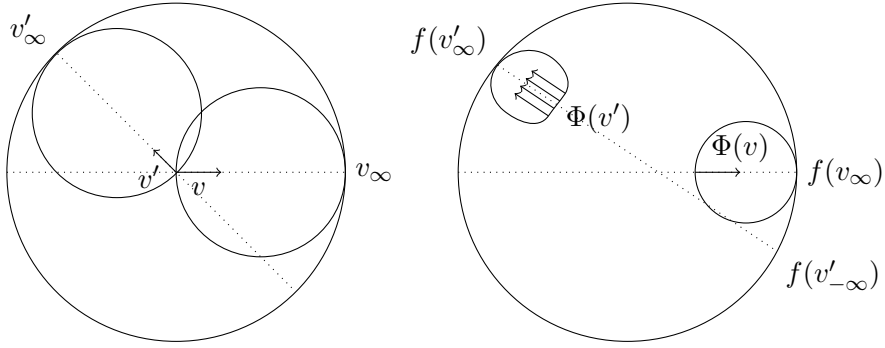
\begin{figure} 
\begin{tikzpicture}[scale=1.50]
\draw [] (-2,0) circle [radius=1.5cm];
\draw [] (2,0) circle [radius=1.5cm];
\draw [->] (-2,0) -- (-1.6,0);
\node[below] at (-1.8,0) {$v$};
\draw [->] (-2,0) -- (-2.2,0.2);
\node[below left] at (-2.05,0.15) {$v'$};
\draw [] (-1.25,0) circle [radius=0.75cm];
\draw [] (-2.525,.525) circle [radius=0.75cm];
\draw [->] (2.6,0) -- (3,0);
\node[above] at (3, 0) {$\Phi(v)$};
\draw [] (3.05,0) circle [radius=0.45cm];
\draw [->] (1.3,0.5) -- (1,0.7);
\draw [->] (1.45,0.7) -- (1.15,0.9);
\draw [->] (1.35,0.5666) -- (1.05,0.7666);
\draw [->] (1.4,0.6333) -- (1.1,0.8333);
\node[below right] at (1.365, 0.7) {$\Phi(v')$};
\draw [] (1.45,0.7) to [out=60, in = -40] (1.35, 1) to [out = 140, in= 50] (0.85,0.95) to [out = 230, in = 140] (0.9, 0.55) to [out=-40, in = 235] (1.3,0.5) to [out=55, in = 235] (1.45, 0.7);
\draw [dotted] (-3.5,0) -- (-0.5,0);
\draw [dotted] (-3.05, 1.05) -- (-0.95, -1.05);
\node [right] at (-0.5,0) {$v_{\infty}$};
\node [above left] at (-3.05, 1.05) {$v'_{\infty}$};
\draw [dotted] (0.5,0) -- (3.5,0);
\draw [dotted] (0.85,0.95) -- (3.31,-0.69);
\node[above left] at (0.85, 0.95) {$f(v'_{\infty})$};
\node[below right] at (3.31, -0.69) {$f(v'_{-\infty})$};
\node[right] at (3.5,0) {$f(v_{\infty})$};
\end{tikzpicture}
\caption{The vector $v$ is sent to the vector $\Phi(v)$. The derivative $\frac{Ê\partial f_* \rho_x}{\partial \rho_y}(f(v_{\infty}))$ determines, which horosphere $\Phi(v)$ needs to be placed on. If $f(v'_{\infty}), f(v'_{-\infty})$ have several connecting bi-infinite geodesics, the choice of $\Phi(v')$ is no longer unique and we obtain a non-trivial equivalence class.}
\label{fig:DefiningPhi} 
\end{figure}

\begin{lem}[cf. \cite{Biswas15}] \label{lem:DefiningPhi}
There exists a unique $y \in \gamma$, such that $\frac{Ê\partial f_* \rho_x}{\partial \rho_y}(f(v_{\infty})) = 1$.

Furthermore, if $\gamma'$ is another geodesic from $f(v_{-\infty})$ to $f(v_{\infty})$ and $y'$ the unique point on $\gamma'$ such that $\frac{Ê\partial f_* \rho_x}{\partial \rho_y'}(f(v_{\infty})) = 1$, then $\overrightarrow{yf(v_{\infty})} \sim \overrightarrow{y'f(v_{\infty})}$.

Finally, if $v \sim v'$ and $\pi(v') =: x'$, then for all $y \in Y$, $\frac{Ê\partial f_* \rho_x}{\partial \rho_y}(f(v_{\infty})) =\frac{Ê\partial f_* \rho_{x'}}{\partial \rho_y}(f(v_{\infty}))$.
\end{lem}

We define $\Phi([v])$ to be the equivalence class of the unit vector at this unique point $y$ that points to $f(v_{\infty})$ (see Figure \ref{fig:DefiningPhi}). By Lemma \ref{lem:DefiningPhi}, $\Phi$ is well-defined. Whenever we use an equivalence class $[v]$ as an input for $\Phi$, we simply write $\Phi(v)$.

\begin{proof}
Let $\gamma$ be a bi-infinite geodesic from $f(v_{-\infty})$ to $f(v_{\infty})$ and $y$, $y' \in \gamma$. By the Chain Rule for metric derivatives, we have

\begin{equation*}
\begin{split}
\frac{\partial f_* \rho_x}{\partial \rho_{y'}}(f(v_{\infty})) & = \frac{\partial f_* \rho_x}{\partial \rho_y}(f(v_{\infty})) \frac{\partial \rho_y}{\partial \rho_{y'}}(f(v_{\infty}))\\
& = \frac{\partial f_* \rho_x}{\partial \rho_y}(f(v_{\infty})) e^{\frac{1}{2} B(y,y', f(v_{\infty}))}.
 \end{split}
 \end{equation*}

Since for any geodesic representative $\gamma$ of $\xi$, $B(\gamma(t), \gamma(t'), \xi) = t-t'$, and since the metric derivative is always a positive number by the way it is introduced in Definition \ref{defprop:Derivative}, any point $y' \in \gamma$ provides us with a unique number $t - t'$ and a unique point $y \in \gamma$ such that $\frac{\partial f_* \rho_x}{\partial \rho_{y}}(f(v_{\infty})) = 1$. This implies existence and uniqueness.\\

For the second statement, let $\gamma'$ be another geodesic from $f(v_{-\infty})$ to $f(v_{\infty})$ and $y'$ the unique point on $\gamma'$ such that $\frac{Ê\partial f_* \rho_x}{\partial \rho_{y'}}(f(v_{\infty})) = 1$. Using the Chain Rule and Remark \ref{rem:DerivativesonBoundaries}, we get
\begin{equation*}
\begin{split}
B(y', y, f(v_{\infty})) & = \ln \left( \frac{Ê\partial \rho_{y'}}{\partial \rho_{y}}(f(v_{\infty})) \right)\\
& = \ln \left( \frac{Ê\partial \rho_{y'}}{\partial f_*\rho_x}(f(v_{\infty})) \cdot \frac{Ê\partial f_* \rho_x}{\partial \rho_{y}}(f(v_{\infty})) \right)\\
& = \ln(1) = 0.
\end{split}
\end{equation*}

Since $\gamma$ and $\gamma'$ have the same endpoints, it follows that the unit vectors at $y$ and $y'$ respectively, pointing at $f(v_{\infty})$ are equivalent.\\

To prove the last statement, let $v \sim v'$, $\pi(v') =: x'$ and $y \in Y$. For the same reasons as above, we have
\begin{equation*}
\begin{split}
\frac{Ê\partial f_* \rho_x}{\partial \rho_y}(f(v_{\infty})) & = \frac{ \partial \rho_x }{ \partial \rho_{x'} }(v_{\infty}) \cdot \frac{Ê\partial f_* \rho_{x'}}{\partial \rho_y}(f(v_{\infty}))\\
& = e^{B(x,x',v_{\infty})} \frac{Ê\partial f_* \rho_{x'}}{\partial \rho_y}(f(v_{\infty}))\\
& =  \frac{Ê\partial f_* \rho_{x'}}{\partial \rho_y}(f(v_{\infty})).
\end{split}
\end{equation*}

\end{proof}

\begin{lem} \label{lem:oppositevectors}
For all $v \in T_x X$, we have $\Phi(-v) = -\Phi(v)$.
\end{lem}

\begin{proof}
By the Geometric mean value theorem, for all $y$ on a geodesic $\gamma$ from $f(v_{-\infty})$ to $f(v_{\infty})$,

\begin{equation*}
\begin{split}
\frac{ \partial f_* \rho_x}{\rho_y}(f(v_{\infty})) \frac{\partial f_* \rho_x}{\rho_y}(f(v_{-\infty})) = \frac{\rho_y(f(v_{\infty}), f(v_{-\infty}))^2}{\rho_x(v_{\infty}, v_{-\infty})^2} = \frac{1}{1} = 1.
\end{split}
\end{equation*}

This implies that $\frac{\partial f_* \rho_x}{\rho_y}(f(v_{\infty})) = 1$ if and only if $\frac{\partial f_* \rho_x}{\rho_y}(f(v_{-\infty})) = 1$.
\end{proof}

Throughout the following, we will want to consider Busemann functions that are evaluated on a point in the set $\pi([v])$. We denote $\pi \circ \Phi(v)$ to be the foot point of a chosen representative of $\Phi(v)$.

\begin{lem}[cf. \cite{Biswas15}] \label{lem:GeodesicConjugacy}
For all $x, x' \in X, \xi \in \partial X$,
\[ B(\pi \circ \Phi(\overrightarrow{x\xi}), \pi \circ \Phi(\overrightarrow{x'\xi}), f(\xi)) = B(x, x', \xi). \]

By Lemma \ref{lem:DefiningPhi}, the left-hand-side does not depend on the choice of representative and is thus well-defined.

\end{lem}

\begin{proof}
\begin{equation*}
\begin{split}
B(\pi \circ \Phi(\overrightarrow{x\xi}), \pi \circ \Phi(\overrightarrow{x'\xi}), f(\xi)) & = \ln \left( \frac{ \partial \rho_{\pi \circ \Phi(\overrightarrow{x\xi})}}{\partial \rho_{\pi \circ \Phi(\overrightarrow{x'\xi})}}(f(\xi)) \right)\\
& = \ln \left( \frac{ \partial \rho_{\pi \circ \Phi(\overrightarrow{x\xi})} }{ \partial f_*\rho_{x} }(f(\xi)) \cdot \frac{ \partial f_* \rho_x}{\partial f_* \rho_{x'}}(f(\xi)) \cdot \frac{ \partial f_* \rho_{x'} }{ \partial \rho_{ \pi \circ \Phi(\overrightarrow{x'\xi})} } (f(\xi))  \right)\\
& = \ln \left( \frac{ \partial \rho_{x} }{ \partial \rho_{x'} }(\xi) \right)\\
& = B(x, x', \xi).
\end{split}
\end{equation*}
\end{proof}

The map $\Phi$ is natural in the following sense.

\begin{lem} \label{lem:Functoriality}
Given two M\"obius bijections $f : \partial X \rightarrow \partial Y, g : \partial Y \rightarrow \partial Z$ that are homeomorphisms and preserve visible pairs, we have
\[ \Phi_g \circ \Phi_f = \Phi_{g \circ f}. \]
Furthermore, $\Phi_{Id} = Id$.
\end{lem}

\begin{proof}
Let $u \in T^1X$ with $\pi(u) = x$. Choose $v \in \Phi_f(u), w \in \Phi_g(v), w' \in \Phi_{g \circ f}(u)$ and denote $y := \pi(v), z := \pi(w), z' := \pi(w')$. By construction of $\Phi$, $w'_{\infty} = g(f(u_{\infty})) = w_{\infty}, w'_{-\infty} = g(f(u_{-\infty})) = w_{-\infty}$ and $\frac{ \partial (g \circ f)_* \rho_x}{ \partial \rho_z}(w'_{\infty}) = 1$. By the Chain Rule,
\begin{equation*}
\begin{split}
\frac{ \partial g_* \rho_y }{ \partial \rho_z}(w_{\infty}) \cdot \frac{ \partial f_* \rho_x }{ \partial \rho_y}(v_{\infty}) & = \frac{\partial g_* f_* \rho_x}{Ê\partial \rho_z}(w_{\infty}) = 1
\end{split}
\end{equation*}

and therefore, $w$ is in the equivalence class of $\Phi_{g \circ f}(v)$. The identity $\Phi_{Id} = Id$ is immediate.
\end{proof}

\begin{rem}
We did not say that $\Phi$ is part of a functor because there are open questions regarding a potential category of boundaries to use. Specifically, for an object to be part of a `boundary category' on which the construction above makes sense, this object needs to admit a `filling' by a Hadamard manifold. This is sometimes called the inverse problem for M\"obius geometry. The only case the author is aware of, where the inverse problem is solved, is the case when the boundary is a circle (see \cite{Buyalo19}).
\end{rem}

Lemma \ref{lem:Functoriality} implies in particular that $\Phi$ is invertible and its inverse is the map induced by $f^{-1}$. In \cite{Biswas15}, Biswas shows that, if $X$ and $Y$ are both $\mathrm{CAT(-1)}$ spaces, the map $\Phi$ is a homeomorphism. Since $X, Y$ can contain flat strips under our assumptions, his proof does not generalize directly. We will present a way around this in the next section. Nevertheless, we raise the following

\begin{quest*}
Is the map $\Phi : \overline{T^1X} \rightarrow \overline{T^1Y}$ a homeomorphism?
\end{quest*}

%---------------------------------------------------------------------------------------------------------------------------------------
%DEFINING F
%---------------------------------------------------------------------------------------------------------------------------------------

\subsection{Constructing $F$} \label{subsec:F}

Let $x \in X$. Consider the unit-tangent sphere $T_x^1 X$ at $x$. Every point $\xi \in \partial X$ can be represented by a unit vector $\overrightarrow{x\xi} \in T_x^1 X$. Applying the map $\Phi$ to all $\overrightarrow{x\xi}$, we obtain a collection of equivalence classes in $\overline{T^1Y}$. Note that we may not be able to choose representatives of these equivalence classes, such that all representatives share the same foot point. We want $F(x) \in Y$ to be 'in the middle' of the family $\Phi(T_x^1 X)$. For all $x \in X$, $y \in Y$, $\xi \in \partial X$, we define

\[ u_{x,y}(\xi) := B(\pi \circ \Phi( \overrightarrow{x\xi}), y, f(\xi)). \]

By Lemma \ref{lem:DefiningPhi}, the expression above is independent of the choice of $\pi \circ \Phi(\overrightarrow{x\xi})$. We start by showing important properties of $u_{x,y}(\xi)$.

\begin{lem}[cf. \cite{Biswas15}] \label{lem:derivativeBusemann}
For all $x \in X$, $y \in Y$, $\xi \in \partial X$,
\[ \frac{\partial f_* \rho_x}{\partial \rho_y}(f(\xi)) = e^{u_{x,y}(\xi)}. \]
\end{lem}

\begin{proof}
We have
\begin{equation*}
\begin{split}
\frac{ \partial f_* \rho_x}{\partial \rho_y}(f(\xi)) & = \frac{ \partial f_* \rho_x}{\partial \rho_{\pi \circ \Phi(\overrightarrow{x\xi})}}(f(\xi)) \frac{ \partial \rho_{\pi \circ \Phi(\overrightarrow{x\xi})}}{\partial \rho_y}(f(\xi))\\
& = 1 \cdot e^{B(\pi \circ \Phi(\overrightarrow{x\xi}), y, f(\xi))}\\
& = e^{u_{x,y}(\xi)},
\end{split}
\end{equation*}
where we used the definition of $\Phi$ and Remark \ref{rem:DerivativesonBoundaries} in the second step.
\end{proof}

\begin{lem} \label{lem:Continuityu}
The map $u_{x,y}(\xi)$ is continuous in $x,y$ and $\xi$.
\end{lem}

\begin{proof}
Clearly, $u$ is continuous in $y$. To show continuity in $x$, denote the extension of $f^{-1}$ by $\Psi : \overline{T^1X} \rightarrow \overline{T^1Y}$. By Lemma \ref{lem:Functoriality}, $\Psi = \Phi^{-1}$. By Lemma \ref{lem:GeodesicConjugacy}, we have
\[ u_{x,y}(\xi) = B(\pi \circ \Phi(\overrightarrow{x\xi}), \pi \circ \Phi(\Psi(\overrightarrow{yf(\xi)})), f(\xi)) = B(x, \pi \circ \Psi(\overrightarrow{yf(\xi)}), \xi), \]
which is continuous in $x$.

To prove continuity in $\xi$, we note that this is equivalent to continuity of $\frac{ \partial f_*\rho_x }{\partial \rho_y}(f(\xi))$ in $\xi$ by Lemma \ref{lem:derivativeBusemann}. Since $f$ is continuous by assumption, we are left to prove continuity of specific metric derivatives. By definition,
\[ \frac{ \partial f_*\rho_x }{\partial \rho_y}(f(\xi)) = \frac{ \rho_x(\xi, \eta) \rho_x(\xi, \zeta) \rho_y(f(\eta), f(\zeta))}{ \rho_x(\eta, \zeta) \rho_y(f(\xi), f(\eta)) \rho_y(f(\xi), f(\zeta)) } \]

for any $\eta, \zeta \in \partial X$ such that $(\xi, \eta, \zeta)$ is an algebraically visible triple. Since every point in $\partial X$ is in a rank 1 hinge, we can additionally choose $\eta$ and $\zeta$, such that $(\xi, \eta)$ and $(\xi, \zeta)$ are connected by a rank 1 geodesic. Since $f$ preserves visible pairs, Lemma \ref{lem:rank1characterisation}, implies that $(f(\xi), f(\eta))$ and $(f(\xi), f(\zeta))$ can be connected by a rank 1 geodesic. Corollary \ref{cor:ContinuityGromovProduct} then implies that the expression above is continuous in $\xi$. This proves continuity of $u$ in $\xi$.
\end{proof}

Since $\partial X$ is compact, continuity implies that the supremum-norm $\Vert u_{x,y} \Vert_{\infty} < \infty$. By Lemma \ref{lem:Convexity}, we know that $\Vert u_{x, y} \Vert_{\infty}$ is convex in $y$. Furthermore, the function $y \mapsto \Vert u_{x,y} \Vert_{\infty}$ is proper, since for any diverging sequence $y_n$, we have
\[ \sup_{\xi \in \partial X} \{ u_{x,y_n}(\xi) \} = \sup _{\xi \in \partial X} \{ u_{x,y_0}(\xi) + B(y_0, y_n, f(\xi)) \} \xrightarrow{n \rightarrow \infty} \infty, \]
because for every $n$, we can choose $\xi$ such that $f(\xi)$ is the endpoint of the geodesic from $y_0$ to $y_n$ which yields $\sup_{\xi \in \partial X} \{ B(y_0, y_n, f(\xi)) \} = d(y_0, y_n) \rightarrow \infty$, while $\Vert u_{x,y_0} \Vert_{\infty} < \infty$.

Since $y \mapsto \Vert u_{x,y} \Vert_{\infty}$ is proper and convex, the function
\[ M(x) := \min_{y \in Y}\{ \Vert u_{x,y} \Vert_{\infty} \} \]
is well-defined. In addition, we define
\[ M_x := \{ y \in Y \vert \Vert u_{x,y} \Vert_{\infty} = M(x) \} \]
the set of points where the minimum is obtained. Finally, we define for any $x \in X, y \in Y$
\[ K_{x,y} := \{ \xi \in \partial X \vert u_{x,y}(\xi) = \Vert u_{x,y} \Vert_{\infty} \} \]
the set of points in the boundary where the $u_{x,y}$ obtains its supremum. Analogously, for every $y \in Y$, we obtain sets $M_y \subset X$ and $K_{y,x} \subset \partial Y$ by working with $f^{-1}$ and $\Phi^{-1}$. Lemma \ref{lem:MaxMin} and Lemma \ref{lem:derivativeBusemann} together imply that
\[ \forall y \in M_x: M(x) = \max_{\xi \in \partial X} \{ u_{x,y}(\xi)Ê\} = - \min_{\xi \in \partial X} \{ u_{x,y}(\xi) \}.\]
In particular, we conclude that $K_{x,y}$ is non-empty for all $x \in X, y \in M_x$.

We would like to define $F(x)$ to be the unique point in $M_x$. However, if $Y$ is not a $\mathrm{CAT(-1)}$ space, it is absolutely not clear that $M_x$ consists only of one point. As we will see in a moment, issues arise whenever $\Phi(\overrightarrow{x\xi})$ is an equivalence class that contains more than one vector. We solve this by defining an equivalence relation on $Y$: We define $\sim$ to be the equivalence relation generated by demanding that $y \sim y'$, whenever there exists $x \in X$ such that $y,y' \in M_x$. Denote $\overline{Y} := \faktor{Y}{\sim}$. We define $F : X \rightarrow \overline{Y}$ to be the map that sends $x \mapsto [M_x]$. We call $F$ the {\it circumcenter extension} of $f$.

We can characterise elements of $M_x$ as follows.

\begin{lem}[cf. \cite{Biswas17b}] \label{lem:ConvexHull}
Let $x \in X, y \in Y$. The following are equivalent:
\begin{enumerate}
\item $y \in M_x$
\item For all $w \in T_y^1 Y$, there exists $\xi \in K_{x,y}$ such that $\langle w, \overrightarrow{yf(\xi)} \rangle \leq 0$.
\item The convex hull of the set $\{ \overrightarrow{yf(\xi)} \vert \xi \in K_{x,y} \}$ in $T_y Y$ contains the zero vector.
\end{enumerate}
\end{lem}

\begin{proof}[Proof of Lemma \ref{lem:ConvexHull}]
(1) $\Rightarrow$ (2): Suppose not. Then, we find $x \in X, y \in M_x$ and $w \in T_y^1Y$ such that for all $\xi \in K_{x,y}$, $\langle w, \overrightarrow{yf(\xi)} \rangle > 0$. Let $\gamma$ be the geodesic passing through $y$ at time zero with tangent vector $w$. Since the inner product is continuous and $K_{x,y}$ is compact, we find $\epsilon, \epsilon' > 0$ and a neighbourhood $N$ of $K_{x,y}$, such that for all $\xi' \in N, \langle w, \overrightarrow{yf(\xi')} \rangle > \epsilon$ and for all $\xi' \in X \diagdown N$, $u_{x,y}(\xi') < M(x) - \epsilon'$. Using the fact that the gradient of the map $y \mapsto B(y',y,\eta)$ is equal to $- \overrightarrow{y\eta}$, we obtain for all $\xi' \in N$ and $t$ sufficiently small
\begin{equation*}
\begin{split}
u_{x, \gamma(t)}(\xi') & = B(\pi \circ \Phi(\overrightarrow{x\xi}), \gamma(t), f(\xi))\\
& = B(\pi \circ \Phi(\overrightarrow{x\xi}), y, f(\xi)) + B(y, \gamma(t), f(\xi))\\
& \leq M(x) + t (- \langle \overrightarrow{yf(\xi)}, w \rangle) + o(t)\\
& < M(x)
\end{split}
\end{equation*}

For $\xi' \in \partial X \diagdown N$, we have $u_{x,\gamma(t)}(\xi') = u_{x,y}(\xi') + B(y,\gamma(t), f(\xi')) < M(x) - \epsilon' + t < M(x)$ for $t$ sufficiently small. We conclude that, for $t > 0$ sufficiently small, $\Vert u_{x,\gamma(t)} \Vert_{\infty} < M(x)$, which contradicts the definition $M(x) = \inf_{y \in Y} \Vert u_{x,y} \Vert_{\infty}$. Therefore, such a vector $w$ cannot exist.\\

(2) $\Rightarrow$ (3): Suppose not. Then, there exists an affine hyperplane $h \subset T_y Y$ separating the zero vector from the convex hull $C$ of $\{ \overrightarrow{yf(\xi)} \vert \xi \in K_{x,y} \}$. Let $n$ be the unit normal vector of the hyperplane parallel to $h$, going through zero, pointing towards $h$. Then, $\langle n, w \rangle > 0$ for all $w \in C$. This is a contradiction to (2), hence $h$ cannot exist. This implies (3).\\

(3) $\Rightarrow$ (1): Suppose not. Then, there exists $y' \in Y$ such that $\Vert u_{x,y'} \Vert_{\infty} < \Vert u_{x,y} \Vert_{\infty}$. Let $\gamma$ be the geodesic from $y$ to $y'$ and let $\xi \in K_{x,y}$. Then,
\[ B(\pi \circ \Phi(\overrightarrow{x\xi}), y, f(\xi)) = u_{x,y}(\xi) > \Vert u_{x,y'} \Vert_{\infty} \geq u_{x,y'}(\xi) = B(\pi \circ \Phi(\overrightarrow{x\xi}),y',f(\xi)). \]

Since $B(z,y,f(\xi))$ is convex in $y$, we conclude that $B(\pi \circ \Phi(\overrightarrow{x\xi}), \gamma(t), f(\xi))$ is strictly decreasing for $t \geq 0$ sufficiently small. Therefore, for all $\xi \in K_{x,y}$,
\[ 0 > \frac{d}{dt}\vert_{t = 0} B(\pi \circ \Phi(\overrightarrow{x\xi}), \gamma(t), f(\xi)) = - \langle \gamma'(0), \overrightarrow{yf(\xi)} \rangle. \]

If there were points $\xi_1, \dots, \xi_k \in K_{x,y}$ and a convex combination such that
\[ \sum_{i=1}^k \alpha_i \overrightarrow{yf(\xi_i)} = 0, \]
then we compute
\[ 0 = \langle \gamma'(0), \sum_{i=1}^k \alpha_i \overrightarrow{yf(\xi_i)} \rangle > 0. \]
This is a contradiction to (3). We conclude that (3) implies (1), which completes the proof.
\end{proof}

The following is an important property of the function $M$.

\begin{lem}[cf. \cite{Biswas18a}] \label{lem:LipschitzContinuityM}
The map $M : X \rightarrow \mathbb{R}$ is $1$-Lipschitz continuous. Furthermore, the maps $x \mapsto \Vert u_{x,y} \Vert_{\infty}$ for fixed $y$ and $y \mapsto \Vert u_{x,y} \Vert_{\infty}$ for fixed $x$ are both $1$-Lipschitz.
\end{lem}

\begin{proof}[Proof of Lemma \ref{lem:LipschitzContinuityM}]
Let $x, x' \in X$, $y \in M_x$, $y' \in M_{x'}$, $\xi \in K_{x',y}$. Using Lemma \ref{lem:GeodesicConjugacy}, we compute
\begin{equation*}
\begin{split}
M(x') = \Vert u_{x',y'} \Vert_{\infty} & \leq \Vert u_{x',y} \Vert_{\infty}\\
& = B(\pi \circ \Phi(\overrightarrow{x'\xi}), y, f(\xi))\\
& = B(x',x,\xi) + B(\pi \circ \Phi(\overrightarrow{x\xi}),y,f(\xi))\\
& \leq d(x,x') + \Vert u_{x,y} \Vert_{\infty}.
\end{split}
\end{equation*}
We conclude that $M(x') \leq d(x,x') + M(x)$. Since the argument is symmetric in $x, x'$, we conclude that $M$ is $1$-Lipschitz continuous. This estimate also proves the $1$-Lipschitz continuity of the map $x \mapsto \Vert u_{x,y} \Vert_{\infty}$. For the last map, the proof is analogous with $\xi \in K_{x,y'}$.
\end{proof}

It turns out that $F$ has several nice properties.

\begin{lem} \label{lem:ContinuityF}
The map $F$ is continuous with respect to the quotient topology.
\end{lem}

\begin{proof}[Proof of Lemma \ref{lem:ContinuityF}]
Since $X$ is first countable, it is enough to show that $F$ is sequentially continuous. Let $y \in Y$ such that $[y] \in \Ima(F)$. Let $[y] \in \overline{U} \subset \overline{Y}$ open, i.e. $[y] \subset U := P^{-1}(\overline{U})$, $U$ open, where $P$ denotes the projection $Y \rightarrow \overline{Y}$. Let $x \in X$ such that $y \in M_x$ and thus, $F(x) = [y]$. We have $\Vert u_{x,y} \Vert_{\infty} = M(x)$. Let $x_n \rightarrow x$ and let $y_n \in M_{x_n}$.

We first show that $(y_n)_n$ is bounded. Suppose not. Since the map $y \mapsto \Vert u_{x,y} \Vert_{\infty}$ is proper, we conclude that there is a subsequence, also denoted $(y_n)_n$ such that $\Vert u_{x, y_n} \Vert_{\infty} \rightarrow \infty$. On the other hand, since $M$ is $1$-Lipschitz, $\Vert u_{x_n,y_n} \Vert_{\infty} = M(x_n) \rightarrow M(x) = \Vert u_{x,y} \Vert_{\infty}$. In addition, since $x \mapsto \Vert u_{x,y'} \Vert_{\infty}$ is $1$-Lipschitz for all $y'$, we conclude that
\[ M(x) \geq M(x_n) - d(x, x_n) \geq \Vert u_{x, y_n} \Vert_{\infty} - 2d(x,x_n) \rightarrow \infty. \]
This is a contradiction, hence $(y_n)_n$ is bounded.

By properness of $Y$, any subsequence of $(y_n)_n$ has a converging subsequence $(y_{n_i})_i$ that converges to some $y' \in Y$. We claim that $y' \sim y$. Since $\Vert u_{x_n, y_n} \Vert_{\infty}$ is $1$-Lipschitz continuous in both variables, we have
\[ \Vert u_{x,y} \Vert_{\infty} = M(x) \xleftarrow{n \rightarrow \infty} M(x_n) = \Vert u_{x_n, y_n} \Vert_{\infty} \xrightarrow{n \rightarrow \infty} \Vert u_{x,y'} \Vert_{\infty}. Ê\]

Therefore, $y' \in M_x$. Suppose now, that $(y_n)_n$ admits a subsequence $(y_{n_i})_n$ such that for all $n_i, y_{n_i} \notin U$. No subsequence of $(y_{n_i})_i$ can converge to an element in $M_x$, a contradiction to our argument above. Therefore, for all large $n$, $[y_n] \in \overline{U}$ and $[y_n] \rightarrow [y]$. We conclude that $F$ is sequentially continuous.
\end{proof}

We are now ready to prove that the equivalence relation on $Y$ affects only specific parts of $Y$.

\begin{prop} \label{prop:NowhereDensity}
The union $E_Y := \bigcup_{x \in X : \vert M_x \vert \geq 2} M_x$ satisfies $E_Y \subset F_Y$. In particular, the projection $P : Y \rightarrow \overline{Y}$ is a homeomorphism on $Y \diagdown F_Y$.
\end{prop}

In order to prove this, we need to do some preparation which will be of further use in later sections.

\begin{mydef} \label{def:AntipodalMaps}
Let $x \in X, \xi \in \partial X$. We define $a_x : \partial X \rightarrow \partial X$ to be the map that sends $\xi \in \partial X$ to the forward endpoint of the geodesic ray induced by the vector $-\overrightarrow{x\xi}$. We call $a_x$ the {\it antipodal map with respect to $x$}.
\end{mydef}

By definition, $a_x = \exp_x \circ (-Id) \circ \exp_x^{-1}$. Since the visual topology coincides with the standard topology on the unit tangent sphere $T_x^1X$, we immediately see that $a_x$ is a homeomorphism.

\begin{prop}[cf. \cite{Biswas18a}] \label{prop:ExtremalDirections}
Fix $x \in X$ and $y \in M_x$. Let $\xi \in \partial X$. If $\frac{\partial f_* \rho_x}{\partial \rho_y}(f(\xi))$ is minimal among all $\xi$, then there exists a bi-infinite geodesic $\gamma \in [f(a_x(\xi)), f(\xi)]$ such that $y$ lies on $\gamma$. In particular, $f(a_x(\xi)) = a_y(f(\xi))$.
\end{prop}

\begin{proof}
A point $y$ lies on a geodesic from $f(a_x(\xi))$ to $f(\xi)$ if and only if $\rho_y(f(a_x(\xi)), f(\xi)) = 1$. Combining Lemma \ref{lem:MaxMin} and Lemma \ref{lem:derivativeBusemann}, we know that the minimal value obtained by $\frac{ \partial f_* \rho_x}{\partial \rho_x}(f(\xi))$ is equal to $e^{-M(x)}$. We compute
\begin{equation*}
\begin{split}
\rho_y(f(a_x(\xi)) f(\xi))^2 & = \frac{\partial \rho_y}{\partial f_* \rho_x}(f(a_x(\xi))) \frac{\partial \rho_y}{\partial f_* \rho_x}(f(\xi)) f_* \rho_x(f(a_x(\xi)), f(\xi))^2\\
& = \frac{\partial \rho_y}{\partial f_* \rho_x}(f(a_x(\xi))) e^{M(x)} 1^2\\
& \geq e^{-M(x)} e^{M(x)} = 1,
\end{split}
\end{equation*}
where we used the fact that $x \in (a_x(\xi), \xi)$ by construction. This concludes the proof.
\end{proof}

\begin{cor}[cf. \cite{Biswas18a}] \label{cor:AntipodalPoints}
If $x \in X, y \in M_x, \xi \in \partial X$, then $\frac{ \partial f_* \rho_x}{\partial \rho_y}(f(\xi))$ is maximal if and only if $\frac{ \partial f_* \rho_x}{\partial \rho_y}(f(a_x(\xi)))$ is minimal.
\end{cor}

\begin{proof}
If $\frac{ \partial f_* \rho_x}{\partial \rho_y}(f(a_x(\xi)))$ is minimal, then $y$ lies on a geodesic from $f(a_x(\xi))$ to $f(\xi)$. Then, $\frac{ \partial f_* \rho_x}{\partial \rho_y}(f(\xi)) = \frac{ \partial f_* \rho_x}{\partial \rho_y}(f(a_x(\xi)))^{-1} = e^{M(x)}$ by the Geometric mean value theorem.

On the other hand, if $\frac{ \partial f_* \rho_x}{\partial \rho_y}(f(\xi))$ is maximal, then $\frac{ \partial f^{-1}_* \rho_y}{\partial \rho_x}(\xi)$ is minimal by the Chain rule. By Proposition \ref{prop:ExtremalDirections}, this implies that $a_x(\xi) = f^{-1}(a_y(f(\xi)))$. The argument above implies that $\frac{ \partial f^{-1}_* \rho_y}{\partial \rho_x}(a_x(\xi))$ is maximal and therefore, $\frac{ \partial f_* \rho_x}{\partial \rho_y}(f(a_x(\xi)))$ is minimal.
\end{proof}

\begin{cor} \label{cor:AtleastthreeExtremalDirections}
For all $x \in X, y \in M_x$, the set $K_{x,y}$ contains at least three points.
\end{cor}

\begin{proof}
By Lemma \ref{lem:ConvexHull}, $K_{x,y}$ contains at least two points, as any non-trivial convex combination requires at least two vectors. Suppose it consisted of exactly two points $\xi, \eta$. Then $0 = \alpha_1 \overrightarrow{yf(\xi)} + \alpha_2 \overrightarrow{yf(\eta)}$ for $\alpha_1, \alpha_2 > 0$. Since this is a sum of unit vectors, we conclude that $a_y(f(\xi)) = f(\eta)$. As $\xi \in K_{x,y}$, Corollary \ref{cor:AntipodalPoints} implies that $u_{x,y}(a_x(\xi))$ is minimal. By Proposition \ref{prop:ExtremalDirections}, $f(a_x(\xi)) = a_y(f(\xi)) = f(\eta)$. Thus, $u_{x,y}(\eta)$ is both maximal and minimal. Since $\min_{\xi \in \partial X} \{ u_{x,y}(\xi) \} = -\max_{\xi \in \partial X} \{ u_{x,y}(\xi) \}$, we obtain that $u_{x,y}(\eta) = 0$ and $u_{x,y} \equiv 0$. Therefore, $K_{x,y} = \partial X$, which contains infinitely many points.
\end{proof}

The following result provides us with more information about $M_x$, which may be of general interest in further study of this construction.

\begin{lem} \label{lem:IntersectionofHorospheres}
Let $x \in X$. The set $M_x$ is convex and contained in an intersection of at least three horospheres in $Y$. Furthermore, $\diam(M_x) \leq 2M(x) < \infty$. In particular, $M_x$ is compact and has codimension at least two in $Y$.
\end{lem}

\begin{proof}[Proof of Lemma \ref{lem:IntersectionofHorospheres}]
If $M_x$ consists of exactly one point, this is trivial. Suppose, $M_x$ contains at least two points. Let $y \neq y' \in M_x$ and denote the geodesic from $y$ to $y'$ by $\gamma$. Since $\Vert u_{x,\gamma} \Vert_{\infty}$ is convex, greater or equal to $M(x)$ and equal to $M(x)$ at both endpoints, we conclude that $\Vert u_{x, \gamma} \Vert_{\infty} \equiv M(x)$. Therefore, $M_x$ is convex.

Let $p$ be any point on $\gamma$ strictly between $y$ and $y'$. Since $u_{x,p}$ is continuous, we find at least one $\xi \in K_{x,p}$. Since $u_{x,\gamma}(\xi)$ is convex (for any $\xi \in \partial X$), we obtain that it is either constant or increasing in one direction. If it was increasing, then $\Vert u_{x,\gamma}Ê\Vert_{\infty} > \vert u_{x,p}(\xi) \vert = M(x)$ for some $\gamma(t)$ near, but not equal, to $p$. This contradicts the fact that $\Vert u_{x, \gamma} \Vert_{\infty} \equiv M(x)$. Therefore, $u_{x, \gamma}(\xi) \equiv M(x)$ along $\gamma$.\\

By Proposition \ref{prop:ExtremalDirections}, we conclude that every $y \in M_x$ lies on a geodesic from $f(a_x(\xi))$ to $f(\xi)$. Therefore, $\gamma$ is contained in a horosphere centered at $f(\xi)$ intersected with a flat strip from $f(a_x(\xi))$ to $f(\xi)$. Furthermore, we see that for every point $p$ on $\gamma$ that is not an end point, any element $\xi \in K_{x,p}$ realises the supremum $\Vert u_{x,\gamma(t)} \Vert_{\infty}$ at every point on the geodesic $\gamma$. In particular, if we extend $\gamma$ to its maximal length such that it is still contained in $M_x$, the points in $\partial X$ that obtain $u_{x,\gamma}(\xi) = M(x)$ are the same along the entire geodesic, except for some extremal points that appear only at the endpoints of the extended geodesic.

Choose $y_0$ on $\gamma$ not an endpoint. By Lemma \ref{lem:ConvexHull} and Corollary \ref{cor:AtleastthreeExtremalDirections}, there exist $k \geq 3, \xi_1, \dots, \xi_k \in K_{x,y_0}$ and $\alpha_1, \dots, \alpha_k > 0$ such that $\sum_{i=1}^k \alpha_i = 1$ and $\sum_{i=1}^k \alpha_i \overrightarrow{yf(\xi_i)} = 0$. In particular, $\gamma$ is contained in the intersection of horospheres centered at $f(\xi_1), \dots f(\xi_k)$. Suppose, $M_x$ is not contained in the intersection of these horospheres. Then we find $y \in M_x$ such that $B(y, y_0, \xi_i) \neq 0$ for some $i$. Without loss of generality, $B(y, y_0, \xi_1) \neq 0$. Since $M(x) \geq u_{x,y}(\xi_i) = u_{x,y_0}(\xi_i) + B(y_0, y, \xi_i) = M(x) + B(y_0, y, \xi_i)$, we conclude that $B(y_0, y, \xi_i) \leq 0$ for all $i$ and $B(y_0, y, \xi_1) < 0$. Let $\delta$ be the geodesic from $y_0$ to $y$. By convexity, $B(y_0, \delta(t), \xi_1)$ is decreasing for small, positive $t$. Therefore,
\[ 0 > \frac{d}{dt}\vert_{t=0} B(y_0, \delta(t), f(\xi_1)) = \langle \delta'(0), \overrightarrow{yf(\xi_1)} \rangle. \]

On the other hand,
\[ 0 = \langle \delta'(0), \sum_{i = 1}^k \alpha_i \overrightarrow{yf(\xi_i)} \rangle = \alpha_1 \langle \delta'(0), \overrightarrow{yf(\xi_1)} \rangle + \sum_{i=2}^k \alpha_i \langle \delta'(0), \overrightarrow{yf(\xi_i)} \rangle. \]

Since $\alpha_i > 0$ for all $k$, we conclude that
\[ \frac{d}{dt}\vert_{t=0} B(y_0, \delta(t), f(\xi_i)) = \langle \delta'(0), \overrightarrow{yf(\xi_i)} \rangle > 0 \]
for some $i \geq 2$. In particular, $u_{x,\delta(t)}(\xi_i) = u_{x, y_0}(\xi_i) + B(y_0, \delta(t), f(\xi_i)) > u_{x,y_0}(\xi_i) = M(x)$ for $t > 0$ sufficiently small. Since $y, y_0 \in M_x$ and $M_x$ is convex, we have found an element $\delta(t)$ in $M_x$ for which $\Vert u_{x,\delta(t)} \Vert_{\infty}$ is not minimal, a contradiction. Therefore, there can be no point $y$ outside of the intersection of the horospheres centered at the points $f(\xi_1), \dots, f(\xi_k)$. Since, among any three distinct horospheres with non-empty intersection, at least two of them intersect transversely and horospheres have codimension one, we conclude that $M_x$ has codimension at least two. This proves the Lemma except for the bound on the diameter.\\

To estimate the diameter, consider $y, y' \in M_x$, let $\delta$ be the geodesic from $y$ to $y'$ and let $\xi \in \partial X$ such that $f(\xi)$ is the forward end-point of the geodesic ray induced by $\delta$. Then
\[ u_{x,y}(\xi) = u_{x,y'}(\xi) + B(y', y, f(\xi)) = u_{x,y'}(\xi) + d(y,y'). \]

Since $y, y' \in M_x$, the expressions $u_{x,y}(\xi), u_{x,y'}(\xi)$ are both bounded in absolute value by $M(x)$. The equation above shows that, whenever $d(y,y') > 2M(x)$, i.e. the length of $\delta$ greater than $2M(x)$, this bound is violated by at least one of the two terms. We obtain that any two points in $M_x$ are connected by a geodesic of length at most $2M(x)$. This provides the bound on the diameter.
\end{proof}

\begin{proof}[Proof of Proposition \ref{prop:NowhereDensity}]
By the proof of Lemma \ref{lem:IntersectionofHorospheres}, if $M_x$ contains at least two points, any geodesic in $M_x$ is contained within a flat strip. Therefore, $E_Y$ is contained in the union of all flat strips in $Y$. Since every bi-infinite geodesic in a flat strip admits a perpendicular, parallel Jacobi field, every flat strip in $Y$ is contained in $F_Y$. Therefore, $E_Y \subset F_Y$.
\end{proof}

Based on Lemma \ref{lem:IntersectionofHorospheres}, it makes sense to define the set
\[ K_x := \{ \xi \in \partial X \vert \forall y \in M_x : u_{x,y}(\xi) = M(x) \}. \]

The set $K_x$ is non-empty, compact and, by the proof of Lemma \ref{lem:IntersectionofHorospheres}, contains at least three points.\\

In \cite{Biswas17a}, the map $F$ is constructed as the limit of a sequence of circumcenters. There is another geometric interpretation of $M_x$ and $M(x)$, which we present here. Any vector $v \in TX$ defines a horoball in $X$, namely the set

\[ HB(v) := \{ x \in X \vert B(\pi(v), x, v_{\infty}) \leq 0 \}. \]

Consider the horoballs $HB(\Phi(v))$ for all $v \in T_x^1 X$. Define $\Phi(v)^t$ to be the vector obtained by applying the geodesic flow on $Y$ to the vector $\Phi(v)$ (the geodesic flow sends equivalence classes in $T_y Y$ to equivalence classes). Since $\Phi(-v) = -\Phi(v)$, we know that the intersection $\bigcap_{v \in T_x^1 X} HB(\Phi(v))$ is the smallest non-empty intersection in the sense that $\bigcap_{v \in T_x^1 X} HB(\Phi(v)^t) = \emptyset$ for all $t > 0$. If this intersection is empty, there is a minimal $t$, such that $\bigcap_{v \in T_x^1 X} HB(\Phi(v)^{-t})$ is non-empty. This minimal $t$ equals $M(x)$ and the intersection of the horoballs $\Phi(v)^{-t}$ equals $M_x$ (see figure \ref{fig:DefiningF} for the situation where $M_x$ consists of one point).

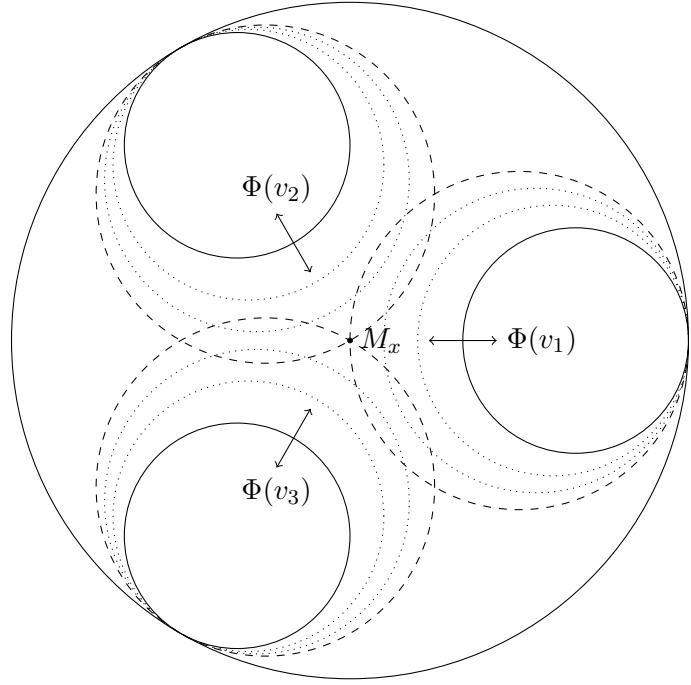
\begin{figure} 
\begin{tikzpicture}[scale=1.50]
\draw [] (0,0) circle [radius=3cm];
\draw [->] (1,0) -- (1.3,0);
\node[right] at (1.3,0) {$\Phi(v_1)$};
\draw [->] (1,0) -- (0.7,0);
\draw [] (2,0) circle [radius=1cm];
\draw [dotted] (1.8,0) circle [radius=1.2cm];
\draw [dotted] (1.65,0) circle [radius=1.35cm];
\draw [dashed] (1.5,0) circle [radius=1.5cm];
\draw [->] (-0.5, 0.866) -- (-0.65, 1.126);
\node[above] at (-0.65,1.126) {$\Phi(v_2)$};
\draw [->] (-0.5, 0.866) -- (-0.35, 0.606);
\draw [] (-1, 1.732) circle [radius = 1cm];
\draw [dotted] (-0.9, 1.559) circle [radius = 1.2cm];
\draw [dotted] (-0.825, 1.429) circle [radius = 1.35cm];
\draw [dashed] (-0.75, 1.299) circle [radius=1.5cm];
\draw [->] (-0.5, -0.866) -- (-0.65, -1.126);
\node[below] at (-0.65,-1.126) {$\Phi(v_3)$};
\draw [->] (-0.5, -0.866) -- (-0.35, -0.606);
\draw [] (-1, -1.732) circle [radius = 1cm];
\draw [dotted] (-0.9, -1.559) circle [radius = 1.2cm];
\draw [dotted] (-0.825, -1.429) circle [radius = 1.35cm];
\draw [dashed] (-0.75, -1.299) circle [radius=1.5cm];
\draw [fill] (0,0) circle [radius=0.02cm];
\node [right] at (0,0) {$M_x$};
\end{tikzpicture}
\caption{If we flow the images $\Phi(v_i)$ backwards in $Y$, we increase the drawn horoballs until they all intersect (which happens for the first time in the case of the dashed horoballs). The candidates for $F(x)$ are all the points in the mutual intersection of the dashed horoballs when going over all $v \in T_x^1 X$.}
\label{fig:DefiningF} 
\end{figure}

%---------------------------------------------------------------------------------------------------------------------------------------

%HLDER AND LIPSCHITZ CONTINUITY OF F

%---------------------------------------------------------------------------------------------------------------------------------------

\section{H\"older and Lipschitz continuity of $F$} \label{sec:HolderLipschitz}

Recall that, in section \ref{subsec:JacobiFields}, we defined $F_X$ to be the set of all points in $X$ that admit a geodesic ray $\gamma$ starting at $x$ and a perpendicular, parallel Jacobi field along $\gamma$. The goal of this section is to prove that $F$ is locally H\"older continuous on $F^{-1}(Y \diagdown F_Y)$ and to provide a sufficient condition for $F$ to be locally Lipschitz continuous. To do so, we will use geometric properties arising from bounds on the second derivative of the Busemann function.

We first introduce some notation. Given a function $g : Y \rightarrow \mathbb{R}$ that is twice continuously differentiable, we can consider its Hessian, i.e.\,the bilinear form induced by its second differential. Since Busemann functions on $\CAT$ manifolds are twice continuously differentiable, we can consider the Hessian of the Busemann function $y \mapsto B(y',y,\eta)$, which we denote by $H_{y_2}B^{\eta}(y)$. Since a change of $y'$ changes the function $y \mapsto B(y',y,\eta)$ by a constant independent of $y$, $H_{y_2}B^{\eta}(y)$ is independent of $y'$. Since Busemann functions are convex in their second variable, $H_{y_2}B^{\eta}(y)$ is semi-positive definite.

Let $\gamma$ be the geodesic ray from $y$ to $\eta$. Since $\frac{d}{dt}\vert_{t=0} B(y',\gamma(t), \eta) = -1$, we see that $H_{y_2}B^{\eta}(y)(\gamma'(0),w) = 0$ for all $w \in T_y Y$. Therefore, we are interested in the restriction of the Hessian to the orthogonal complement of $\gamma'(0) = \overrightarrow{y\eta}$, which we denote by $\overrightarrow{y\eta}^{\perp}$. Let $wÊ\in T_y Y$, $\eta \in \partial Y$. We write $w^{\perp \eta}$ for the orthogonal projection of $w$ onto $\overrightarrow{y\eta}^{\perp}$.

\begin{lem} \label{lem:Hessianbound}
Let $y_0 \in Y \diagdown F_Y$. Then there exists an open neighbourhood $U \subset Y \diagdown F_Y$ of $y_0$ and a constant $\epsilon > 0$, such that for all $y \in U$, $w \in T_y Y$, we have
\[ H_{y_2}B^{\eta}(w,w) \geq \epsilon \Vert w^{\perp \eta} \Vert^2.Ê\]
\end{lem}

\begin{proof}
Since $y_0 \in Y \diagdown F_Y$, we know that for all $\eta \in \partial Y$ and all $w \in \overrightarrow{y_0 \eta}^{\perp}$ with $\Vert w \Vert = 1$, the unique stable Jacobi field $J_w$ along the geodesic ray $\eta_{y_0}$ satisfies
\[ \frac{ d } {dt}\vert_{t=0} \Vert J_w(t) \Vert^2 < 0. \]

Since $\frac{ d }{dt}\vert_{t=0} \Vert J_w(t) \Vert^2$ depends continuously on $y_0$, $\eta$ and $w$ and since $\partial Y$ and $\overrightarrow{y \eta}^{\perp} \cap T_{y}^1 Y$ are compact for all $y \in Y$, we find some constant $\epsilon > 0$ and an open neighbourhood $U$ of $y_0$, such that $U \subset Y \diagdown F_Y$ and for all $y \in U$, all $\eta \in \partial Y$, all $w \in \overrightarrow{y \eta}^{\perp}$ with $\Vert w \Vert = 1$ and all stable Jacobi fields $J_w$ along the geodesic from $y$ to $\eta$, we have
\begin{equation*} \label{eq:CurvatureEstimate} \frac{ d } {dt}\vert_{t=0} \Vert J_w(t) \Vert^2 \leq -2\epsilon. \end{equation*}

Using Proposition 3.1 in \cite{HiH} (see section \ref{subsec:JacobiFields}), we have for all $y \in U, \eta \in \partial Y, w \in \overrightarrow{y\eta}^{\perp}$ with $\Vert w \Vert = 1$:
\[ H_{y_2} B^{\eta}(y)(w,w) = \langle \nabla_w (-\overrightarrow{y\eta}), w \rangle = -\frac{1}{2} \frac{d}{dt}\vert_{t = 0} \Vert J_w(t) \Vert^2 \geq \epsilon. \]
Since $H_{y_2}B^{\eta}(y)(\overrightarrow{y\eta}, w) = 0$ for all $w \in T_y Y$ and since the Hessian is bilinear, we obtain the estimate stated in the Lemma.
\end{proof}

We need one more piece of notation before stating the results on local H\"older and local Lipschitz continuity. We define the sets
\begin{equation*}
\begin{split}
& \mathcal{D}_X := F^{-1}(Y \diagdown F_Y)\\
& L_X := \{ x \in \mathcal{D}_X \vert \exists U \text{ open neighbourhood of } x, \exists \epsilon > 0 :\\
& \qquad \qquad \qquad \quad \; \forall x' \in U, \forall w \in T_{F(x')}^1 Y, \exists \xi \in K_x : \langle w, \overrightarrow{F(x')f(\xi)} \rangle > \epsilon \}.
\end{split}
\end{equation*}

\begin{rem}
By Lemma \ref{lem:ConvexHull}, any $x \in X$ and any $w \in T_{F(x)}^1 Y$ admits $\xi \in K_x$ such that $\langle w, \overrightarrow{ F(x)f(\xi)} \rangle \geq 0$. However, in dimension three and higher, it is very unclear, if strict inequality can be obtained in general and if it can be obtained uniformly in an open neighbourhood of $x$.
\end{rem}

\begin{prop} \label{prop:Holdercontinuity}
The map $F$ is locally $\frac{1}{2}$-H\"older continuous on $D_X$ and locally Lipschitz continuous on $L_X$.
\end{prop}

The statement on H\"older continuity generalises a result by Biswas in \cite{Biswas17a}. The proof is, however, different.

\begin{proof}
Let $x_0 \in \mathcal{D}_X$. By Lemma \ref{lem:Hessianbound}, we find an open neighbourhood $U$ of $x$ and $\epsilon > 0$, such that for all $x \in U$, the Hessian $H_{y_2}B^{f(\xi)}(F(x))$ is positive definite on the subspace $\overrightarrow{F(x)f(\xi)}^{\perp}$ and its positive Eigenvalues are at least $\epsilon$. Let $x, x' \in U$ and let $f(\xi_0)$ be the point represented by the geodesic ray obtained by extending the geodesic from $F(x')$ to $F(x)$. Note that $\overrightarrow{F(x)F(x')} = - \overrightarrow{F(x)f(\xi_0)}$. There are two cases.\\

Case 1: If $\xi_0 \in K_x$, then
\begin{equation*}
\begin{split}
d(F(x),F(x')) & = B(F(x),F(x'),f(\xi_0))\\
& = B(F(x), \pi \circ \Phi(\overrightarrow{x\xi_0}), f(\xi_0)) + B(x,x',\xi_0) + B(\pi \circ \Phi(\overrightarrow{x'\xi_0}), F(x'),f(\xi_0))\\
& \leq -M(x) + d(x,x') + M(x')\\
& \leq 2d(x,x'),
\end{split}
\end{equation*}
where we use that $M$ is $1$-Lipschitz continuous by Lemma \ref{lem:LipschitzContinuityM}.\\

Case 2: Suppose, $\xi_0 \notin K_x$. By continuity of the Riemannian metric, there exists $\delta > 0$, such that for all $\xi \in K_x$, $\langle -\overrightarrow{F(x)F(x')}, \overrightarrow{F(x)f(\xi)} \rangle \leq 1 - \delta^2$. By Lemma \ref{lem:ConvexHull}, we find $\xi \in K_x$ such that $\langle -\overrightarrow{F(x)F(x')}, \overrightarrow{F(x)f(\xi)} \rangle \geq 0$. In particular, this $\xi$ satisfies $\Vert \overrightarrow{F(x)F(x')}^{\perp f(\xi)} \Vert^2 \geq \delta^2$. By Taylor approximation, we know that
\begin{equation*}
\begin{split}
B(F(x),F(x'),f(\xi)) & = -\left\langle \overrightarrow{F(x)F(x')}, \overrightarrow{F(x)f(\xi)} \right\rangle d(F(x),F(x'))\\
& \quad \, + H_{x_2}B^{f(\xi)}\left(\overrightarrow{F(x)F(x')}, \overrightarrow{F(x)F(x')}\right) d(F(x),F(x'))^2\\
& \quad \, + o\left(d(F(x),F(x'))^2\right).
\end{split}
\end{equation*}
Let $0 < \lambda < 1$. For $d(F(x),F(x'))$ sufficiently small (`sufficiently small' depending on $\lambda$), this implies
\begin{equation*}
\begin{split}
B(F(x),F(x'),f(\xi)) & \geq \lambda \epsilon \left\Vert \overrightarrow{F(x)F(x')}^{\perp f(\xi)} \right\Vert^2 d(F(x),F(x'))^2\\
& \geq \lambda \epsilon \delta^2 d(F(x),F(x'))^2
\end{split}
\end{equation*}

Let $U_{\lambda} \subset U$ be an open neighbourhood of $x_0$, such that for all $x, x' \in U$, $d(F(x),F(x'))$ is sufficiently small in the sense above. On the other hand,
\begin{equation*}
\begin{split}
B(F(x),F(x'),f(\xi)) & = B(F(x), \pi \circ \Phi(\overrightarrow{x\xi}), f(\xi)) + B(x,x',\xi) + B(\pi \circ \Phi(\overrightarrow{x'\xi}), F(x'),f(\xi))\\
& \leq -M(x) + d(x,x') + M(x')\\
& \leq 2d(x,x').
\end{split}
\end{equation*}

We conclude that for all $x, x' \in U_{\lambda} \subset U$,
\[ d(F(x),F(x'))^2 \leq \frac{2}{\lambda \epsilon \delta^2} d(x,x'). \]

Combining both cases, we conclude that $F$ is locally $\frac{1}{2}$-H\"older continuous.\\

The proof of local Lipschitz continuity follows the same line of computation. Let $x_0 \in L_X$. We find an open neighbourhood $U$ of $x_0$ and $\epsilon > 0$, such that for all $x \in U, w \in T_{F(x)}^1 Y$, there is a $\xi \in K_x$ such that $- \langle w, \overrightarrow{F(x)f(\xi)} \rangle > \epsilon$. Additionally, we choose $U$ sufficiently small such that for all $x \in U$, the positive Eigenvalues of the Hessian $H_{y_2}B^{f(\xi)}(F(x))$ are at least $\epsilon' > 0$. Let $x, x' \in U$ and let $f(\xi_0)$ be the point represented by the geodesic ray obtained by extending the geodesic from $F(x')$ to $F(x)$. We have the same cases as before.

Case 1: If $\xi_0 \in K_x$, then $d(F(x),F(x')) \leq 2d(x,x')$ by the same argument as above.

Case 2: If $\xi_0 \notin K_x$, then -- as before -- we find $\delta > 0$, such that for all $\xi \in K_x$, $\langle - \overrightarrow{F(x)F(x')}, \overrightarrow{F(x)f\xi)} \rangle \leq 1 - \delta^2$. By construction of $U$, we find $\xi \in K_x$, such that $\epsilon < \langle - \overrightarrow{F(x)F(x')}, \overrightarrow{F(x)f(\xi)} \rangle \leq 1 - \delta^2$ and $\Vert \overrightarrow{F(x)F(x')}^{\perp f(\xi)} \Vert^2 \geq \delta^2$. Therefore, we have
\begin{equation*}
\begin{split}
B(F(x), F(x'),f(\xi)) & = - \langle \overrightarrow{F(x)F(x')}, \overrightarrow{F(x)f(\xi)} \rangle d(F(x),F(x'))\\
& \quad \; + H_{y_2}B^{f(\xi)}(F(x))\left(\overrightarrow{F(x)F(x')}, \overrightarrow{F(x)F(x')} \right) d(F(x),F(x'))^2\\
&\quad \; + o(d(F(x),F(x'))^2)\\
& \geq \epsilon d(F(x),F(x')) + \epsilon' \delta^2 d(F(x),F(x'))^2 + o(d(F(x),F(x'))^2)
\end{split}
\end{equation*}
For $d(F(x),F(x'))$ sufficiently small, we obtain
\begin{equation*}
\begin{split}
\epsilon d(F(x),F(x')) & \leq B(F(x),F(x'),f(\xi))\\
& \leq -M(x) + d(x,x') + M(x')\\
& \leq 2d(x,x').
\end{split}
\end{equation*}
Let $x_0 \in V \subset U$ with $V$ open such that for all $x, x' \in V$, $d(F(x),F(x'))$ is sufficiently small in the sense of the inequality above. We conclude that, for all $x, x' \in V$, $d(F(x),F(x')) \leq \frac{2}{\epsilon} d(x,x')$. Therefore, $F$ is locally Lipschitz continuous near all $x_0 \in L_X$.

\end{proof}

\begin{cor} \label{cor:differentiabilityofF}
The map $F : L_X \rightarrow Y$ is differentiable almost everywhere, i.e. there exists a Lebesgue zero set in $L_X$, such that $F$ is differentiable outside of this zero set.
\end{cor}

This is an immediate application of Rademacher's theorem, exploiting the fact that manifolds are second countable.

\begin{rem}
It is important to note that it is a-priori not clear, whether $\mathcal{D}_X \neq \emptyset$. One of the most crucial obstacles to proving that $\mathcal{D}_X$ is non-empty is the lack of injectivity results for the map $F$. If $F$ was locally injective, some assumptions about $F_Y$ being small would carry over to $F^{-1}(F_Y)$ -- e.g. $F_Y$ having codimension one. If we additionally understood more about the topology of $\overline{Y}$, even more general conditions about $F_Y$ being small -- e.g. $F_Y$ being nowhere dense -- would translate into statements about $F^{-1}(F_Y)$ being small.

As we will see in the next section, there are results of this type for certain special cases, but at the time of writing, little is known about injectivity in the general case.
\end{rem}

Summarising the last two sections, we have proven the following theorem.

\begin{thm} \label{thm:Extensionresult}
Let $X, Y$ be Hadamard manifolds whose sectional curvatures are bounded from below by $-b^2$ such that $\partial X, \partial Y$ satisfy (4v) and all points in $\partial X$ and $\partial Y$ are in a rank 1 hinge. Let $f : \partial X \rightarrow \partial Y$ be a M\"obius homeomorphism, such that $f$ and $f^{-1}$ send visible pairs to visible pairs. Then, there exists an equivalence relation $\sim$ on $Y$, such that the projection $P : Y \rightarrow \faktor{Y}{\sim}$ restricted to $Y \diagdown F_Y$ is a homeomorphism onto its image and there exists a continuous map $F : X \rightarrow \faktor{Y}{\sim}$, which is locally $\frac{1}{2}$-H\"older continuous on $X \diagdown F^{-1}(F_Y)$.
\end{thm}

%---------------------------------------------------------------------------------------------------------------------------------------

%APPLICATIONS

%---------------------------------------------------------------------------------------------------------------------------------------

\section{Applications} \label{sec:Applications}

We now turn to several special cases, in which we can show additional properties of the map $F$. The proofs below are all based either on getting more out of the continuity proof in section \ref{sec:HolderLipschitz}, or on a better understanding of the function $M$.

%---------------------------------------------------------------------------------------------------------------------------------------
%SURFACES
%---------------------------------------------------------------------------------------------------------------------------------------

\subsection{Surfaces} \label{subsec:surfaces}

The goal of this section is to prove the following result.

\begin{thm} \label{thm:Surfaces}
Let $X, Y$ be $2$-dimensional Hadamard manifolds whose sectional curvature is bounded from below by $-b^2$, such that $\partial X$ and $\partial Y$ satisfy (4v) and all points in $\partial X$ and $\partial Y$ are in a rank 1 hinge. Suppose, $f : \partial X \rightarrow \partial Y$ is a M\"obius homeomorphism such that $f$ and $f^{-1}$ preserve visible pairs. Then the circumcenter extension of $f$ is a homeomorphism $F : X \rightarrow Y$. In addition, it is locally Lipschitz continuous on a dense subset and differentiable almost everywhere.

Furthermore, if $F$ and $M$ are differentiable at $x$ and $K_x$ contains at least five points, then $DF_x : T_x X \rightarrow T_{F(x)} Y$ is an isometry of tangent spaces equipped with their respective Riemannian metric. In particular, if $K_x$ has at least five points for almost every $x$, then $F$ is a metric isometry.
\end{thm}

We start by showing that $F$ is a map between $X$ and $Y$ in this instance. Let $x \in X$. By Lemma \ref{lem:IntersectionofHorospheres}, the set $M_x$ is contained in an intersection of at least three distinct horospheres. Since two horospheres centered at $\xi_1, \xi_2$ can only intersect non-transversally in points that lie on a geodesic line connecting $\xi_1$ with $\xi_2$, we conclude that at least two of these horospheres intersect transversally. Consequently, $\codim(M_x) \geq 2$. Since $Y$ is $2$-dimensional, this implies that $M_x$ has dimension $0$. Since $M_x$ is convex, this implies that $M_x$ is a single point. We conclude that $F : X \rightarrow Y$ is well-defined on all of $X$.

Next, we show that $F$ is invertible.

\begin{prop} \label{prop:InversesDimension2}
Let $F$ denote the circumcenter extension of $f$ and $G$ the circumcenter extension of $f^{-1}$. Then $G = F^{-1}$.
\end{prop}

The proof relies on an elementary result about $2$-dimensional vector spaces. Let $V$ be a $2$-dimensional, real vector space with an inner product. The set of all unit vectors in $V$ with respect to this inner product is homeomorphic to the $1$-dimensional circle. After choosing an orientation on the circle, we can speak of the order of a set of points on the unit-sphere in $V$. We have the following result.

\begin{lem} \label{lem:OrderingDimension2}
Let $V$ be a $2$-dimensional, real vector space with an inner product. Let $v_1, v_2, v_3$ be unit-vectors with respect to this inner product. Then, the following are equivalent:
\begin{enumerate}
\item The zero vector is contained in the convex hull of $\{ v_1, v_2, v_3 \}$.
\item After reordering the indices, the vectors $\{ \pm v_1, \pm v_2, \pm v_3 \}$ are ordered as $(v_1, -v_2, v_3, -v_1, v_2, -v_3)$.
\end{enumerate}
\end{lem}

\begin{proof}[Proof of Proposition \ref{prop:InversesDimension2}]
By Lemma \ref{lem:ConvexHull}, we know that $F(x)$ is characterised as the unique point $y \in Y$ such that $0 \in T_yY$ is contained in the convex hull of the set $\{ \overrightarrow{yf(\xi)} \vert \xi \in K_{x,y} \}$. By Carath\'eodory's theorem on convex hulls and since $Y$ is $2$-dimensional, we know that the zero vector can be expressed by a convex combination of at most three vectors of the form $\overrightarrow{yf(\xi)}$ with $\xi \in K_{x,y}$. Combining this with Corollary \ref{cor:AtleastthreeExtremalDirections}, we find $\xi_1, \xi_2, \xi_3 \in K_x$ such that the convex hull of $\{ \overrightarrow{F(x)f(\xi_i)} \vert i \in \{1,2,3\} \}$ contains $0 \in T_yY$. By Lemma \ref{lem:OrderingDimension2}, this means that, after rearranging the indices, the following six points have the following ordering in $\partial Y$:
\[ ( f(\xi_1), a_{F(x)}(f(\xi_2)), f(\xi_3), a_{F(x)}(f(\xi_1)), f(\xi_2), a_{F(x)}(f(\xi_3)) ). \]

Since $f$ is a homeomorphism, we conclude that, after changing the orientation of $\partial X$ if necessary, we have the following ordering on $\partial X$:
\[ ( \xi_1, a_x(\xi_2), \xi_3, a_x(\xi_1), \xi_2, a_x(\xi_3) ). \]
Using Lemma \ref{lem:OrderingDimension2} again, we conclude that the convex hull of the vectors $\{ -\overrightarrow{x\xi_i} \vert i \in \{1,2,3 \} \}$ contains $0 \in T_xX$. By Lemma \ref{lem:MaxMin}, Corollary \ref{cor:AntipodalPoints} and the Chain rule for metric derivatives, we know that $K_{F(x),x} = a_x(K_{x,F(x)})$, and therefore, the zero vector in $T_xX$ is contained in the convex hull of the set $\{ \overrightarrow{x\xi} \vert \xi \in K_{F(x),x} \}$. By Lemma \ref{lem:ConvexHull}, this implies that $G(F(x)) = x$.

We conclude that $G \circ F = Id_X$. By symmetry, the same argument also proves that $F \circ G = Id_Y$. Therefore, $G = F^{-1}$.
\end{proof}

Next, we show that $F$ is differentiable almost everywhere. We do this by showing that the pointwise Lipschitz constant of $F$ is finite for all $x \in X$.

\begin{prop} \label{prop:locallyLipschitzDimension2}
The map $F$ is locally Lipschitz continuous on a dense subset of $X$. Furthermore, the pointwise Lipschitz constant $Lip_x(F) := \limsup_{x' \rightarrow x} \frac{d(F(x),F(x'))}{d(x,x')}$ is finite for all $x \in X$.
\end{prop}

\begin{proof}
Let $x_0 \in X$. We need to distinguish two cases.\\

Case 1: Suppose, $M(x_0) = 0$. Let $x \in X$ and $\xi \in \partial X$ such that $f(\xi)$ is the endpoint of the geodesic ray obtained by extending the geodesic segment from $F(x)$ to $F(x_0)$. As in the proof of Proposition \ref{prop:Holdercontinuity}, we have
\begin{equation*}
\begin{split}
d(F(x_0),F(x)) & = B(F(x_0),F(x), f(\xi))\\
& \leq M(x_0) + B(x_0, x, \xi)) + M(x)\\
& \leq 2d(x_0,x).
\end{split}
\end{equation*}

If $x_0$ lies in the interior of the set $\{ x \in X \vert M(x) = 0 \}$, then we find an open neighbourhood $U$ of $x_0$, such that the estimate above becomes $d(F(x),F(x')) \leq d(x,x')$ for all $x, x' \in U$.\\

Case 2: Suppose, $M(x_0) > 0$. The proof has three steps.

Step 1: We show that there exists $\epsilon > 0$ and an open neighbourhood $U$ of $x_0$, such that for all $x \in U, \xi \in K_{x}$ and $\xi' \in \partial X$ such that $\langle \overrightarrow{F(x)f(\xi)}, \overrightarrow{F(x)f(\xi')} \rangle \leq -1 + \epsilon$, we have $\xi' \notin K_x$.

Since $M(x) \neq 0$, we know that for all $\xi \in K_x$, $a_x(\xi) \notin K_x$. In fact, $u_{x,F(x)}(a_x(\xi)) = -M(x) < 0$. Suppose, the statement of Step 1 was not true. Then we find a sequence $x_n \rightarrow x_0$ and sequences $\xi_n, \xi'_n \in K_{x_n}$ such that $\langle \overrightarrow{F(x_n)f(\xi_n)}, \overrightarrow{F(x_n)f(\xi'_n)} \rangle < -1 + \frac{1}{n}$. Since $\partial X$ is compact, we can assume without loss of generality that $\xi_n \rightarrow \xi$ and $\xi'_n \rightarrow \xi'$ (choosing subsequences if necessary). Since $u_{x,F(x)}(\xi)$ is Lipschitz continuous in its two index variables and continuous in $\xi$, we have
\[ u_{x_n, F(x_n)}(\xi_n) \rightarrow u_{x,F(x)}(\xi).\]
On the other hand, $\xi_n \in K_{x_n}$ and therefore,
\[ u_{x_n,F(x_n)}(\xi_n) = M(x_n) \rightarrow M(x). \]
We conclude that $\xi, \xi' \in K_x$. However,
\[ \langle \overrightarrow{F(x_n)f(\xi_n)} , \overrightarrow{F(x_n)f(\xi'_n)} \rangle \leq -1 + \frac{1}{n} \rightarrow -1, \]
which implies that $\overrightarrow{F(x)f(\xi')} = -\overrightarrow{F(x)f(\xi)}$. Therefore, $\xi' = a_x(\xi)$ and both of them are contained in $K_x$ by the argument above. This is a contradiction. We thus find $U$ and $\epsilon > 0$ as described in the statement of Step 1.\\

Step 2: Let $U$ be the open neighbourhood from Step 1. We show that there exists $\delta > 0$ such that for all $x \in U, w \in T_{F(x)}^1 Y$ there exists $\xi \in K_{x}$ such that $\langle w, \overrightarrow{F(x)f(\xi)} \rangle \geq \delta$.

We first introduce the following notation. Given a vector $w \in T_y Y$ and $\alpha > 0$, we define
\[ S_{\alpha}(w) := \{ w' \in T_y Y \vert \angle(w,w') \leq \alpha \}. \]
This is a sector in $T_y Y$, whose middle line is generated by the vector $w$. Note that the angle-width of the sector $S_{\alpha}(w)$ is $2\alpha$.

Suppose the statement of Step 2 is not true. We find sequences $x_n \in U$ and $w_n \in T_{F(x_n)}^1 Y$ such that for all $\xi \in K_{x_n}$, $\langle w_n, \overrightarrow{F(x_n)f(\xi)} \rangle < \frac{1}{n}$. Equivalently, the angle between these two vectors satisfies $\angle(w_n, \overrightarrow{F(x_n)f(\xi)}) > \frac{\pi}{2} - \alpha_n $ with $\alpha_n \rightarrow 0$. Define $\alpha := \pi - \cos^{-1}(-1 + \epsilon) \in (0,\pi)$, where $\epsilon$ is the number found in Step 1. Choose $n$ so that $\alpha_n < \frac{\alpha}{2}$. We conclude that the sector
\[ S_{\frac{\pi}{2} - \frac{\alpha}{2}}(w_n) = \{ w' \in T_{F(x_n)} Y \vert \angle(w', w_n) \leq \frac{\pi}{2} - \frac{\alpha}{2} \} \]
does not contain any elements of the form $\overrightarrow{F(x_n)f(\xi)}$ with $\xi \in K_x$.

By Step 1, we know that for all $\xi \in K_{x_n}$, the sector
\[ S_{\alpha}(-\overrightarrow{F(x_n)f(\xi)}) = \left\{ w' \in T_{F(x_n)} Y \vert \angle(w', -\overrightarrow{F(x_n)f(\xi)} ) \leq \alpha \right\} \]
does not contain any elements of the form $\overrightarrow{F(x_n)f(\xi')}$ with $\xi' \in K_{x_n}$. 

Since $T_{F(x_n)} Y$ is $2$-dimensional, Lemma \ref{lem:ConvexHull} implies that there exists $\xi \in K_{x_n}$ such that
\[ \frac{\pi}{2} > \angle(w_n, \overrightarrow{F(x_n)f(\xi)}) > \frac{\pi}{2} - \frac{\alpha}{2}.Ê\]
We conclude that, for this $\xi$,
\[ \frac{\pi}{2} < \angle(w_n, -\overrightarrow{F(x_n)f(\xi)}) < \frac{\pi}{2} + \frac{\alpha}{2}. \]

This implies that the two sectors $S_{\alpha}(-\overrightarrow{F(x_n)f(\xi)})$ and $S_{\frac{\pi}{2} - \frac{\alpha}{2}}(w_n)$ intersect. Thus, their union is a sector $S_{\beta}(u)$. Since $-\overrightarrow{F(x_n)f(\xi)} \notin S_{\frac{\pi}{2} - \frac{\alpha}{2}}(w_n)$, the angle-width of this union is strictly greater than $2(\frac{\pi}{2} - \frac{\alpha}{2}) + \alpha = \pi$. Since both $S_{\alpha}(-\overrightarrow{F(x_n)f(\xi)})$ and $S_{\frac{\pi}{2} - \frac{\alpha}{2}}$ do not contain any vector of the form $\overrightarrow{F(x_n)f(\xi')}$ with $\xi' \in K_{x_n}$, we conclude that the set $\{ \overrightarrow{F(x_n)f(\xi')} \vert \xi' \in K_{x_n} \}$ is contained in the complement of $S_{\beta}(u)$, which is a sector with angle-width strictly less than $\pi$. Therefore, the convex hull of $\{ \overrightarrow{F(x_n)f(\xi')} \vert \xi' \in K_{x_n} \}$ cannot contain the zero-vector of $T_{F(x_n)} Y$, which is a contradiction to Lemma \ref{lem:ConvexHull}. We conclude that for all $x \in U, w \in T_{F(x)} Y$, we find $\xi \in K_x$, such that $\angle(w, \overrightarrow{F(x)f(\xi)}) \leq \frac{\pi}{2} - \frac{\alpha}{2}$. Applying cosine to this inequality, we find $\delta > 0$, such that for all $x \in U$, $w \in T_{F(x)}^1 Y$, we find $\xi \in K_x$ such that $\langle w, \overrightarrow{F(x)f(\xi)} \rangle \geq \delta$. This proves Step 2.\\

Step 3: We show Lipschitz-continuity on $U$. Let $x,x' \in U$. By Step 2, we find $\xi \in K_x$ such that $-\cos(\angle_{F(x)}(F(x'),f(\xi))) = \langle -\overrightarrow{F(x)F(x')}, \overrightarrow{F(x)f(\xi)} \rangle \geq \delta$. Since $Y$ is non-positively curved, we know that $\angle^{(0)}_{F(x)}(F(x'),f(\xi)) \geq \angle_{F(x)}(F(x'),f(\xi))$ and, therefore,
\[ -\cos(\angle^{(0)}_{F(x)}(F(x'),f(\xi))) \geq -\cos(\angle_{F(x)}(F(x'),f(\xi))) \geq \delta. \]
We have
\begin{equation*}
\begin{split}
\delta d(F(x),F(x')) & \leq -\cos(\angle^{(0)}_{F(x)}(F(x'),f(\xi))) d(F(x),F(x'))\\
& = B(F(x),F(x'),f(\xi))\\
& \leq B(F(x), \pi \circ \Phi(\overrightarrow{x\xi}), f(\xi)) + B(x,x',\xi) + B(\pi \circ \Phi(\overrightarrow{x'\xi}), F(x'),f(\xi))\\
& \leq -M(x) + d(x,x') + M(x')\\
& \leq 2d(x,x').
\end{split}
\end{equation*}

We conclude that for all $x,x' \in U$,
\[ d(F(x),F(x')) \leq \frac{2}{\delta} d(x,x'). \]
Therefore, $F$ is Lipschitz continuous on $U$.\\

Combining the two cases, we conclude that $F$ is locally Lipschitz-continuous on $X \diagdown \partial \{ x \vert M(x) = 0 \}$. Since the set $\{ x \vert M(x) = 0 \}$ is closed, the complement of its topological boundary is dense in $X$. This completes the proof.
\end{proof}

Using Stepanov's theorem, we conclude that $F$ is differentiable almost everywhere. Applying Stepanov's theorem (or Rademacher's) to $M$, we conclude that $M$ is differentiable almost everywhere as well. We conclude that for almost every $x$, both $F$ and $M$ are differentiable at $x$. We are left to prove the sufficient condition for isometry from Theorem \ref{thm:Surfaces}.

\begin{lem} \label{lem:Gradientequation}
Let $x \in X$ such that $F$ and $M$ are differentiable at $x$. For all $v \in T_x X$ and $\xi \in K_x$, we have
\[ \langle \nabla M_x, v \rangle = \langle v, \overrightarrow{x\xi} \rangle - \langle DF_x(v), \overrightarrow{F(x)f(\xi)} \rangle. \]

\end{lem}

\begin{proof}
Let $v \in T_x X$ and let $x'$ vary along the geodesic that starts at $x$ and goes in direction $v$. For $\xi \in K_x$, we compute
\begin{equation*}
\begin{split}
M(x') - M(x)& \geq B(\pi \circ \Phi(\overrightarrow{x'\xi}), F(x'),f(\xi)) - B(\pi \circ \Phi(\overrightarrow{x\xi}), F(x), f(\xi))\\
& = B(x', x, \xi) + B(\pi \circ \Phi(\overrightarrow{x\xi}), F(x),f(\xi)) + B(F(x), F(x'), f(\xi))\\
& \quad - B(\pi \circ \Phi(\overrightarrow{x\xi}), F(x),f(\xi))\\
& = \langle v, \overrightarrow{x\xi} \rangle t - \langle DF_x(v), \overrightarrow{F(x)f(\xi)} \rangle t + O(t^2).
\end{split}
\end{equation*}

We obtain that, for all $v \in T_x X$ and for all $\xi \in K_x$,
\[ \langle \nabla M_x, v \rangle \geq  \langle v, \overrightarrow{x\xi} \rangle - \langle DF_x(v), \overrightarrow{F(x)f(\xi)} \rangle. \]
Replacing $v$ by $-v$ yields the opposite inequality, which implies equality.
\end{proof}

\begin{lem} \label{lem:FivePointCondition}
Let $x \in X$ such that $F$ and $M$ are differentiable at $x$. Suppose $K_x$ contains at least five points. Then $DF_x$ is an isometry between tangent spaces.
\end{lem}

\begin{proof}
By definition of adjoint maps, Lemma \ref{lem:Gradientequation} implies that for every $\xi \in K_x$,
\[ DF_x^*(\overrightarrow{F(x)f(\xi)}) = \overrightarrow{x\xi} - \nabla M_x. \]
Furthermore, since $F$ is invertible by Proposition \ref{prop:InversesDimension2}, the map $DF_x^*$ is invertible. Therefore, the map $DF_x^* + \nabla M_x : T_{F(x)} Y \rightarrow T_x X$ is an invertible affine map that sends a subset of the unit circle $T_{F(x)}^1 Y$ -- namely the set $\{ \overrightarrow{F(x)f(\xi)} \vert \xi \in K_x \}$ -- to a subset of the unit circle $T_x^1 X$.

We are given an invertible, affine map $x \mapsto Ax + b$ between $2$-dimensional vector spaces with an inner product. Since affine maps send ellipses to ellipses and thus circles to ellipses, there are three possibilities what the image of the unit circle under this map may look like.
\begin{enumerate}
\item The image of the unit circle is an ellipse with non-vanishing eccentricity. It can intersect the unit circle in the target space in at most four points.

\item The image of the unit circle is a circle, but not the unit circle of the target space. It can intersect the unit circle in the target space in at most two points.

\item The image of the unit circle is equal to the unit circle in the target space. Then the affine map is of the form $x \mapsto Ax$ and $A$ is norm-preserving. Since an inner product can be expressed purely in terms of its induced norm, $A$ is orthogonal.
\end{enumerate}

Since every point in $K_x$ corresponds to a unit vector in $T_{F(x)} Y$ which is sent to a unit vector by $DF_x^* + \nabla M_x$, we see that, if $K_x$ contains at least five points, the map $DF_x^* + \nabla M_x$ has to be the last of the options above. This implies that $M_x = 0$ and $DF_x^*$ is orthogonal. Thus, $DF_x$ is orthogonal, i.e.\,an isometry of tangent spaces equipped with the Riemannian metric.

\end{proof}

If $K_x$ contains at least five points for almost every $x$, then the Lemma above implies that for almost every $x$, $F$ is differentiable and $DF$ has operator norm at most $1$. It is a standard result that such a map is $1$-Lipschitz. Since $F^{-1}$ equals the circumcenter extension of $f^{-1}$ and, therefore, $K_{F(x)} = f(a_x(K_x))$, we conclude that $F^{-1}$ is $1$-Lipschitz as well. This implies that $F$ is a metric isometry and concludes the proof of Theorem \ref{thm:Surfaces}.

\begin{rem} \label{rem:HigherDimension}
If $X$ and $Y$ are higher-dimensional and we have a situation where we can show that $F$ is differentiable, then Lemma \ref{lem:Gradientequation} implies that $K_x$ is contained in the intersection of an $(n-1)$-dimensional ellipsoid with the $(n-1)$-dimensional unit sphere, or, if $DF_x$ is not invertible, in the intersection of a `full' ellipsoid of dimension at most $n-1$ with the $(n-1)$-dimensional unit sphere. In either case, this tells us that $DF_x$ is an isometry of tangent spaces, whenever $K_x$ is not distributed in a rather specific way. This criterion may be worth further investigation. However, differentiability of $F$ remains an issue in higher dimensions for now.
\end{rem}

%---------------------------------------------------------------------------------------------------------------------------------------
%ROUGH ISOMETRIES FOR CAT(-1) SPACES
%---------------------------------------------------------------------------------------------------------------------------------------

\subsection{Rough isometries for CAT(-1) spaces} \label{subsec:CAT(-1)}

The goal of this section is to prove the following result.

\begin{thm} \label{thm:roughisometryCAT(-1)}
Let $X, Y$ be Hadamard manifolds whose sectional curvature is bounded from below by $-b^2$ and suppose that $X, Y$ are also $\mathrm{CAT(-1)}$ spaces. Let $f : \partial X \rightarrow \partial Y$ be a M\"obius homeomorphism such that $f$ and $f^{-1}$ preserve visible pairs. Then the circumcenter extension of $f$ is a $\left(1, \ln(2)\right)$-quasi-isometry.

Furthermore, if $Y$ is $2$-dimensional, then the circumcenter extension is a $\left(1, \ln\left(\frac{4}{3}\right) \right)$-quasi-isometry.
\end{thm}

\begin{proof}
Since $Y$ is $\mathrm{CAT(-1)}$, it contains no flat strips and $M_x$ consists of exactly one point for every $x \in X$. Therefore, the circumcenter extension is a map from $X$ to $Y$. By Theorem 1.1 in \cite{Biswas17a}, we know that $F$ is coarsely surjective. Let $x, x' \in X$ and $\xi \in \partial X$ such that $f(\xi)$ is represented by the geodesic ray obtained by extending the geodesic from $F(x')$ to $F(x)$. We compute
\begin{equation*}
\begin{split}
d(F(x),F(x')) & = B(F(x),F(x'),f(\xi))\\
& = B(F(x), \pi \circ \Phi(\overrightarrow{x\xi}), f(\xi)) + B(x,x',\xi) + B(\pi \circ \Phi(\overrightarrow{x'\xi}),F(x'),f(\xi))\\
& \leq M(x) + d(x,x') + M(x').
\end{split}
\end{equation*}

Putting $\xi' \in \partial X$ to be represented by the geodesic ray obtained by extending the geodesic from $x'$ to $x$, we obtain
\begin{equation*}
\begin{split}
d(F(x),F(x')) & \geq B(F(x),F(x'),f(\xi'))\\
& = B(F(x), \pi \circ \Phi(\overrightarrow{x\xi}), f(\xi')) + B(x,x',\xi') + B(\pi \circ \Phi(\overrightarrow{x'\xi}),F(x'),f(\xi'))\\
& \geq -M(x) + d(x,x') - M(x').
\end{split}
\end{equation*}

We conclude that, if $M$ is bounded on $X$, then $F$ is a $(1, 2\Vert M \Vert_{\infty})$-quasi-isometry. We are left to prove that $M$ is bounded.

Let $x \in X$ and $\xi, \xi' \in K_{x,F(x)}$. Since $\frac{ \partial f_* \rho_x}{\partial \rho_{F(x)}}(\xi) = e^{u_{x,F(x)}(\xi)}$ and $M(x) = u_{x,F(x)}(\xi) = u_{x,F(x)}(\xi')$, we have
\[ \rho_x(\xi, \xi')^2 = e^{2M(x)} \rho_{F(x)}(f(\xi), f(\xi'))^2 \]
and therefore,
\[ (f(\xi) \vert f(\xi'))_{F(x)} - M(x) = (\xi \vert \xi')_x. \]
Since Gromov products are non-negative, this implies that $M(x) \leq (f(\xi) \vert f(\xi'))_{F(x)}$ for all $\xi, \xi' \in K_{x,F(x)}$.

Let $Y$ be of dimension at least three and let $\xi \in K_{x,F(x)}$. By Lemma \ref{lem:ConvexHull}, we know that there exists $\xi' \in K_{x,F(x)}$ such that $\langle \overrightarrow{F(x)f(\xi)}, \overrightarrow{F(x)f(\xi)} \rangle \leq 0$, i.e.\, the angle between $f(\xi)$ and $f(\xi')$ at $F(x)$ is at least $\frac{\pi}{2}$. Since $Y$ is $\mathrm{CAT(-1)}$, we have that
\[ (f(\xi) \vert f(\xi'))_{F(x)} \leq (\eta \vert \eta')_z, \]
where $\eta, \eta' \in \partial \mathbb{H}^2$ such that their representing geodesic rays starting at $z \in \mathbb{H}^2$ depart at an angle of $\frac{\pi}{2}$. We are left to compute the Gromov product of two specific geodesics in $\mathbb{H}^2$.

If $Y$ is $2$-dimensional, since $K_{x,F(x)}$ contains at least three points by Corollary \ref{cor:AtleastthreeExtremalDirections}, we conclude that there are $\xi, \xi' \in K_{x,F(x)}$ such that $\angle_{F(x)}\left(\overrightarrow{F(x)f(\xi)}, \overrightarrow{F(x)f(\xi')}\right) \geq \frac{2\pi}{3}$. We are left to compute the Gromov product of two geodesic rays in $\mathbb{H}^2$ that start at the same point and depart at an angle of $\frac{2\pi}{3}$.\\

The Theorem now follows from the following formula, which is a standard computation.

\begin{lem} \label{lem:GromovProductComputations}
Let $\gamma, \tilde{\gamma}$ be geodesic rays in $\mathbb{H}^2$ that start at the same point $o$ and depart at an angle $\alpha$. Then
\[ ( \gamma \vert \tilde{\gamma})_o = -\ln \left( \sin \left( \frac{\alpha}{2} \right) \right). \]
\end{lem}

Let $\alpha = \frac{\pi}{2}$. Since $\sin \left( \frac{\pi}{4} \right) = \frac{1}{\sqrt{2}}$, we obtain
\[ (\gamma_{\frac{\pi}{8}} \vert \tilde{\gamma}_{\frac{\pi}{8}} )_i = \ln \left( \sqrt{2} \right). \]

In the $2$-dimensional case, we put $\alpha = \frac{2 \pi}{3}$. Since $\sin \left( \frac{\pi}{3} \right) = \frac{\sqrt{3}}{2}$, we obtain
\[ (\gamma_{\frac{\pi}{8}} \vert \tilde{\gamma}_{\frac{\pi}{8}} )_i = \ln \left( \frac{2}{\sqrt{3}} \right) \approx 0.143841. \]

Since $F$ is a $(1, 2 \Vert M \Vert_{\infty} )$-quasi-isometry, this proves the Theorem.
\end{proof}

%---------------------------------------------------------------------------------------------------------------------------------------
%ADDING A COCOMPACT ACTION
%---------------------------------------------------------------------------------------------------------------------------------------

\subsection{Adding a cocompact action} \label{subsec:cocompactaction}

In this section, we prove the following result.

\begin{thm} \label{thm:roughisometry}
Let $X, Y$ be Hadamard manifolds whose sectional curvature is bounded from below by $-b^2$, such that $\partial X$ and $\partial Y$ satisfy (4v) and all points in $\partial X$ and $\partial Y$ are in a rank 1 hinge. Suppose, there is a group $G$ which acts cocompactly by isometries on $X$ and $Y$. Let $f : \partial X \rightarrow \partial Y$ be a $G$-equivariant M\"obius homeomorphism such that $f$ and $f^{-1}$ preserve visible pairs. Then, the function $M : X \rightarrow \mathbb{R}$ is bounded and there is a $G$-equivariant $(1, 2\Vert M \Vert_{\infty})$-quasi-isometry $F : X \rightarrow Y$.
\end{thm}

\begin{proof}
Since $f$ is $G$-equivariant and $G$ acts by isometries, we have for all $g \in G, x \in X, y \in Y, \xi \in \partial X$.
\[ \frac{\partial f_*\rho_x}{\partial \rho_y}(f(\xi)) = \frac{\partial f_*\rho_{gx}}{\partial \rho_{gy}}(gf(\xi)). \]

This implies that, $u_{gx,gy}(g\xi) = u_{x,y}(\xi)$, $\Vert u_{gx,gy} \Vert_{\infty} = \Vert u_{x,y} \Vert_{\infty}$ and therefore, $M_{gx} = g(M_x)$. In particular, thinking of the circumcenter extension as a map $F : X \rightarrow \faktor{Y}{\sim}$, the action of $G$ sends equivalence classes of $\sim$ to equivalence classes and $F$ is $G$-equivariant. If $M_x$ consists of more than one point, we can choose $y \in M_x$ and use its $G$-orbit to define $\tilde{F}(gx) := gy$ in order to get a map $\tilde{F} : X \rightarrow Y$. By abuse of notation, we call this map $F$ as well. Note that this map may not be continuous anymore. Since $F$ is $G$-equivariant and $G$ acts cocompactly on $Y$, $F$ is coarsely surjective.

The same argument as in the proof of Theorem \ref{thm:roughisometryCAT(-1)} shows that $F$ is 
%Let $x, x' \in X$. We find $\xi$ such that $f(\xi)$ is the endpoint of the geodesic ray obtained by extending the geodesic from $F(x')$ to $F(x)$. We compute
%\begin{equation*}
%\begin{split}
%d(F(x),F(x')) & = B(F(x),F(x'),f(\xi))\\
%& = B(F(x),\pi \circ \Phi(\overrightarrow{x\xi}), f(\xi)) + B(x,x',\xi) + B(\pi \circ \Phi(\overrightarrow{x'\xi}), F(x'), f(\xi))\\
%& \leq M(x) + d(x,x') + M(x').
%\end{split}
%\end{equation*}
%We conclude that $F$ is 
a $(1, 2\Vert M \Vert_{\infty})$-quasi-isometry, if $M$ is bounded. Since
\[ M(gx) = \Vert u_{gx,F(gx)}(\cdot) \Vert_{\infty} = \Vert u_{x,F(x)}(g^{-1}\cdot) \Vert_{\infty} = M(x), \]
we obtain that it is sufficient to bound $M$ on a compact fundamental domain of the $G$-action on $X$. Since $M$ is Lipschitz continuous, $M$ is bounded on any compact set. We conclude that $M$ is bounded and $F$ is a $G$-equivariant $(1, 2\Vert M \Vert_{\infty})$-quasi-isometry.
\end{proof}

\bibliographystyle{alpha}
\bibliography{mybib}

\end{document}